\documentclass[11pt]{article}

\usepackage[utf8]{inputenc}
\usepackage[T1]{fontenc}
\usepackage[main=english,ngerman]{babel}
\usepackage{lmodern}
\usepackage{tcolorbox}
\usepackage{enumerate}
\usepackage{multirow}
\usepackage{nicefrac}
\usepackage{caption}
\usepackage{graphicx}
\usepackage{tikz}
\usetikzlibrary{patterns}
\usetikzlibrary{arrows}
\usetikzlibrary{decorations}
\usetikzlibrary{decorations.pathreplacing}
\usepackage{tikz-3dplot}
\usepackage{subfigure}
\usepackage{pgfplots}
\usepackage{hyperref}
\usepackage{picture}
\usepackage{float}

\usepackage[a4paper, left=2.5cm, right=2.5cm, top=3cm]{geometry}
\usepackage{media9}

\usepackage{amsmath,amssymb,amsthm}
\usepackage{mathtools}
\usepackage[mathscr]{euscript}
\usepackage{xpatch}
\usepackage{asymptote}
\usepackage[affil-it]{authblk}

\newtheorem{definition}{Definition}
\theoremstyle{remark}
\newtheorem{remark}{Remark}
\theoremstyle{theorem}
\newtheorem{theorem}{Theorem}
\theoremstyle{theorem}
\newtheorem{lemma}{Lemma}

\newtheorem{corollary}{Corollary}
\newtheorem{assumption}{Assumption}

\usepackage[rightcaption]{sidecap}
\renewcommand{\d}{\textup{div}}

\renewcommand{\o}{\Omega}
\newcommand{\s}{\mathbb{S}}
\newcommand{\ts}{\tilde{S}}
\renewcommand{\k}{\mathbb{K}}

\renewcommand{\j}{\tilde{{J}}}
\newcommand{\h}{\hat{\partial}}
\renewcommand{\S}{\boldsymbol{{S}}}
\newcommand{\J}{\tilde{\boldsymbol{J}}}
\newcommand{\tS}{\tilde{\boldsymbol{S}}}

\newcommand{\p}{\textbf}
\newcommand{\f}{\boldsymbol}
\renewcommand\t[1]{\boldsymbol{\mathcal{#1}}}

\newcommand\norm[1]{\left\lVert#1\right\rVert}
\newcommand\n[1]{\left\lVert#1\right\rVert}
\graphicspath{{images/}}
\begin{document}
	
	\title{Mixed Isogeometric Discretizations for Planar Linear Elasticity}
	
	\author{Jeremias Arf and Bernd Simeon \thanks{E-mail: \texttt{\{arf, simeon\}@mathematik.uni-kl.de, \noindent \newline \noindent \ \ \emph{2020  Mathematics Subject Classification. \ 65N30; 65N12} , }}}	   
\affil{TU Kaiserslautern, Germany}

\date{}
\maketitle
	
	\begin{abstract}
		In this article we suggest two  discretization methods based on isogeometric analysis (IGA) for  planar linear elasticity.  On the one hand, we apply the well-known  ansatz of weakly imposed symmetry for  the stress tensor and obtain a well-posed mixed formulation.  Such modified mixed problems have been already studied by different authors. But we concentrate on the exploitation of IGA  results  to handle also  curved boundary geometries.		
		On the other hand, we consider the more complicated situation of strong symmetry, i.e. we discretize the mixed weak form determined by the so-called Hellinger-Reissner variational principle. We show the existence of suitable approximate fields  leading to an inf-sup stable saddle-point problem. 
		For both discretization approaches we prove convergence statements  and in case of weak symmetry we illustrate the approximation behavior by means of several numerical experiments.

	\end{abstract}
	\section{Introduction}

In this article we identify a bounded Lipschitz domain $\o \subset \mathbb{R}^{2}$  with some elastic body that is loaded according to some body force $\f{f}$. The  deformation of the body can be expressed through the displacement field $\f{u} \colon \o \rightarrow \mathbb{R}^2$ that measures the difference between current configuration $\o + \f{u}$ and initial configuration $\o$. Governing equations  guaranteeing a physically reasonable description of the deformation process are  given by the theory of linear elasticity and read:
\begin{align}
	\label{linear_elas_1}
	\f{A} \f{\sigma} &= \f{\varepsilon}(\f{u}) \hspace{0.2cm}  \textup{and} \\
	\nabla \cdot \f{\sigma} &= \f{f} \hspace{0.1cm} \textup{in} \ \ \o,
	\label{linear_elas_2}
\end{align}
where $\f{\varepsilon}(\f{u})= \frac{1}{2} (\nabla\f{u}+\nabla \f{u}^T) $ denotes the symmetric gradient that  corresponds to the mechanical strain within  the infinitesimal strain theory. Further, $ \f{A}= \f{A}(\f{x}) \colon \mathbb{S} \rightarrow\mathbb{S}, \ \f{x} \in \o $ is the  compliance tensor which comprises the material  properties and $\f{\sigma} \colon  \o \rightarrow \mathbb{S}$ is the symmetric stress tensor, where $\mathbb{S}$ stands for the space of symmetric matrices in $\mathbb{R}^{2 \times 2}$. We note that the linear relation between strain and stress as given in equation \eqref{linear_elas_1} is the key assumption in the linearized framework.  A classical conservation of momentum argument relates the  stress tensor and the force $\f{f}$ and as a result one obtains equation \eqref{linear_elas_2}. An important special case are homogeneous and isotropic materials. Then the compliance tensor simplifies to $$\f{A}\f{\sigma} = \frac{1}{2 \mu} \Big( \f{\sigma}- \frac{\lambda}{2\lambda+2\mu} \textup{tr}(\f{\sigma}) \f{I}\Big),$$ and the material properties are condensed in form of  so-called Lamé parameters $\lambda \geq0, \ \mu > 0$. In all sections below we assume $\f{A}$ to be symmetric, positive definite and bounded w.r.t. the $L^2$-norm.  Actually, we have to specify some boundary conditions to complete the model, where we first consider the easy case of pure displacement conditions, i.e. we have $\f{u}= \f{u}_D $ on $\Gamma \coloneqq  \partial \o$ for some given  $\f{u}_D$. Supposing the invertibility of $\f{A}$ one can obtain a  formulation which  depends only on the displacement variable. A  combination of \eqref{linear_elas_1} and \eqref{linear_elas_2} yields \begin{equation}
	\label{primal_formulation_Dirichlet}
	\nabla \cdot (\f{C} \f{\varepsilon}(\f{u})) = \f{f} ,  \ \ \ \f{C} \coloneqq \f{{A}}^{-1} , 
\end{equation} with $\f{C}$ as the  elasticity tensor.  
The standard weak form  of the mentioned equation in case of homogeneous boundary conditions   is the primal formulation: \begin{equation}
	\label{eq_primal_zero_BC}
	\textup{Find} \ \ \f{u}\in \f{H}_0^1(\o)  \ \ \textup{s.t.} \ \  \langle \f{C} \f{\varepsilon}(\f{u}), \f{\varepsilon}(\f{v}) \rangle = \langle \f{f}, \f{v}\rangle,  \ \ \forall \f{v} \in \f{H}_0^1(\o),
\end{equation}  with $\f{H}_0^1(\o)$ denoting the  Sobolev space  with square-integrable weak derivatives, square- \\ integrable component functions and zero boundary values. Further $\langle \cdot, \cdot \rangle $ denotes the $L^2$ inner product. Unfortunately,  the primal formulation  becomes unstable for incompressible problems, i.e. if $\lambda \rightarrow \infty$. Another point is the non-locality of stress-strain relations in more general material situations. Consequently, the primal weak formulation is not optimal also in view of generalizability. This is why we  find in literature   mixed formulations  that seek for stress and displacement simultaneously.  A standard mixed formulation reads: 

Find $\f{\sigma} \in \boldsymbol{\mathcal{{H}}}(\Omega,\textup{div},\s), $  $\f{u} \in \boldsymbol{L}^2(\Omega)$ s.t. 
\begin{alignat}{4}
	\label{weak_form_strong_symmetry_contiuous_1}
	&\langle \p{A} \boldsymbol{\sigma} , \boldsymbol{\tau} \rangle \ + \ &&\langle   \f{u} , \nabla \cdot \boldsymbol{\tau} \rangle \  \color{black}  &&= \langle \boldsymbol{\tau} \cdot \p{n};  \f{u}_D \rangle_{\Gamma}, \ \hspace{0.4cm} &&\forall \boldsymbol{\tau} \in \boldsymbol{\mathcal{{H}}}(\Omega,\textup{div},\s), \\
	&\langle \nabla \cdot \boldsymbol{\sigma} , \f{v} \rangle &&  &&= \langle \f{f},\f{v}\rangle, \ \hspace{0.4cm} && \forall \f{v} \in \boldsymbol{L}^2(\Omega). \nonumber
\end{alignat}
Above $\langle \boldsymbol{\tau} \cdot \p{n};  \f{u}_D \rangle_{\Gamma}$ denotes the duality pairing determined by the normal trace which corresponds to the  $L^2(\Gamma)$ inner product on the boundary if $\t{\tau}$ is regular enough. Furthermore, we wrote  $\boldsymbol{\mathcal{{H}}}(\Omega,\textup{div},\s) $ for the space of square-integrable symmetric matrix fields which have a square-integrable divergence and $\f{L}^2(\o)$ denotes the space of square-integrable vector fields.
 As explained in \cite{Arnold2007MixedFE},   background of this saddle-point problem  is the \emph{Hellinger-Reissner} functional where  $(\f{u},\f{\sigma})$ is given as the corresponding unique critical point.  Regarding the discretization of the mixed weak system  there are two main ways which are frequently used. On the one hand, one can try to define  discrete   spaces $\t{V}_h \subset \t{H}(\d,\o,\mathbb{S}), \ \f{V}_h \subset \f{L}^2(\o)$ to approximate the stress tensor and the displacement. However, the choice of suitable subspaces is not trivial. To obtain a well-posed numerical method the discretized saddle-point problem has to satisfy \emph{Brezzi}'s stability conditions (see \cite{Brezzi}). It becomes apparent that the symmetry condition and divergence compatibility for $\t{V}_h$ is hard to achieve simultaneously, especially  in case of more complex computational domains $\o$. Methods that follow the strong symmetry ansatz can be found e.g. in \cite{Arnold_rectangle_1} and \cite{Arnold2002MixedFE_1}. However,   the mentioned discretizations are quite complicated. For example  in \cite{Arnold2002MixedFE_1} one needs at least $24$ degrees of freedom per triangle element.
 Therefore, on the other hand authors in several articles (e.g. \cite{Arnold2007MixedFE,Rettung,Arnold2015,Falk}) study a modified mixed problem which guarantees symmetry only in a variational sense. Such formulations with weakly imposed symmetry  are related to the de Rham complex and the elegant theory of Finite Element Exterior Calculus, see \cite{Arnold2007MixedFE} and \cite{Falk}, but require an  additional Lagrange multiplier. 
 
 In this article we shall consider both ways of discretization, namely strong and weak symmetry conservation. Our aim is to show that there are discrete spaces satisfying the Brezzi conditions and thus lead to well-posedness. In contrast to classical Finite Elements (FEM) based on simplicial decompositions or quadrilateral elements we use B-splines and results from Isogeometric Analysis (IGA) to define test and trial fields. Advantages of IGA are the possibility to consider  geometries with curved boundaries  and the fact that the global regularity of the B-splines can be increased efficiently.   Such special features are useful if the method should be applicable in more complex scenarios. Although our ansatz for strongly imposed symmetry  is only valid in the so-called single-patch case, it is  capable to handle domains with curved boundaries. Nevertheless, for the weak symmetry approach the generalization to multi-patch situations is possible.

The paper is structured as follows. In Section \ref{sec_pre} we introduce briefly some mathematical notation and basic notions in the context of IGA. 
Next, in Section \ref{section_weak_form_weak_symmetry} we explain the mixed formulation of  linear elasticity with weakly imposed symmetry and  consider spline spaces that are suitable for discretization in planar domains. In particular, we establish a corresponding error estimate. Afterwards, in Section \ref{section_strong} we face the $2D$ case with strong symmetry, where we define likewise proper discrete spaces.  Section  \ref{sec_numerics} is dedicated to  numerical examples to validate  the theoretical statements, where we restrict ourselves to the  weak symmetry case. A  detailed numerical study of the strong symmetry ansatz  should be the issue of a further article.  We close  with a short conclusion in Section \ref{sec_conc}.

	\section{Mathematical preliminaries and notation}
	\label{sec_pre}
	\subsection{Mathematical notation}
	In the section we introduce some notation and define several spaces. \\
	For some bounded Lipschitz domain $D \subset \mathbb{R}^d$ we write for the standard Sobolev spaces $H^0(D)=L^2(D)$ and $\ H^k(D), \ k \in \mathbb{N}$, where $L^2(D)$ stands for the Hilbert space of square-integrable functions endowed with the inner product $\langle \cdot , \cdot \rangle$. The norms $\norm{\cdot}_{H^k(D)}$ denote the classical Sobolev norms in $H^k(D)$. In case of vector or matrix fields we can define Sobolev spaces by requiring the component functions to be in suitable Sobolev spaces. To distinguish latter case from the scalar-valued one, we use a bold-type notation. For example we have for $\boldsymbol{v} \coloneqq (v_1, v_2)^T, \ \f{v} \in \boldsymbol{H}^k(D) \colon \Leftrightarrow \ v_i \in H^k(\Omega), \, \forall i  $ and $\f{M} \coloneqq \big(M_{ij}\big), \ \f{M} \in \t{H}^k(D) \colon \Leftrightarrow \ M_{ij} \in H^k(\Omega), \, \forall i,j$ and we define the norms
	\begin{align*}
		\norm{\boldsymbol{v}}_{\boldsymbol{H}^k(D)}^2 \coloneqq \sum_i \norm{v_i}_{H^k(D)}^2, \ \ \ \  \ \ \ \norm{\f{M}}_{\t{H}^k(D)}^2 \coloneqq \sum_{i,j} \norm{M_{ij}}_{{H}^k(D)}^2.
	\end{align*}
	We note that the inner product $\langle \cdot , \cdot \rangle$ introduces straightforwardly an inner product on $\boldsymbol{L}^2(D)= \f{H}^0(D)$ and $ \ \t{L}^2(D)=\t{H}^0(\o)$ which will we denote with $\langle \cdot , \cdot \rangle$, too. Now let us consider vector-valued mappings and set
	\small
	\begin{alignat*}{3}
		\boldsymbol{H}(D,\textup{curl}) &\coloneqq \{ \f{v} \in \boldsymbol{L}^2(D) \ | \ \nabla \times \f{v} \in \boldsymbol{L}^2(D)  \}, \  \hspace{0.3cm} && \norm{ \f{v}}_{\boldsymbol{H}(D,\textup{curl})}^2  \coloneqq \norm{\f{v}}_{\boldsymbol{L}^2(D)}^2 + \norm{\nabla \times \f{v}}_{\boldsymbol{L}^2(D)}^2,\\
		\boldsymbol{H}(D,\textup{div}) &\coloneqq \{ \f{v} \in \boldsymbol{L}^2(D) \ | \ \nabla \cdot \f{v} \in {L}^2(D)  \}, \   && \norm{ \f{v}}_{\boldsymbol{H}(D,\textup{div})}^2  \coloneqq \norm{\f{v}}_{\boldsymbol{L}^2(D)}^2 + \norm{\nabla \cdot \f{v}}_{{L}^2(D)}^2.
	\end{alignat*}
	\normalsize
	Above we wrote $\nabla$  for the classical nabla operator. 
	The definition for $ \boldsymbol{H}(D,\textup{div}) $ and the corresponding norm can be generalized to the matrix setting, namely $ 
	 \ \t{H}(D,\textup{div})$, by a row-wise definition. In particular,  the divergence $\nabla \cdot$ acts  row-wise, too. 	On  $\boldsymbol{H}(D,\textup{div})$ and thus on  $\t{H}(D,\textup{div})$  one can introduce a  \emph{normal trace operator} $\gamma_n$; see Lemma \ref{lemma_trace_theorem}. If $\f{\sigma} \in \t{H}(D,\textup{div})$,  we write $\langle \f{\sigma} \cdot \f{n} ;  \cdot \rangle_{\Gamma} $ for the underlying duality pairing , where $\f{n}$ stands for the outer unit normal to the domain boundary. We denote the space of square-integrable symmetric and skew-symmetric matrix fields with $\t{L}^2(D,\mathbb{S})$ and  $\t{L}^2(D,\mathbb{K})$ and later the next operators become useful.
	\begin{align*}
		\textup{Skew}(q) &\coloneqq \begin{pmatrix}
			0 & -q \\ q & 0
		\end{pmatrix} \in \t{L}^2(D,\mathbb{K}),\ \ \ q \in L^2(D),  \\
		\textup{SYM}(s_1, s_2,s_3 ) &\colon D  \rightarrow \mathbb{R}^{2 \times 2} , \ \ \f{x} \mapsto  \begin{pmatrix}
			s_1(\f{x}) & s_2(\f{x})  \\
			s_2(\f{x}) & s_3(\f{x})  
		\end{pmatrix},
	\end{align*}
	if $s_i \colon D \rightarrow \mathbb{R}$. For   vector spaces we  likewise use the formal notation $\textup{SYM}()$, e.g. \\ 
	$\textup{SYM}(L^2(D), \dots ,L^2(D) ) = \t{L}^2(D,\s)$. 
	Besides we define
	\begin{alignat*}{3}
		\t{H}(D,\d,\mathbb{S}) &\coloneqq \{ \f{S} \in \t{L}^2(D,\mathbb{S}) \ | \ \nabla \cdot  \f{S} \in \f{L}^2(D) \}.
	\end{alignat*}
	Further, for a matrix $\f{M}=(M_{ij})$ we use an upper index $\f{M}^j$ to denote the $j$-th column and a lower index $\f{M}_i$ for the $i$-th row. With $C^k(D)$ we denote the space of continuous functions which have continuous derivatives up to order $k$.  If we have vector or matrix fields we utilize again the notation $C^k(D)$ and mean that all components are $C^k$-regular.
	Finally, we introduce the $2D$ curl operator through $\textup{curl}(v) \coloneqq (\partial_2v,-\partial_1 v)^T, \ v \in H^1(D)$.
	
	After stating some basic notation we proceed  with the consideration of  B-splines.

	\subsection{B-splines}
	\label{section:splines}
	Here, we state a short overview of B-spline functions, spaces respective, and some related basic results. \\
	Following \cite{IGA1,IGA3} for a brief
	exposition, we call an  increasing sequence of real numbers $\Xi \coloneqq \{ \xi_1 \leq  \xi_2  \leq \dots \leq \xi_{n+p+1}  \}$ for some $p \in \mathbb{N}$   \emph{knot vector}, where we assume  $0=\xi_1=\xi_2=\dots=\xi_{p+1}, \ \xi_{n+1}=\xi_{n+2}=\dots=\xi_{n+p+1}=1$, and call such knot vectors $p$-open. 
	Further, the multiplicity of the $j$-th knot is denoted by $m(\xi_j)$ and we require $m(\xi_j) \leq p+1$.
	Then  the univariate B-spline functions $\widehat{B}_{j,p}(\cdot)$ of degree $p$ corresponding to a given knot vector $\Xi$ are defined recursively by the \emph{Cox-DeBoor formula}:
	\begin{align}
		\widehat{B}_{j,0}(\zeta) \coloneqq \begin{cases}
			1, \ \ \textup{if}  \ \zeta \in [\xi_{j},\xi_{j+1}) \\
			0, \ \ \textup{else},
		\end{cases}
	\end{align}
	\textup{and if }  $p \in \mathbb{N}_{\geq 1} \ \textup{we set}$ 
	\begin{align}
		\widehat{B}_{j,p}(\zeta)\coloneqq \frac{\zeta-\xi_{j}}{\xi_{j+p}-\xi_j} \widehat{B}_{j,p-1}(\zeta)  +\frac{\xi_{j+p+1}-\zeta}{\xi_{j+p+1}-\xi_{j+1}} \widehat{B}_{j+1,p-1}(\zeta),
	\end{align}
	where one puts $0/0=0$ to obtain  well-definedness. The knot vector $\Xi$ without knot repetitions is denoted by $\{ \psi_1, \dots , \psi_m \}$. \\
	The multivariate extension of the last spline definition is achieved by a tensor product construction. In other words, we set for a given tensor knot vector   $\boldsymbol{\Xi} \coloneqq \Xi_1 \times   \dots \times \Xi_d $, where the $\Xi_{l}=\{ \xi_1^{l}, \dots , \xi_{n_l+p_l+1}^{l} \}, \ l=1, \dots , d$ are $p_l$-open,   and a given \emph{degree vector}   $\f{p} \coloneqq (p_1, \dots , p_d)$ for the multivariate case
	\begin{align}
		\widehat{B}_{\f{i},\f{p}}(\boldsymbol{\zeta}) \coloneqq \prod_{l=1}^{d} \widehat{B}_{i_l,p_l}(\zeta_l), \ \ \ \ \forall \, \f{i} \in \mathit{\mathbf{I}}, \ \  \boldsymbol{\zeta} \coloneqq (\zeta_1, \dots , \zeta_d),
	\end{align}
	with  $d$ as  the underlying dimension of the parametric domain $\widehat{\Omega} \coloneqq  (0,1)^d$ and $\textup{\p{I}}$ the multi-index set $\textup{\p{I}} \coloneqq \{ (i_1,\dots,i_d) \  | \  1\leq i_l \leq n_l, \ l=1,\dots,d  \}$.\\
	B-splines  fulfill several properties and for our purposes the most important ones are:
	\begin{itemize}
		\item If  for all internal knots the multiplicity satisfies $1 \leq m(\xi_j) \leq m \leq p , $ then the B-spline basis functions $\widehat{B}_{i,p}({\xi})$ are globally $C^{p-m}$-continuous. Therefore we define  in this case the regularity integer $r \coloneqq p-m$. Obviously, by the product structure, we get splines $\widehat{B}_{\f{i},\f{p}}(\boldsymbol{\zeta})$ which are $C^{r_l}$-smooth  w.r.t. the $l$-th coordinate direction if the internal multiplicities fulfill $1 \leq m(\xi_j^l) \leq m_l  \leq p_l, \ r_l \coloneqq p_l-r_l, \ \forall l \in 1, \dots , d$ in the multivariate case.  In case of $r_i = -1$ we have discontinuous splines w.r.t. the $i$-th coordinate direction. 
		\item The B-splines $ \{\widehat{B}_{\f{i},\f{p}} \ | \ \ \f{i} \in \p{I} \}$ are linearly independent.
		\item For continuous univariate splines $\widehat{B}_{i,p}, \ p \geq 1$ we have
		\begin{align}
			\label{eq:soline_der}
			\partial_{\zeta} \widehat{B}_{i,p}(\zeta) = \frac{p}{\xi_{i+p}-\xi_i}\widehat{B}_{i,p-1}(\zeta) +  \frac{p}{\xi_{i+p+1}-\xi_i}\widehat{B}_{i+1,p-1}(\zeta),
		\end{align} 
		with $\widehat{B}_{1,p-1}(\zeta)\coloneqq \widehat{B}_{n+1,p-1}(\zeta) \coloneqq 0$.
		\item  The support of the spline $\widehat{B}_{i,p} $ is the interval $[\xi_i,\xi_{i+p+1}]$. Moreover,  the knots $\psi_j$ define a subdivision of the interval $(0,1)$ and for each element $I = (\psi_j,\psi_{j+1})$ we find a $i$ with $(\psi_j,\psi_{j+1})= (\xi_i,\xi_{i+1})$ and write $\tilde{I} \coloneqq  (\xi_{i-p},\xi_{i+p+1})$ for the so-called \emph{support extension}.
	\end{itemize}
	The space spanned by all univariate splines $\widehat{B}_{i,p}$ corresponding to  given knot vector, degree $p$ and global regularity $r$  is denoted by $$S_p^r \coloneqq \textup{span}\{ \widehat{B}_{i,p} \ | \ i = 1,\dots , n \}.$$
	For the multivariate case we just define the spline space as the product space $$S_{p_1, \dots , p_d}^{r_1,\dots,r_d} \coloneqq S_{p_1}^{r_1} \otimes \dots \otimes S_{p_d}^{r_d} = \textup{span} \{\widehat{B}_{\f{i},\f{p}} \ | \  \f{i} \in \mathit{\mathbf{I}}  \}$$ of proper univariate spline spaces.\\
	To define discrete spaces based on splines we require a  mapping $\f{F} \colon \widehat{\Omega} \rightarrow \mathbb{R}^d$ which parametrizes the computational domain $\o \coloneqq \f{F}(\widehat{\o})$.
	The knots stored in the knot vector $  \boldsymbol{\Xi} $, corresponding to  the underlying  splines, determine a mesh in the parametric domain $\widehat{\Omega} $, namely  $\widehat{M} \coloneqq \{ K_{\f{j}}\coloneqq (\psi_{j_1}^1,\psi_{j_1+1}^1 ) \times \dots \times (\psi_{j_{d}}^{d},\psi_{j_{d}+1}^{d} ) \ | \  \f{j}=(j_1,\dots,j_{d}), \ \textup{with} \ 1 \leq j_i <N_i\},$ and
	with ${\boldsymbol{\Psi}}= \{\psi_1^1, \dots, \psi _{N_1}^1\}  \times \dots \times \{\psi_1^{d}, \dots, \psi _{N_{d}}^{d}\}$  \  \textup{as  the knot vector} \ ${\boldsymbol{\Xi}}$ \ 
	\textup{without knot repetitions}.   
	The image of this mesh under the mapping $\f{F}$, i.e. $\mathcal{M} \coloneqq \{{\f{F}}(K) \ | \ K \in \widehat{M} \}$, gives us a mesh structure in the physical domain. By inserting knots without changing the parametrization  we can refine the mesh, which is the concept of $h$-refinement; see \cite{IGA2,IGA1}.
	For a mesh $\mathcal{M}$ we define the global mesh size $h \coloneqq \max\{h_{\mathcal{K}} \ | \ \mathcal{K} \in \mathcal{M} \}$, where for $\mathcal{K} \in \mathcal{M}$ we denote with $h_{\mathcal{K}} \coloneqq \textup{diam}(\mathcal{K}) $ the \emph{element size}.
	\begin{assumption}{(Regular mesh)}\\
		\label{assum_reg_mesh}
		There exists a constant $c_u $ independent from the mesh size such that $h_{\mathcal{K}} \leq h \leq c_u \, h_{\mathcal{K}}$ for all mesh elements $\mathcal{K} \in \mathcal{M}$. 
		Let the parametrization mapping $\f{F}$ be a diffeomorphism and let the restrictions of $\f{F}$ to the closure of each mesh element be smooth.
	\end{assumption}

	\subsection{Isogeometric de Rham spaces}
	\label{subsec:iag_spaces}
	Here and in the following we have $\o = \f{F}(\widehat{\o})$,  where $\f{F}$ meets the conditions in Assumption \ref{assum_reg_mesh} and $\f{J} \coloneqq D\f{F}$ stands for the Jacobian of $\f{F}$.
	
	Basis for the design of stable FEM discretizations of the mixed linear elasticity formulation with weakly imposed symmetry  is the connection between the de Rham complexes 
	
	\begin{tikzpicture}
		\hspace*{-0.4cm}
		\node at (-1.5,1.4) {$\mathbb{R}$}; 
		
		\node at (0,1.4) {$H^1(\Omega)$};
		
		\node at (3,1.4) {$\f{H}(\Omega,\textup{curl})$};

		\node at (6.5,1.4) {$\boldsymbol{H}(\Omega,\textup{div})$};

		\node at (9.5,1.4) {$L^2(\Omega)$};

		\node at (1.3,1.7) {$\nabla$};
		
		\node at (-0.99,1.6) {$\iota$};
		\node at (4.7,1.7) {$\nabla \times $};

		\node at (8.2,1.7) {$\nabla \cdot $};

		\draw[->] (-1.2,1.4) to (-0.7,1.4);
		\draw[->] (0.8,1.4) to (1.8,1.4);
		
		\draw[->] (4.2,1.4) to (5.3,1.4);
		
		\draw[->] (7.7,1.4) to (8.75,1.4);
		\draw[->] (10.2,1.4) to (10.9,1.4);
		\node at (10.55,1.7) {$0$};
		
		\node at (12.15,1.4) {$\{0\}  ,$ in $3D$,};

		\hspace*{1.5cm}

					\node at (0,1.4-1) {$H^1(\Omega)$};			
		\node at (3,1.4-1) {$\boldsymbol{H}(\Omega,\textup{div})$};
		\node at (6.05,1.4-1) {$L^2(\Omega)$};			
		\node at (8.8,1.4-1) {$\{0 \}, \ $  in $2D$,};						
		\node at (1.3,1.7-1) {$\textup{curl}$};			
		\node at (4.7,1.7-1) {$\nabla \cdot $};			
		\node at (7.1,1.7-1) {$0 $};			
		\draw[->] (0.8,1.4-1) to (1.8,1.4-1);			
		\draw[->] (4.2,1.4-1) to (5.3,1.4-1);			
		\draw[->] (6.8,1.4-1) to (7.5,1.4-1);			
		\draw[->] (-1.6,1.4-1) to (-0.72,1.4-1);			
		\node at (-1.09,1.65-1) {$\iota$};			
		\node at (-2.25,1.4-1) {$\mathbb{R}$};
	\end{tikzpicture}\\	
	and the so-called elasticity complex with weak symmetry. For more information we refer to  \cite{Arnold2007MixedFE} and \cite{Falk}.
	Therefore, a structure-preserving discretization of the de Rham chain utilizing proper spline spaces might be a reasonable starting point for the derivation of numerical methods. A detailed discussion on underlying  isogeometric discrete differential forms in $3D$ can be found in \cite{Buffa2011IsogeometricDD} and below we state very shortly some important results from the last reference. 
	
	Define the pullbacks
	\begin{alignat}{3}
		\label{eq_def_pull_1}
	\mathcal{Y}_1(q)& \coloneqq  q \circ \f{F}, &&q \in L^2(\o),\\
	\mathcal{Y}_2(\f{v})& \coloneqq  \f{J}^T \big(\f{v} \circ \f{F}), && \f{v} \in \f{L}^2(\o), \\
	\mathcal{Y}_3(\f{v}) &\coloneqq  \textup{det}(\f{J}) \f{J}^{-1} \big(\f{v} \circ \f{F} \big), \ \ \ \ && \f{v} \in \f{L}^2(\o), \ \ 	\label{eq_def_pull_3} \\
	\mathcal{Y}_4(q)&\coloneqq  \textup{det}(\f{J}) \big(q \circ \f{F} \big), && q \in L^2(\o). 	\label{eq_def_pull_4}
	\end{alignat}
	
	Then, for  $\Omega \subset \mathbb{R}^3, \ p_i>r_i \geq 0$,  the spaces 
	\begin{alignat*}{3}
		V_{h,1}^3 &\coloneqq  \mathcal{Y}_1^{-1}(\widehat{V}^3_{h,1}) \ \ \  \textup{with}  	\hspace{0.3cm}&&\widehat{V}^3_{h,1} \coloneqq  S_{p_1,p_2,p_3}^{r_1,r_2,r_3} ,\nonumber \\
		\boldsymbol{V}_{h,2}^3 &\coloneqq \mathcal{Y}_2^{-1}(\widehat{\boldsymbol{V}}^3_{h,2})\ \ \  \textup{with}  	\hspace{0.3cm} &&\widehat{\boldsymbol{V}}^3_{h,2} \coloneqq  \big(S_{p_1-1,p_2,p_3}^{r_1-1,r_2,r_3} \times S_{p_1,p_2-1,p_3}^{r_1,r_2-1,r_3} \times S_{p_1,p_2,p_3-1}^{r_1,r_2,r_3-1} \big)^T  
		 ,&&  \nonumber\\
		\boldsymbol{V}_{h,3}^3 &\coloneqq \mathcal{Y}_3^{-1}(\widehat{\boldsymbol{V}}^3_{h,3})\ \ \  \textup{with}  	\hspace{0.3cm} &&\widehat{\boldsymbol{V}}^3_{h,3} \coloneqq  \big(S_{p_1,p_2-1,p_3-1}^{r_1,r_2-1,r_3-1} \times S_{p_1-1,p_2-2,p_3-1}^{r_1-1,r_2,r_3-1} \times S_{p_1-1,p_2-1,p_3}^{r_1-1,r_2-1,r_3} \big)^T   
		 &&   ,\nonumber \\
		V_{h,4}^3 &\coloneqq  \mathcal{Y}_4^{-1}(\widehat{V}^3_{h,4}) \ \ \  \textup{with}  	\hspace{0.3cm}&&\widehat{V}^3_{h,4} \coloneqq  S_{p_1-1,p_2-1,p_3-1}^{r_1-1,r_2-1,r_3-1},\nonumber 
	\end{alignat*}
	form an exact subcomplex of the de Rham chain, i.e. there are  continuous projections 
	\begin{align*}
		{\Pi}^3_{h,i} \colon L^2(\Omega) \rightarrow {V}_{h,i}^3,\  i \in \{1,4\} \ \ \ \textup{and} \ \ \ {\Pi}^3_{h,i} \colon \boldsymbol{L}^2(\Omega) \rightarrow {\boldsymbol{V}}_{h,i}^3, \ i \in \{2,3 \},
	\end{align*} such that we obtain the  commutative diagram
	
	\begin{figure}[H]
		\centering
		\begin{tikzpicture}
			\node at (0,1.4) {$H^1(\Omega)$};
			
			\node at (3,1.4) {$\boldsymbol{H}(\Omega,\textup{curl})$};

			\node at (6.5,1.4) {$\boldsymbol{H}(\Omega,\textup{div})$};

			\node at (9.5,1.4) {${L}^2(\Omega)$};

			\node at (1.3,1.7) {$\nabla$};

			\node at (4.7,1.7) {$\nabla \times $};

			\node at (8.2,1.7) {$\nabla \cdot $};

			\draw[->] (0.8,1.4) to (1.8,1.4);

			\draw[->] (4.2,1.4) to (5.3,1.4);
			
			\draw[->] (7.7,1.4) to (8.75,1.4);
			\draw[->] (10.2,1.4) to (10.9,1.4);
			\draw[->] (-1.6,1.4) to (-0.72,1.4);
			\node at (10.55,1.7) {$0$};
			\node at (-1.09,1.65) {$\iota$};
			\node at (11.45,1.4) {$\{0\}$};
			\node at (-2.25,1.4) {$\mathbb{R}$};

			\draw[->] (3,1.1) -- (3,0.2);
			\draw[->] (-2.25,1.1) -- (-2.25,0.2);
			\draw[->] (11.45,1.1) -- (11.45,0.2);
			\draw[->] (0,1.1) -- (0,0.2);
			\draw[->] (6.5,1.1) -- (6.5,0.2);
			\draw[->] (9.5,1.1) -- (9.5,0.2);

			\node at (0,-0.1) {${V}_{h,1}^{3}$};
			
			\node at (3,-0.1) {$\boldsymbol{V}_{h,2}^{3}$};

			\node at (6.5,-0.1) {$\boldsymbol{V}_{h,3}^{3}$};

			\node at (9.5,-0.1) {${V}_{h,4}^{3}$};

			\node at (1.3,0.2) {${\nabla}$};

			\node at (4.7,0.2) {${\nabla} \times $};

			\node at (8.2,0.2) {${\nabla} \cdot $};

			\draw[->] (0.8,-0.1) to (1.8,-0.1);

			\draw[->] (4.2,-0.1) to (5.3,-0.1);
			
			\draw[->] (7.7,-0.1) to (8.75,-0.1);
			\draw[->] (10.2,-0.1) to (10.9,-0.1);
			\draw[->] (-1.6,-0.1) to (-0.72,-0.1);
			\node at (10.55,0.2) {$0$};
			\node at (-1.09,0.15) {$\iota$};
			\node at (11.57,-0.1) {$\{0\} \ .$};
			\node at (-2.25,-0.1) {$\mathbb{R}$};
			
			\node[left] at (-2.2,0.7) { \small $\textup{id}$};	
			\node[left] at (0.05,0.7) { \small $\Pi_{h,1}^{3}$};
			\node[left] at (3.05,0.7) { \small $\Pi_{h,2}^{3}$};
			\node[left] at (6.55,0.7) { \small $\Pi_{h,3}^{3}$};
			\node[left] at (9.55,0.7) { \small $\Pi_{h,4}^{3}$};
			\node[left] at (11.45,0.7) { \small $\textup{id}$};
		\end{tikzpicture}
	\end{figure} \noindent
	Due to the smoothness properties of the splines  the discrete spaces are indeed subspaces of the continuous pendants.
	For reasons of simplification let us assume in the following $ p\coloneqq p_1=p_2=p_3$ and $r \coloneqq r_1 =r_2=r_3$.   Then, as shown in  \cite{Buffa2011IsogeometricDD}, we can define the above projections in such a way that the next approximation estimates are valid.
	
	\begin{lemma}
		\label{lemma:spline_approx_estimates_3D}
		Let $0  \leq  k \leq p, \ 0 \leq r \leq p-1$. Then, assuming sufficient regularity for $\phi$ and $\f{v}$, it is
		\begin{alignat*}{3}
			& \norm{{\phi}-\Pi^3_{h,1}{\phi}}_{H^l} \leq C  h^{k-l+1} \ \norm{{\phi}}_{H^{k+1}}, \ l \in \{0,1\}, \\
			& \norm{{\f{v}}-\Pi^3_{h,2}{\f{v}}}_{\f{H}(\textup{curl})} \leq C  h^{k} \ \norm{{\f{v}}}_{\f{H}^k(\textup{curl})}, \hspace{0.5cm}\\
			& \norm{{\f{v}}-\Pi^3_{h,3}{\f{v}}}_{\f{H}(\textup{div})} \leq C  h^{k} \ \norm{{\f{v}}}_{\f{H}^k(\textup{div})},\\
			& \norm{{\phi}-\Pi^3_{h,4}{\phi}}_{{L}^{2}} \leq C  h^{k} \ \norm{{\phi}}_{{H}^k}.
		\end{alignat*}
		We use the notation  $\norm{{\f{v}}}_{\f{H}^k(\textup{curl})}^2 \coloneqq \norm{ {\f{v}}}_{\f{H}^k}^2 + \norm{ \nabla \times {\f{v}}}_{\f{H}^k}^2$ and $\norm{{\f{v}}}_{\f{H}^k(\textup{div})}^2 \coloneqq \norm{ {\f{v}}}_{\f{H}^k}^2 + \norm{ \nabla \cdot {\f{v}}}_{{H}^k}^2$.	
	\end{lemma}
	\begin{proof}
		See \cite[Theorem 5.3,Theorem 5.4, Remark 5.1]{Buffa2011IsogeometricDD}.
	\end{proof}
	
	Here and in the rest of the article $C<\infty$ denotes a constant only depending on the domain $\Omega, \ \f{F}$ and the polynomial degrees as well as on the regularity parameter $r$ and may vary at different occurrences.

	In fact, we shall study the $2D$ elasticity equation for which the planar de Rham chain  will be useful  to handle the weak symmetry case. Nevertheless, we mentioned the classical isogeometric de Rham spaces in $\mathbb{R}^3$ above in order to  justify the next statements without giving a detailed proof. Namely, it is easy to check that the procedure in  \cite{Buffa2011IsogeometricDD} can be adapted  to the $2D$ setting.   
	With similar steps and proofs 
	one can show the existence of  another commutative diagram  
	\begin{figure}[h!]
		\begin{center}
			\begin{tikzpicture}
				\node at (0,1.4) {$H^1(\Omega)$};			
				\node at (3,1.4) {$\boldsymbol{H}(\Omega,\textup{div})$};
				\node at (6.5,1.4) {$L^2(\Omega)$};			
				\node at (9.5,1.4) {$\{0 \}$};						
				\node at (1.3,1.7) {$\textup{curl}$};			
				\node at (4.7,1.7) {$\nabla \cdot $};			
				\node at (8.2,1.7) {$0 $};			
				\draw[->] (0.8,1.4) to (1.8,1.4);			
				\draw[->] (4.2,1.4) to (5.3,1.4);			
				\draw[->] (7.7,1.4) to (8.75,1.4);			
				\draw[->] (-1.6,1.4) to (-0.72,1.4);			
				\node at (-1.09,1.65) {$\iota$};			
				\node at (-2.25,1.4) {$\mathbb{R}$};			
				\draw[->] (3,1.1) -- (3,0.2);
				\draw[->] (-2.25,1.1) -- (-2.25,0.2);			
				\draw[->] (0,1.1) -- (0,0.2);
				\draw[->] (6.5,1.1) -- (6.5,0.2);
				\draw[->] (9.5,1.1) -- (9.5,0.2);			
				\node at (0,-0.1) {${V}_{h,1}^{2}$};			
				\node at (3,-0.1) {$\boldsymbol{V}_{h,2}^{2}$};		
				\node at (6.5,-0.1) {${V}_{h,3}^{2}$};			
				\node at (9.62,-0.1) {$\{0 \} \ ,$};			
				\node at (1.3,0.2) {${\textup{curl}}$};			
				\node at (4.7,0.2) {${\nabla} \cdot $};			
				\node at (8.2,0.2) {$0 $};		
				\draw[->] (0.8,-0.1) to (1.8,-0.1);		
				\draw[->] (4.2,-0.1) to (5.3,-0.1);
				\draw[->] (7.7,-0.1) to (8.75,-0.1);		
				\draw[->] (-1.6,-0.1) to (-0.72,-0.1);			
				\node at (-1.09,0.15) {$\iota$};			
				\node at (-2.25,-0.1) {$\mathbb{R}$};			
				\node[left] at (-2.2,0.7) { \small $\textup{id}$};	
				\node[left] at (0.05,0.7) { \small $\Pi_{h,1}^{2}$};
				\node[left] at (3.05,0.7) { \small $\Pi_{h,2}^{2}$};
				\node[left] at (6.55,0.7) { \small $\Pi_{h,3}^{2}$};
				\node[left] at (9.55,0.7) { \small $\textup{id}$};
			\end{tikzpicture}
		\end{center}
		\caption{There is a IGA discretization of the planar de Rham complex preserving the chain structure.}
		\label{Fig:diagram2}
	\end{figure} \\ \noindent
	where we have the spaces   \begin{alignat*}{3}
		V_{h,1}^2 &\coloneqq  \mathcal{Y}_1^{-1}(\widehat{V}^2_{h,1}) \ \ \  \textup{with}  	\hspace{0.3cm}&&\widehat{V}^2_{h,1} \coloneqq  S_{p,p}^{r,r},  \\
		\boldsymbol{V}_{h,2}^2 &\coloneqq \mathcal{Y}_3^{-1}(\widehat{\boldsymbol{V}}^2_{h,2})\ \ \  \textup{with}  	\hspace{0.3cm} &&\widehat{\boldsymbol{V}}^2_{h,2} \coloneqq  \big(S_{p,p-1}^{r,r-1} \times S_{p-1,p}^{r-1,r}  \big)  
		,\nonumber\\
		V_{h,3}^2 &\coloneqq  \mathcal{Y}_4^{-1}(\widehat{V}^2_{h,3}) \ \ \  \textup{with}  	\hspace{0.3cm}&&\widehat{V}^2_{h,3} \coloneqq  S_{p-1,p-1}^{r-1,r-1},  \ \ \ 
	\end{alignat*}
	and there are compatible $L^2$-continuous projections $\Pi_{h,i}^2$ onto the respective spaces. One notes the same notation for the pullbacks in  $2D$ and  mean the straight-forward adaption of \eqref{eq_def_pull_1}, \ \eqref{eq_def_pull_3} and \eqref{eq_def_pull_4} to the two-dimensional setting. An important aspect showing the appropriateness of  the $\mathcal{Y}_i$ are the properties \begin{align}
		\label{eq_trans_compa}
		\widehat{\nabla} \cdot \mathcal{Y}_3(\f{v}) = \mathcal{Y}_4(\nabla \cdot \f{v}) , \ \ \ \ \widehat{\textup{curl}}(\mathcal{Y}_1(\phi)) = \mathcal{Y}_3(\textup{curl}(\phi)),
	\end{align}
	which can be proven easily. The    notation "  $\widehat{ \ \cdot \ }$ " indicates  that the  differentiation is  w.r.t. the parametric coordinates. 
	Analogously to the $3D$ case one can derive approximation estimates for the mentioned spaces. 
	\begin{lemma}
		\label{lemma_approx_spline_spaces_2D}
		Let $0  \leq k \leq p, \ 0 \leq r \leq p-1$. The  compatible projections  $\Pi_{h,i}^2$  satisfy the estimates
		\begin{alignat*}{3}
			& \norm{{\phi}-\Pi^2_{h,1}{\phi}}_{H^{l}} \leq C  h^{k-l+1} \ \norm{{\phi}}_{H^{k+1}}, \ l \in \{0,1\}, \\
			& \norm{{\f{v}}-\Pi^2_{h,2}{\f{v}}}_{\f{H}(\textup{div})} \leq C  h^{k} \ \norm{\f{v}}_{\f{H}^k(\textup{div})}, \hspace{0.5cm}\\
			& \norm{\phi-\Pi^2_{h,3}{\phi}}_{{L}^{2}} \leq C  h^{k} \ \norm{\phi}_{{H}^k},
		\end{alignat*}	
		assuming sufficient regularity.
	\end{lemma}
	\begin{proof}
		With similar proofs like for the $3D$ de Rham chain this can be shown with the  results from \cite{Buffa2011IsogeometricDD}. 
	\end{proof}

	At the end of this short section  we want to mention the Taylor-Hood type isogeometric
	spaces; see e.g. \cite{Buffa2010IsogeometricAN}. They are defined through
	\begin{align*}
		\boldsymbol{W}_{h}^{TH}(p,r) \coloneqq \big(	\boldsymbol{V}_{h}^{TH},	V_{h}^{TH} \big) \coloneqq \big( (S_{p,p}^{r,r}  \times S_{p,p}^{r,r} )\circ \f{F}^{-1},  S_{p-1,p-1}^{r,r}  \circ \f{F}^{-1} \big), \ \ p-1 > r>=0.
	\end{align*}
	Actually, latter spaces fit not directly to discrete differential forms. Nevertheless, the space pair will prove useful later and 
	utilizing  \cite[Theorem 3.2]{IGA3} or Lemma \ref{lemma_approx_spline_spaces_2D}  we see the existence of a  projection $\Pi_{h}^{TH} \colon L^2(\o) \rightarrow V_{h}^{TH}$ with approximation result
	\begin{equation}
		\label{eq_approx_th}
		\norm{q -\Pi_{h}^{TH}q}_{L^2} \leq C  h^{k} \norm{q}_{H^k}, \ \ p \geq k  \geq 0.
	\end{equation}
	
	After the introduction of elementary notions and spline spaces we look at the mixed weak formulation for linear elasticity with weak symmetry below.

	\section{Mixed formulation with weakly imposed symmetry}
	\label{section_weak_form_weak_symmetry}
	To obtain a mixed weak formulation for the elasticity system that does not require test spaces with strong symmetry, i.e. subspaces of  $\t{H}(\o,\d,\mathbb{S}) $, one can use the weak symmetry ansatz utilizing a Lagrange multiplier as introduced in  \cite{Veubeke} and applied several times e.g. in  \cite{Arnold84}  or \cite{Arnold2007MixedFE}: \emph{Find}	$\boldsymbol{\sigma} \in \boldsymbol{\mathcal{{H}}}(\Omega,\textup{div}), \ \f{u} \in \boldsymbol{L}^2(\Omega)$ and $\f{p} \in \boldsymbol{\mathcal{{L}}}^2(\Omega,\mathbb{K})$ s.t. 
	\begin{alignat}{5}
		\label{weak_form_wek_symmetry_contiuous}
		&\langle \f{A} \boldsymbol{\sigma} , \boldsymbol{\tau} \rangle \ + \ &&\langle \f{u}, \nabla \cdot \boldsymbol{\tau}  \rangle \  + \ && \langle \f{p}, \boldsymbol{\tau} \rangle  \color{black}  &&= \langle \boldsymbol{\tau} \cdot \f{n}; \f{u}_D \rangle_{\Gamma}, \ \hspace{0.4cm} &&\forall \boldsymbol{\tau} \in \boldsymbol{\mathcal{{H}}}(\Omega,\textup{div}), \nonumber\\
		&\langle \nabla \cdot \boldsymbol{\sigma} , \f{v} \rangle &&  && &&= \langle \f{f},\f{v}\rangle, \ \hspace{0.4cm} && \forall \f{v} \in \boldsymbol{L}^2(\Omega), \\
		&\langle \boldsymbol{\sigma}, \f{q} \rangle &&  &&  &&= 0,  \ \hspace{0.4cm} &&\forall \f{q} \in \boldsymbol{\mathcal{{L}}}^2(\Omega,\mathbb{K}), \nonumber
	\end{alignat}
	where $\f{u}_D \in \f{H}^{1/2}(\Gamma), \ \Gamma \coloneqq \partial \o,$ determines  the displacement boundary data and $\f{f} \in \f{L}^2(\o).$  For an explanation of  boundary Sobolev spaces we refer the reader to  \cite[Chapter 2]{steinbach}. 
	Starting point for the  derivation of latter system is the standard mixed weak formulation \eqref{weak_form_strong_symmetry_contiuous_1}.
	Then the symmetry condition is dropped and replaced by the requirement 
	\begin{equation*}
		\langle \t{\sigma} , \f{s} \rangle  = 0 , \ \  \forall \f{s} \in \t{L}^2(\o, \mathbb{K}), 
	\end{equation*}	 
	which is equivalent to 
	$$  \langle  s_i , \sigma_{ij}-\sigma_{ji} \rangle =0, \  \ \forall s_i \in L^2(\Omega), \ \ j\neq i.$$
	Latter constraint is incorporated by means of a Lagrange multiplier
	and this explains  the additional variable $\f{p}$  as well as the last line in system \eqref{weak_form_wek_symmetry_contiuous}. It is easy to see the equivalence of the two formulations \eqref{weak_form_strong_symmetry_contiuous_1} and  \eqref{weak_form_wek_symmetry_contiuous} in the continuous setting, where the Lagrange multiplier is related to the solution $(\f{\sigma},\f{u})$ by $\f{p}= -\textup{Skew}\big(\frac{1}{2}(\partial_2u_1 - \partial_1 u_2) \big) $.  The well-posedness of  \eqref{weak_form_wek_symmetry_contiuous} in  $2D$ is shown in \cite{Arnold2015} and the consideration of the weak symmetry formulation is in some sense  justified. Main advantage of the new form is the obtained flexibility for discretization approaches compared to \eqref{weak_form_strong_symmetry_contiuous_1}, but the price to pay is the increased number of degrees of freedom. At least in $3D$ and for fine meshes this might become problematic within numerical computations.
	
	Nevertheless, in the following we want to face the conforming discretization of \eqref{weak_form_wek_symmetry_contiuous} in the two-dimensional setting and thus we have to define proper test function spaces $\t{W}_{h} \coloneqq \t{V}_{h,2}^w \times \f{V}_{h,3}^w \times \t{V}_{h,3}^w \subset  \t{H}(\o,\d) \times \f{L}^2(\Omega) \times \t{L}^2(\o,\k)$ and then we solve the discrete problem: \\
	\emph{Find} $[\f{\sigma}_h,\f{u}_h,\f{p}_h]  \in \t{V}_{h,2}^w \times \f{V}_{h,3}^w \times \t{V}_{h,3}^w $ such that	
	\begin{alignat}{5}
		\label{dis_weak_form_weak_symmetry_contiuous}
		&\langle \f{A} \boldsymbol{\sigma}_h , \boldsymbol{\tau}_h \rangle \ + \ &&\langle \f{u}_h, \nabla \cdot \boldsymbol{\tau}_h \rangle \  + \ && \langle   \f{p}_h, \boldsymbol{\tau}_h \rangle  \color{black}  &&= \langle \boldsymbol{\tau}_h \cdot \f{n} ; \f{u}_D\rangle_{\Gamma}, \ \hspace{0.4cm} &&\forall \boldsymbol{\tau}_h \in \t{V}_{h,2}^w, \nonumber\\
		&\langle \nabla \cdot \boldsymbol{\sigma}_h , \f{v}_h \rangle &&  && &&= \langle \f{f},\f{v}_h\rangle, \ \hspace{0.4cm} && \forall \f{v}_h \in \f{V}_{h,3}^w, \\
		&\langle \boldsymbol{\sigma}_h, \f{q}_h \rangle &&  &&  &&= 0,  \ \hspace{0.4cm} &&\forall \f{q}_h \in \t{V}_{h,3}^w. \nonumber
	\end{alignat}
	
	The index $w$  should indicate that we are in the framework of weak symmetry.
	Then the actual test spaces will be constructed through the spline based spaces defined in Section \ref{subsec:iag_spaces}.
	Before we  specify the discretization method we state the fundamental conditions that have to be satisfied in order to get a stable saddle-point problem \eqref{dis_weak_form_weak_symmetry_contiuous}. The application of Brezzi's stability criterion for saddle-point problems to the mixed formulation (cf. \cite[Section 2]{Arnold2007MixedFE} and \cite{Brezzi}) implies the stability requirements: 
	\begin{itemize}
		\item[(S1)]  It holds $\norm{\f{\tau}_h}_{\t{H}(\d)}^2 \leq C_{S1} \langle \f{A}\f{\tau}_h, \f{\tau}_h \rangle   \ $ for all $\f{\tau}_h \in \t{V}_{h,2}^w $ satisfying $0=\langle \nabla \cdot \f{\tau}_h, \f{v}_h  \rangle$ and $0=\langle  \f{\tau}_h, \f{q}_h  \rangle,$ $ \ \forall (\f{v}_h, \f{q}_h) \in \ \f{V}_{h,3}^w \times \t{V}_{h,3}^w$ and some constant $C_{S1}$.
		\item[(S2)] For all $ (\f{v}_h, \f{q}_h) \in \ \f{V}_{h,3}^w \times \t{V}_{h,3}^w $, there exists a $ \f{0} \neq \f{\tau}_h \in \t{V}_{h,2}^w$ s.t. \\
		{ $	\langle \f{\tau}_h, \f{q}_h \rangle + \langle \nabla \cdot \f{\tau}_h, \f{v}_h \rangle \geq C_{S2} \norm{\f{\tau}_h}_{\t{H}(\d)} \big(\norm{\f{v}_h}_{\f{L}^2}+ \norm{\f{q}_h}_{\t{L}^2} \big)$} for a constant $C_{S2}$.
	\end{itemize}
	If the conditions (S1) and (S2) are fulfilled we get automatically well-posedness and moreover  the next theorem with results from \cite{Arnold2015} and \cite{Rettung}  applies. More precisely, we refer to  Theorem 3.1 in \cite{Rettung}, \cite[Section 2]{Arnold2015} and \cite[Remark 1]{Arnold2015}.
	
	\begin{theorem}
		\label{Theorem_2D_stability}
		Let the test spaces fulfill (S1) and (S2). Then the problem \eqref{dis_weak_form_weak_symmetry_contiuous} has a unique solution. Further, we get the quasi-optimal estimate 
		\begin{equation}
			\label{eq_quasi-optimal_weak}
			\norm{[\f{\sigma}-\f{\sigma}_h,\f{u}-\f{u}_h,\f{p}-\f{p}_h]}_\mathcal{B} \leq C \norm{[\f{\sigma}-\f{\tau}_h,\f{u}-\f{v}_h,\f{p}-\f{q}_h]}_\mathcal{B}  \ \ \forall [\f{\tau}_h,\f{v}_h,\f{q}_h] \in \t{W}_{h}
		\end{equation} between the exact solution $(\f{\sigma},\f{u},\f{p})$ of \eqref{weak_form_wek_symmetry_contiuous} and the solution to \eqref{dis_weak_form_weak_symmetry_contiuous},\\	where $\norm{[\f{\sigma},\f{u},\f{p}]}_\mathcal{B}^2 \coloneqq \norm{\f{\sigma}}_{\t{H}(\d)}^2 + \norm{\f{u}}_{\f{L}^2}^2 + \norm{\f{p}}_{\t{L}^2}^2$ and $C$ depends only  on $C_{S1}, C_{S2}$ and $\f{A}$.\\
		The discretization is inf-sup stable, meaning we have a constant $C_{\mathcal{B}}>0$ independent from $h$  s.t.
		\begin{equation*}
			\underset{[\f{\sigma}_h,\f{u}_h,\f{p}_h] \in \t{W}_h \textbackslash \{\f{0}\}}{\inf} \hspace{0.2cm} \underset{[\f{\tau}_h,\f{v}_h,\f{q}_h] \in \t{W}_h}{\sup}\frac{\mathcal{B}([\boldsymbol{\tau}_h,\f{v}_h,\f{q}_h],[\boldsymbol{\sigma}_h,\f{u}_h,\f{p}_h])}{\norm{[\f{\tau}_h,\f{v}_h,\f{q}_h]}_\mathcal{B} \ \norm{[\f{\sigma}_h,\f{u}_h,\f{p}_h]}_\mathcal{B} } \geq C_{\mathcal{B}},
		\end{equation*}
		with $	\mathcal{B}([\f{\sigma},\f{u},\f{p}],[\f{\tau},\f{v},\f{q}]) \coloneqq \langle \f{A} \boldsymbol{\sigma} , \boldsymbol{\tau} \rangle \ + \langle \f{u}, \nabla \cdot \boldsymbol{\tau}  \rangle \  + \langle \f{p}, \boldsymbol{\tau} \rangle + \langle \nabla \cdot \boldsymbol{\sigma} , \f{v} \rangle+ \langle \boldsymbol{\sigma} , \f{q} \rangle$ denoting the underlying bilinear form of \eqref{dis_weak_form_weak_symmetry_contiuous}.
		Besides, if  the constant identity matrix $\f{I}$ is an element of $ \t{V}_{h,2}^w$, then the constant $C_{\mathcal{B}}$ and the constant $C$ in  \eqref{eq_quasi-optimal_weak} are independent of the Lam\'e coefficient $\lambda$.
	\end{theorem}

	After our problem model is clarified we define  spaces for the two-dimensional setting yielding a stable discretization.

	\subsection{Choice of discrete spaces}
	\label{section_discretization}
	
	The  choice of suitable test spaces for $\Omega \subset \mathbb{R}^2$  satisfying the discrete inf-sup condition is determined by the following lemma that corresponds to  Lemma 3.2 in   \cite{Rettung}. It gives us a helpful guideline how to satisfy the conditions (S1) and (S2).

	\begin{lemma} 
		\label{Lemma_Arnold_15}
		Let    $\f{W}_h^1 \subset \f{H}(\Omega,\textup{div}), \ W_h^2 \subset L^2(\Omega)$ be a pair of spaces  such that 
		\begin{align*}
			\underset{p_h \in W_h^2 \textbackslash \{0 \}}{\inf} \   \underset{\f{v}_h \in \f{W}_h^1 }{\sup}	\ \frac{ \langle \nabla \cdot \f{v}_h, p_h \rangle }{ \norm{\f{v}_h}_{\f{H}(\d)} \norm{p_h}_{L^2} } \geq C_P > 0, 
		\end{align*}
		and let  $\f{W}^0_h \subset \boldsymbol{H}^1(\Omega), \ W_h^3 \subset L^2(\Omega)$ be a second pair of discrete spaces with 
		\begin{align*}
			\underset{q_h \in W_h^3\textbackslash \{0 \}}{\inf} \ \underset{\f{w}_h \in \f{W}_h^0 }{\sup} \ 	\frac{ \langle \nabla \cdot \f{w}_h, q_h \rangle }{ \norm{\f{w}_h}_{\boldsymbol{H}^1}  \norm{q_h}_{L^2}} \geq C_S > 0, 
		\end{align*}
		and assume the  constants $0<C_P, C_S$ to be independent of the underlying mesh size $h$.\\
		If  $  \textup{curl}(w_{h,i})  \in \f{W}_h^1 , \forall \f{w}_h=(w_{h,1},w_{h,2})^T \in \f{W}^0_h  $ and $\nabla \cdot \f{W}_h^1 \subset W_h^2$ 
		then \begin{equation}
			\label{eq_theorem_choice}
			\boldsymbol{\mathcal{V}}^w_{h,2} \coloneqq \big( \boldsymbol{W}_h^1 \times \boldsymbol{W}_h^1 \big)^T, \ \ \f{{V}}_{h,3}^w \coloneqq \big(W_h^2 \times W_h^2\big)^T  \ \ and \ \ \t{V}_{h,3}^w \coloneqq \{ \textup{Skew}(q)  \ |  \ q \in W_h^3  \}
		\end{equation}
		fulfill the conditions (S1) and (S2). \footnote{$\big( \boldsymbol{W}_h^1 \times \boldsymbol{W}_h^1 \big)^T$ denotes  the space of matrix fields whose  rows are elements of $ \boldsymbol{W}_h^1 $. }
	\end{lemma}
	\begin{proof}
		Compare Lemma 3.2 in \cite{Rettung}.
	\end{proof}

	Now we define the isogeometric discrete  spaces meeting the requirements of Lemma \ref{Lemma_Arnold_15} for the two-dimensional setting. In view of the mentioned lemma and the definitions in Section \ref{subsec:iag_spaces} we choose:
	\begin{align}
		\label{choice_1}
		\f{W}_h^1 & \coloneqq \f{V}_{h,2}^2 ,\\
			\label{choice_2}
		W_h^2 & \coloneqq  V_{h,3}^2.
	\end{align}
	And we use a Taylor-Hood  pair, namely
	\begin{align}
		\label{choice_3}
		\f{W}^0_h  \times W^3_h& \coloneqq \f{W}_{h}^{TH}(p,r)=\big(	\boldsymbol{V}_{h}^{TH},	V_{h}^{TH} \big). 
	\end{align}
	
	For the above choices of the discrete spaces  the following two  lemmas resulting from \cite{IGA3} and \cite{IGA1} apply.
	\begin{lemma}
		\label{lemma_B1}
		There exists a constant $0<C_P$ independent of $h$ s.t. 
		\begin{align}
			\underset{p_h \in V_{h,3}^2 \textbackslash \{0 \}}{\inf}  \ \ \underset{\f{v}_h \in \f{V}_{h,2}^2}{\sup}	\frac{ \langle \nabla \cdot \f{v}_h, p_h \rangle }{ \norm{\f{v}_h}_{\f{H}(\d)} \ \norm{p_h}_{L^2} } \geq  C_P > 0 
		\end{align}			
	\end{lemma} 
	\begin{proof}
		First one notes the commutativity of the  diagram in  Fig. \ref{Fig:diagram2} and the $L^2$-boundedness of the corresponding projections $\Pi_{h,i}^2$, meaning there is a structure-preserving discretization of the planar de Rham complex. As already mentioned, this can be seen  by an adaption of the discretization approach in \cite{Buffa2011IsogeometricDD} to the planar de Rham complex. Consequently, we can easily adapt the elegant proof of  Theorem 2 in \cite[pp. 155-193]{IGA1} to our  two-dimensional de Rham chain. 
	\end{proof}

	\begin{lemma}
		\label{lemma_B2}
		Let  $p>=2, \ r=0$ and $\Gamma_B \subsetneq \partial \Omega$, where $\Gamma_B$ is the union of full mesh edges. Then there exists $h_{max}>0$, which depends only on $\f{F}$, and a constant $0<C_S$ independent of $h$ s.t. 
		\begin{align}
			\underset{q_h \in V_{h}^{TH} \textbackslash \{0\}}{\inf} \  \  \underset{\f{w}_h \in \f{V}_{h}^{TH} \cap \f{H}_{\Gamma_B}^1(\o) }{\sup}	 \ \frac{ \langle \nabla \cdot \f{w}_h, q_h \rangle }{ \norm{\f{w}_h}_{\boldsymbol{H}^1} \ \norm{q_h}_{L^2} } \geq C_S > 0,  \ \ \forall h \leq h_{max} .
		\end{align}	We use the abbreviation $\f{H}_{\Gamma_B}^1(\o) \coloneqq \{ \f{v} \in\f{H}^1(\o) \ | \ \f{v}=\f{0} \ \textup{on} \ \Gamma_B \}$
	\end{lemma}
	\begin{proof}
		This follows  from Theorem 5.2 in \cite{IGA3} and its proof for the case of a constant weight function, where one notes the equivalence of the Sobolev semi-norm $| \cdot |_{H^1}$ and the standard Sobolev norm $\norm{\cdot}_{H^1}$ in case of zero Dirichlet boundary conditions on a part of the boundary. Further one observes Assumption \ref{assum_reg_mesh}. 
	\end{proof}	
	And  we have the next  auxiliary result.
	\begin{lemma}
		\label{Lemma_2D_spaces}
		Let $p > r+1 \geq 1$. \\ It holds  
		$\f{V}_{h,2}^2 \subset \f{H}(\Omega,\textup{div}), \ V_{h,3}^2 \subset L^2(\Omega), \ \f{V}_{h}^{TH} \subset \boldsymbol{H}^1(\Omega), \ V_{h}^{TH} \subset L^2(\Omega)$. Further, if $(w_{h,1},w_{h,2})^T \in \f{V}_h^{TH}$, then $\textup{curl}(w_{h,i}) \subset \f{V}_{h,2}^2 $ as well as $\nabla \cdot \f{V}_{h,2}^2 \subset V_{h,3}^2$.
	\end{lemma}
	\begin{proof}
		This  is clear by the regularity properties of B-splines and \eqref{eq_trans_compa}; see  Fig.  \ref{Fig:diagram2}.
	\end{proof}

	Hence, using  the last three lemmas and applying Lemma \ref{Lemma_Arnold_15}, we see directly the next result:
	
	\begin{theorem}[Discrete spaces for weak symmetry]
	If $p\geq 2, r=0$ and if we  define the test spaces according to \eqref{eq_theorem_choice} with \eqref{choice_1} - \eqref{choice_3}. Then, the corresponding discretized weak form with weakly imposed symmetry is well-posed for $h$ small enough.
	\end{theorem} 
	
	\begin{remark}
		\label{remark_stability}
		An interesting point of the mixed formulation is the possibility to reach stability w.r.t. $\lambda$, which gets important in case of nearly incompressible materials. In view of Lemma \ref{Lemma_Arnold_15} we need to guarantee $\f{I} \in \t{V}_{h,2}^w$ which is equivalent to require $\textup{det}(\f{J})  \f{J}^{-1} \in \big( \widehat{\boldsymbol{V}}_{h,2}^2 \times \widehat{\boldsymbol{V}}_{h,2}^2 \big)$. This implies the conditions $(\f{J}_{22},-\f{J}_{21}) \in S_{p,p-1}^{r,r-1} \times S_{p-1,p}^{r-1,r}, \ (-\f{J}_{12},\f{J}_{11}) \in S_{p,p-1}^{r,r-1} \times S_{p-1,p}^{r-1,r} $. Thus, if we use B-spline parametrizations of the form $ \f{F} \in (S_{p,p}^{r,r})^2$    we obtain  $\f{I} \in \t{V}_{h,2}^w$. In particular, the stability condition fits to the isoparametric paradigm within IGA.
	\end{remark}

	\subsection{Error estimation}
	
	Because of  Theorem \ref{Theorem_2D_stability} and in view of the approximation properties of spline spaces, see Lemma \ref{lemma_approx_spline_spaces_2D}, we can estimate easily  the discretization error for the method \ref{dis_weak_form_weak_symmetry_contiuous}.	
	One obtains the following result.
	\begin{theorem}[Discretization error for weak symmetry]
		\label{Theorem_dis_error_2D}
		Let $p\geq2, r=0$ and choose discrete spaces for \eqref{dis_weak_form_weak_symmetry_contiuous} as  determined by \eqref{choice_1}-\eqref{choice_3} and \eqref{eq_theorem_choice}. \\
		Further let the exact solution of \eqref{weak_form_wek_symmetry_contiuous} satisfy the regularity conditions $\f{\sigma} \in \t{H}^{k}(\o,\d), \f{u} \in \f{H}^k(\o), \ \f{p} \in \t{H}^k(\o) \cap \t{L}^2(\o,\mathbb{K})$ with $p \geq k\geq 0$. \footnote{$\f{\tau} \in \t{H}^k(\o,\d)$ iff $\f{\tau} \in \t{H}^k(\o)$ and $ \nabla  \cdot \f{\tau} \in \f{H}^k(\o)$. Further, $\norm{{\f{\tau}}}_{\t{H}^k(\textup{div})}^2 \coloneqq \norm{ {\f{\tau}}}_{\t{H}^k}^2 + \norm{ \nabla \cdot {\f{\tau}}}_{\f{H}^k}^2$.} \\
		Then the estimates 
		\begin{align*}
			\norm{[\f{\sigma}-\f{\sigma}_h,\f{u}-\f{u}_h,\f{p}-\f{p}_h]}_\mathcal{B} &\leq C h^{k} \ \big( \norm{\f{\sigma}}_{\t{H}^k(\d)} + \norm{\f{u}}_{\f{H}^k} + \norm{\f{p}}_{\t{H}^k} \big),\\
			\n{\nabla \cdot \f{\sigma}-\nabla \cdot \f{\sigma}_h}_{\f{L}^2} &\leq C h^k \n{\nabla \cdot \f{\sigma}}_{\f{H}^k}
		\end{align*} 
		are valid for $h$ small enough (compare Lemma \ref{lemma_B2}).
	\end{theorem}
	\begin{proof}
		First let us generalize the projections $\Pi_{h,i}^2$  from Section \ref{subsec:iag_spaces} to the matrix setting, vector setting respectively, by a row-wise definition, i.e. $\Pi_{h,3}^2 \f{u} \coloneqq (\Pi_{h,3}^2 u_1, \Pi_{h,3}^2 u_2)^T$ if $ \f{u} = (u_1,u_2)^T$ and  $\Pi_{h,2}^2 \f{\sigma} \coloneqq 
		\begin{bmatrix}
			\Pi_{h,2}^2 \f{\sigma}_1^T \ \Pi_{h,2}^2 \f{\sigma}_2^T
		\end{bmatrix}^T,$  if $ \f{\sigma} =  \begin{bmatrix}
		\f{\sigma}_1 \ \f{\sigma}_2	
		\end{bmatrix}^T.$  Moreover, for $\f{p}= \textup{Skew}(p)$ we define  $\Pi_h^{TH}\f{p} \coloneqq \textup{Skew}(\Pi_h^{TH}p)$. Then, the first estimate is a  consequence of Theorem \ref{Theorem_2D_stability}, Lemma \ref{lemma_approx_spline_spaces_2D} and \eqref{eq_approx_th}. More precisely, it is
		\begin{align*}
			\norm{[\f{\sigma}-\f{\sigma}_h,\f{u}-\f{u}_h,\f{p}-\f{p}_h]}_\mathcal{B}  &\leq C \norm{[\f{\sigma}-\Pi_{h,2}^2 \f{\sigma},\f{u}-\Pi_{h,3}^2\f{u},\f{p}-\Pi_{h}^{TH}\f{p}]}_\mathcal{B}  \\
			& \leq  C \Big( \norm{\f{\sigma}-\Pi_{h,2}^2 \f{\sigma}}_{\t{H}(\d)}  + \norm{\f{u}-\Pi_{h,3}^2\f{u}}_{\f{L}^2} + \norm{\f{p}-\Pi_{h}^{TH}\f{p}}_{\t{L}^2}  \Big) \\
			& \leq C h^{k} \ \big( \norm{\f{\sigma}}_{\t{H}^k}  + \norm{\f{u}}_{\f{H}^k} + \norm{\f{p}}_{\t{H}^k} \big).
		\end{align*}
		The second estimate of the assertion can be obtained by a standard approach. Due to the second lines of \eqref{dis_weak_form_weak_symmetry_contiuous} and \eqref{weak_form_wek_symmetry_contiuous} together with $\Pi_{h,3}^2\nabla \cdot \f{\sigma} - \nabla \cdot \f{\sigma}_h \in \f{V}_{h,3}^w$,    we have 
		\begin{align*}
			\n{\nabla \cdot \f{\sigma}-\nabla \cdot \f{\sigma}_h}^2_{\f{L}^2} &= \langle \nabla \cdot \f{\sigma}-\nabla \cdot \f{\sigma}_h  , \nabla \cdot \f{\sigma}-\nabla \cdot \f{\sigma}_h \rangle \\
			&= \langle \nabla \cdot \f{\sigma}-\Pi_{h,3}^2 \nabla \cdot \f{\sigma}, \nabla \cdot \f{\sigma}-\nabla \cdot \f{\sigma}_h \rangle  \\ &  \ \ \ \ + \underbrace{\langle \Pi_{h,3}^2\nabla \cdot \f{\sigma}-\nabla \cdot \f{\sigma}_h , \nabla \cdot \f{\sigma}-\nabla \cdot \f{\sigma}_h \rangle}_{=0} \\
			& \leq \n{\nabla \cdot \f{\sigma}- \Pi_{h,3}^2 \nabla \cdot \f{\sigma}}_{\f{L}^2} \ \n{ \nabla \cdot \f{\sigma} - \nabla \cdot \f{\sigma}_h}_{\f{L}^2}.
		\end{align*}
		Thus, again in view of Lemma  \ref{lemma_approx_spline_spaces_2D}, we get \newline  \hspace{1cm} \ \ \ $ \hspace{1cm} \n{\nabla \cdot \f{\sigma} -\nabla \cdot \f{\sigma}_h}_{\f{L}^2} \leq \n{\nabla \cdot \f{\sigma}- \Pi_{h,3}^2 \nabla \cdot \f{\sigma}}_{\f{L}^2} \leq C h^k  \n{\nabla \cdot \f{\sigma}}_{\f{H}^k}.$ 
	\end{proof}
	
	In the next section we try to generalize the discretization approach from above with weakly imposed symmetry to more general settings including  multi-patch parametrizations and the case with traction boundary conditions.
	
	\subsection{Remarks concerning generalizations}
	
	\subsubsection{Higher regularity splines}
	\label{subsub_sec_regu}
	A major advantage of the B-spline spaces is the possibility to increase the global smoothness without worsening the approximation capabilities,  where the polynomial degrees bound the regularity parameter. To be more precise,  one can choose for example    $  r \leq p-1$ in the definition of the isogeometric de Rham spaces.  Therefore, a generalization of the discretization method  for $r>0$ seems natural.  In fact, for most of the steps above we do not  need to set $r=0$ and a lot of  properties like the approximation estimates in Lemma \ref{lemma_approx_spline_spaces_2D} and \eqref{eq_approx_th} are still valid if the   regularity parameter is increased considering the underlying polynomial degrees. Only the inf-sup stability result from Lemma  \ref{lemma_B2} uses $r=0$. But in the IGA community there is numerical evidence  indicating the uniform inf-sup stability of the Taylor-Hood pair also for $r>0$; see \cite{Buffa2010IsogeometricAN}. Hence, if we make the reasonable assumption of inf-sup stability of the Taylor-Hood space pair for increased regularity parameter we obtain   error  estimates analogous to the ones in Theorem \ref{Theorem_dis_error_2D} also for $1 \leq r+1<p$. The numerical experiments we carried  out in Section \ref{sec_numerics} utilizing splines with higher regularity  support the assumption of stability.

	\subsubsection{Mixed boundary conditions}
	Up to now,  we supposed Dirichlet data for the displacement $\f{u}$ which  are incorporated in a weak sense. Consequently, also for $\f{u}_D =0$ the discrete solution might not be exactly zero on $\partial \o$ in contrast to the standard approximations via the primal weak form where zero boundary conditions (BCs) are fulfilled exactly; see \eqref{eq_primal_zero_BC}. Another central boundary condition within linear elasticity is the traction boundary condition corresponding to an application of a boundary force. In the mentioned primal formulation with unknown $\f{u}$ such force conditions can be incorporated in the sense of Neumann type BCs. From the mixed method point of view this changes and traction forces are now implemented   explicitly, in a strong sense, respectively. First of all, we consider the model problem \begin{alignat}{2}
		\f{A} \f{\sigma} &= \f{\varepsilon}(\f{u}),  \hspace{2.6cm} \nabla \cdot \f{\sigma} &&= \f{f} \ \ \ \ \textup{in} \ \ \o, \\
		\f{u}&=\f{u}_D \ \ \ \textup{on} \ \Gamma_D,  \hspace{1.5cm} \f{\sigma} \cdot \f{n} &&= \f{t}_n \ \ \ \textup{on}  \ \Gamma_t, 
	\end{alignat}  
	with $\Gamma_D \cap \Gamma_t = \emptyset, \ \overline{\Gamma_D} \cup \overline{\Gamma_t}= \Gamma \coloneqq \partial \o$ and $\f{n}$ denoting the outer unit normal to the domain boundary. Further we require that $\Gamma_D, \  \Gamma_t $ are composed of full  mesh edges of the initial  mesh and $\emptyset \neq \Gamma_D$. One notes the well-definedness of the boundary condition in case  of $\f{\sigma}\in \t{H}(\d,\o)$ and $\f{t}_n \in \f{H}^{-1/2}(\o)$ due to the next lemma; see e.g. \cite{ArnoldBook}. \footnote{$\f{H}^{-1/2}(\o)$ is the dual space of $\f{H}^{1/2}(\o)$ which is in turn the image of the classical (component-wise) trace operator $ \gamma \colon H^1(\o) \rightarrow H^{1/2}(\Gamma)$; see \cite[Section 2]{steinbach}. }
	
	\begin{lemma}[Normal trace space of $\f{H}(\d)$]
		\label{lemma_trace_theorem}
		There is a bounded trace operator $$\gamma_{n} \colon \f{H}(\o,\d) \rightarrow H^{-1/2}(\partial\o), $$\ 
		s.t. $\gamma_n(\f{v})= (\f{v} \cdot \f{n} )_{|\partial\o}$ for all $\f{v} \in C^1(\overline{\o})$.
	\end{lemma}

\begin{remark}
		The last result can be generalized to matrix fields through a row-wise definition and we refer to \cite[Section 3.4]{ArnoldBook} for more information about spaces like $H^{-1/2}(\partial\o)$. Although the  mentioned reference concentrates on the three-dimensional setting, the $2D$ adaption is straight-forward. 
\end{remark}

	For $\f{t}_n = \f{0}$ the traction condition can be  fulfilled exactly in the numerical computations by dropping all basis functions violating it. 
	 Let  now $\f{t}_n$  be determined by the restriction  $\f{t}_n =  \tilde{\f{\sigma}} \cdot \f{n}$ of some given $ \tilde{\f{\sigma}} \in \t{H}(\o,\d)$ to the boundary, then we can easily adapt the problem to the homogeneous case by seeking $\f{\sigma}-\tilde{\f{\sigma}}$. If such an auxiliary  mapping is not known or if it   does not exist one can try to satisfy the traction BC approximately by computing a suitable approximate mapping $$\tilde{\f{t}}_n \approx \f{t}_n, \ \ \tilde{\f{t}}_n = \textstyle\sum_i c_i \f{\tau}_{h,i}, \ \ \f{\tau}_{h,i} \in \t{V}_{h,2}^{w},\  c_i \in \mathbb{R},$$ utilizing the underlying basis functions. However, in the rest of this section we assume $\f{t}_n=\f{0}$.
	Then the discrete
	spaces have to be changed according to 
	\begin{align}
		&\t{V}_{h,2}^w \rightarrow  \t{V}_{0,h,2}^w  \coloneqq \big(\f{V}_{0,h,2}^2 \times \f{V}_{0,h,2}^2  \big)^T, \ \ \f{V}_{0,h,2}^2 \coloneqq \f{V}_{h,2}^2 \cap \{  \f{v} | \f{v} \cdot \f{n} =0 \ \textup{on} \ \Gamma_t \},\\
		&\textup{and in \eqref{choice_3} we replace }
		\f{V}_{h}^{TH} \rightarrow \f{V}_{0,h}^{TH} \coloneqq \f{V}_{h}^{TH} \cap \{ \f{v} | \f{v} = \f{0} \ \textup{on} \  \Gamma_t \}.
	\end{align}
	Obviously, it is possible to apply Lemma \ref{Lemma_Arnold_15} for the new function spaces to check for inf-sup stability.
	The modification of the auxiliary test space $\f{V}_{h}^{TH}$ above is explained by the next lemma.
	\begin{figure}[h!]
		\centering
		\begin{minipage}{4cm}
			\begin{tikzpicture}[scale=2.6]
				
				\node[inner sep=0pt] (square) at (0.466,0.48)
				{\includegraphics[width=1.088\textwidth]{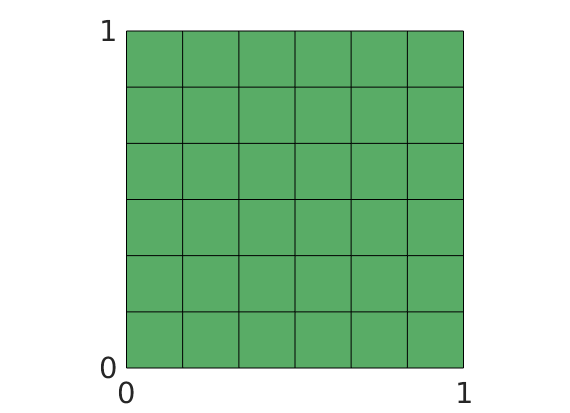}};
				\node[left] at (0,0.5) {$\hat{\Gamma}_1$};
				\node[right] at (1,0.5) {$\hat{\Gamma}_2$};
				\node[below] at (0.5,0) {$\hat{\Gamma}_3$};
				\node[above] at (0.5,1) {$\hat{\Gamma}_4$};
				\node[left,blue] at (-0.05,-0.1) {$\hat{\Gamma}_t$};
				\node[above,red] at (0.85,1) {$\hat{\Gamma}_D$};
				
				\draw[thick, ->] (0,0) -- (1.3,0);
				\draw[thick, ->] (0,0) -- (0,1.3);
				\draw[very thick,blue] (0,0) --(1,0);
				\draw[very thick,blue] (1,0) --(1,1);
				\draw[very thick,blue] (0,0) --(0,1);
				
				\draw[very thick,red] (0,1) --(1,1);
				\node[below] at (1.2,0) {$\zeta_1$};
				\node[left] at (0,1.2) {$\zeta_2$};
			\end{tikzpicture}
		\end{minipage}	
		\hspace{0.3cm}
		\begin{minipage}{1.5cm}
			\begin{tikzpicture}
				\draw[-> ,very thick] (0,0) --(1.5,0);
				\node at (0.75,0.2) {$\f{F}$};
			\end{tikzpicture}
		\end{minipage}			
		\begin{minipage}{5cm}
			\begin{tikzpicture}
				\node[inner sep=0pt] (ring) at (0,0)
				{\includegraphics[width=1.1\textwidth]{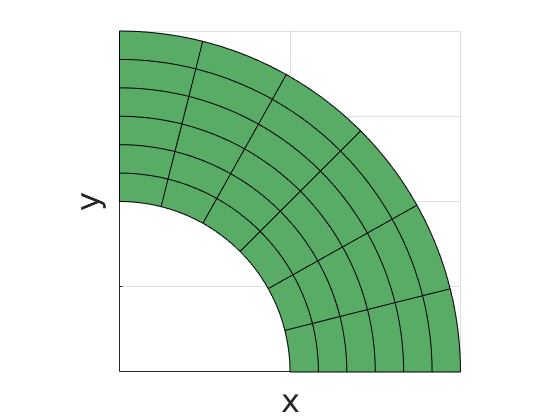}};
				\draw[very thick,blue] (0.1,-1.6) --(1.78,-1.6);
				\draw[very thick,red] (-1.57,0.05) --(-1.57,1.75);
				\draw [blue,very thick,domain=0:90] plot ({1.67*cos(\x)-1.57}, {1.67*sin(\x)-1.6});
				\draw [blue,very thick,domain=0:90] plot ({3.35*cos(\x)-1.57}, {3.35*sin(\x)-1.6});
				\node[left] at (-0.3,-0.5) {${\Gamma}_1$};
				\node[left] at (1.5,0.92) {${\Gamma}_2$};
				\node[left] at (1.38,-1.9) {${\Gamma}_3$};
				\node[left] at (-1.5,1) {${\Gamma}_4$};
			\end{tikzpicture}
		\end{minipage}
		\caption{Notation for the boundary edges. We assume  $\Gamma_i = \f{F}(\hat{\Gamma}_i)$. \footnotesize (Figure generated utilizing \cite{geopdes3.0}.)}
		\label{Fig:boundaries}
	\end{figure}
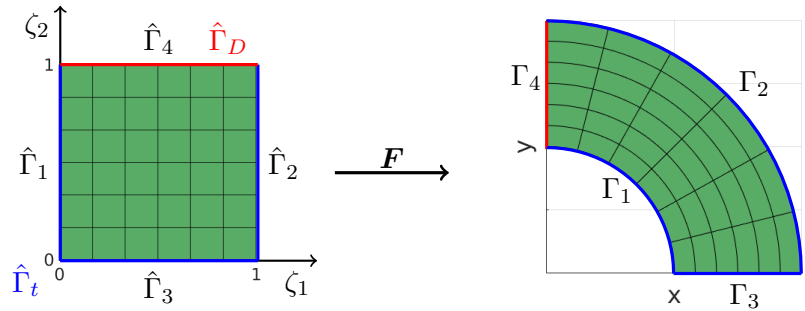

	\begin{lemma}
		Assume $\f{v} \coloneqq (v_1,v_2)^T \in \f{V}_{0,h}^{TH}$. Then we have $\textup{curl}(v_i) \in \f{V}_{0,h,2}^2$.
	\end{lemma}
	\begin{proof}
		Since it is $\textup{curl}(v_i) \in \f{V}_{h,2}^2$ and due to the fact that $\mathcal{Y}_1$ and $\mathcal{Y}_3$ preserve the Dirichlet BC, the zero normal trace, respectively, it is enough to verify the assertion for $\f{F}= \textup{id}$. We check it for the case $\Gamma_t \subset \hat{\Gamma}_1 \coloneqq \{ (0,\zeta_2) \ | \zeta_2 \in [0,1] \}$ of $\widehat{\o}$. And indeed, if $v_i=0 $ on ${\Gamma}_t$, then it is $\hat{\partial}_2 v_i=0$ on ${\Gamma}_t$, where $\h_j $ denotes the derivative w.r.t. the $j$-th parametric coordinate. Thus, $\widehat{\textup{curl}}(v_i) \cdot (-1,0)^T=-\hat{\partial}_2v_i = 0 $ on ${\Gamma}_t$. Other cases follow by a similar reasoning.
	\end{proof}
	
	In view of latter lemma it remains to check whether the both inf-sup conditions of Lemma \ref{Lemma_Arnold_15} are still valid for the spaces with BCs. Currently, we are only able to prove the inf-stability if we make the next assumption. 
	
	\begin{assumption}
		\label{assumption_boundary}
		Let us assume 
		that  there is a boundary edge which is not part of $\Gamma_t$, i.e. we have an edge $\Gamma_i$ s.t. $\Gamma_i \subset \Gamma_D $; see Fig. \ref{Fig:boundaries}.
	\end{assumption}

	Nevertheless, as long as $\Gamma_t \subsetneq \partial \o$, we think that stability is still valid.

	\begin{lemma}
		\label{Lemma_mixed_boubdary_weak_sym}
		Let $\Gamma_D$ fulfill Assumption \ref{assumption_boundary} and $p\geq 2, \ r=0$. Then, for the choice \eqref{choice_2}, \\  $\f{W}_h^1  \coloneqq \f{V}_{0,h,2}^2$ and $\f{W}^0_h  \times W^3_h \coloneqq \big(	\boldsymbol{V}_{0,h}^{TH},	V_{h}^{TH} \big)$ both inf-sup conditions of Lemma 	\ref{Lemma_Arnold_15} are satisfied. 
	\end{lemma}

	\begin{proof}
		W.l.o.g. we suppose $\Gamma_t =  \f{F}(\hat{\Gamma}_t) \subset \f{F}(\hat{\Gamma}_1) \cup \f{F}(\hat{\Gamma}_2) \cup \f{F}(\hat{\Gamma}_3) $; see Fig. \ref{Fig:boundaries}. \\
		Let $\hat{p} \in \widehat{V}_{h,3}^2$ arbitrary but fixed.\\
		Define $\hat{\f{v}} \coloneqq (\hat{v}_{1},\hat{v}_2)^T$ by \begin{align*}
			\hat{v}_{1}(\zeta_1,\zeta_2) &\coloneqq \frac{1}{2} \Big(\int_{0}^{\zeta_1} \hat{p}(\tau,\zeta_2) - \int_{0}^{1} \hat{p}(\eta,\zeta_2) d\eta \ d\tau \Big),  \\
			\hat{v}_{2}(\zeta_1,\zeta_2) &\coloneqq \frac{1}{2} \Big(\int_{0}^{\zeta_2} \hat{p}(\zeta_1,\tau) + \int_{0}^{1} \hat{p}(\eta,\tau) d\eta \ d\tau \Big).
		\end{align*}
		Since $\hat{p} $ is the linear combination of product spline functions of the form  \\ $\widehat{B}_{\f{i},\f{p}}(\zeta_1,\zeta_2)=\widehat{B}_{i_1,p-1}(\zeta_1) \cdot \widehat{B}_{i_2,p-1}(\zeta_2) \in S_{p-1,p-1}^{-1,-1}$. And since we have for such a spline
		\begin{align*}
			\int_{0}^{\zeta_1} \widehat{B}_{i_1,p}(\tau) \cdot  \widehat{B}_{i_2,p}(\zeta_2) d\tau &= \int_{0}^{\zeta_1} \widehat{B}_{i_1,p}(\tau) d \tau  \cdot \widehat{B}_{i_2,p}(\zeta_2) \in S_{p,p-1}^{0,-1} ,\\
			\int_{0}^{\zeta_1} \int_{0}^{1} \widehat{B}_{i_1,p}(\eta) \cdot  \widehat{B}_{i_2,p}(\zeta_2) d \eta d\tau &= \int_{0}^{1} \widehat{B}_{i_1,p}(\eta) \ d\eta \cdot \int_{0}^{\zeta_1} \widehat{B}_{i_2,p}(\zeta_2) d \tau,  \\
			&= \zeta_1  \Big( \int_{0}^{1} \widehat{B}_{i_1,p}(\eta) \ d\eta \Big) \cdot \widehat{B}_{i_2,p}(\zeta_2)  \in S_{p,p-1}^{0,-1}, \\
			\int_{0}^{\zeta_2} \widehat{B}_{i_1,p}(\zeta_1) \cdot  \widehat{B}_{i_2,p}(\tau) d \tau&=   \widehat{B}_{i_1,p}(\zeta_1)  \cdot \int_{0}^{\zeta_2}   \widehat{B}_{i_2,p}(\tau) d \tau \in S_{p-1,p}^{-1,0}, \\
			\int_{0}^{\zeta_2} \int_{0}^{1} \widehat{B}_{i_1,p}(\eta) \cdot  \widehat{B}_{i_2,p}(\tau) d \eta d\tau	&=	\int_{0}^{1} \widehat{B}_{i_1,p}(\eta) d \eta \cdot 	\int_{0}^{\zeta_2} \widehat{B}_{i_2,p}(\tau) d \tau \in S_{p-1,p}^{-1,0},
		\end{align*} we obtain $\hat{\f{v}} \in \widehat{\f{V}}_{h,2}^2$. Using the Cauchy-Schwarz inequality one can check that  $\n{\hat{\f{v}}}_{\f{L}^2} \leq 2 \n{\hat{p}}_{L^2}.$ Further, we have
		$$\widehat{\nabla} \cdot \hat{\f{v}}  = \frac{1}{2} \Big(\hat{p}- \int_{0}^{1} \hat{p}(\eta,\zeta_2) d \eta + \hat{p} + \int_{0}^{1} \hat{p}(\eta,\zeta_2) d \eta \Big)= \hat{p}.$$  Consequently, $\n{\hat{\f{v}}}_{\f{H}(\widehat{\d})} \leq \sqrt{5} \n{\hat{p}}_{L^2}$ and it is  
		\begin{align*}
			{ \langle \nabla \cdot \hat{\f{v}}, \hat{p} \rangle } \geq  \frac{1}{\sqrt{5}} \norm{\hat{p}}_{L^2} \  \norm{\hat{\f{v}}}_{\f{H}(\widehat{\d})} .
		\end{align*}
		By the arbitrariness of $\hat{p}$ and the fact that 
		\begin{align*}
			\hat{v}_1(0,\zeta_2) = \hat{v}_1(1,\zeta_2)= 0  \ \ \ \textup{and} \ \ \  \hat{v}_2(\zeta_1,0)=0,
		\end{align*} we get with $\widehat{\f{V}}_{0,h,2}^2 \coloneqq \mathcal{Y}_3({\f{V}}_{0,h,2}^2)$:
		\begin{align*}
			\underset{\hat{\f{w}} \in \widehat{\f{V}}_{0,h,2}^2}{\sup}	\frac{ \langle \nabla \cdot \hat{\f{w}}, \hat{q} \rangle }{ \norm{\hat{\f{w}}}_{\f{H}(\widehat{\d})}  } \geq  \frac{1}{\sqrt{5}} \norm{\hat{q}}_{L^2}, \hspace{0.3cm} \forall \hat{q} \in \widehat{V}_{h,3}^2 .
		\end{align*}
		Next we use the commutativity property \eqref{eq_trans_compa} of the pullbacks $\mathcal{Y}_i$. We note that for $\hat{\f{v}}= \mathcal{Y}_3(\f{v}) \in \widehat{\f{V}}_{0,h,2}^2 $ and $ \hat{p}= \mathcal{Y}_4(p) \in \widehat{V}_{h,3}^2$ we have $$ \n{\f{v}}_{\f{H}(\d)}  \leq C \n{\hat{\f{v}}}_{\f{H}(\widehat{\d})}, \hspace{1cm} C_1 \n{\hat{p}}_{L^2} \geq\n{p}_{L^2} \geq C_2 \n{\hat{p}}_{L^2}, \ $$
		where $C, C_i$  are suitable constants depending only on $\f{F}$ and $\o$. Hence, one obtains for a $q= \mathcal{Y}_4(\hat{q}) \in V_{h,3}^{2}$ arbitrary 
		\begin{align*}
			\underset{{\f{w}} \in {\f{V}}_{0,h,2}^2}{\sup}	\frac{ \langle \nabla \cdot {\f{w}}, {q} \rangle }{ \norm{{\f{w}}}_{\f{H}(\d)}  } \geq  C  \underset{\hat{\f{w}} \in \widehat{\f{V}}_{0,h,2}^2}{\sup}	\frac{ \langle \nabla \cdot \hat{\f{w}}, \hat{q} \rangle }{ \norm{\hat{\f{w}}}_{\f{H}(\widehat{\d})}  }  \geq C \n{\hat{q}}_{L^2} \geq C \n{q}_{L^2}.
		\end{align*}
		And thus  the first inf-sup condition  follows. \\
		Now we face the second inf-sup condition. By assumption there is an edge ${\Gamma}_i$ which is not part of the traction boundary, i.e. $\Gamma_t \neq \partial  \o$. But then the application of Lemma \ref{lemma_B2} yields the wanted estimate if we set $\Gamma_B= \Gamma_t$. \\
		This finishes the proof.
	\end{proof}

	In view of the Lemma above, the mixed formulation with mixed boundary condition is well-posed, at least if $\f{t}_n =  \tilde{\f{\sigma}} \cdot \f{n}$ for some given $\tilde{\f{\sigma}} \in \t{H}(\o,\d)$ and if Assumption  \ref{assumption_boundary} is fulfilled. We have the subsequent result:
	
	\begin{corollary}
		Let the Assumption \ref{assumption_boundary} be valid and let the boundary data  regular enough, i.e. the exact solution exists and the problem can be reduced to the case of $\f{t}_n = \f{0}$ on $\Gamma_t$. Then, the discrete problem with mixed boundary conditions is well-posed, at least for $h$ small enough. 
	\end{corollary}

	\subsubsection{Multi-patch parametrizations}
	
	Up to now we considered parametrizations of the form $\f{F} \colon [0,1]^2 \rightarrow \mathbb{R}^2$ meaning we have a so-called single-patch parametrization. Frequently in applications computational domains have to be considered which can not be parameterized accurately by means of a single patch. To achieve more flexibility one decomposes $\Omega$ into several simpler shapes $\o_i$ and defines a  parametrization for each patch separately. In other words, in the multi-patch framework one has 
	\begin{align*}
		\overline{\o} = \bigcup_{i=1}^{m} \overline{\o}_i, \ \ \  \ \ \  \hspace{0.1cm}  \f{F}_i \colon [0,1]^2 \rightarrow \overline{\o}_i,
	\end{align*} 
with $m \in \mathbb{N}$ denoting the number of patches.
	For reasons of simplification we restrict ourselves to a special case of  multi-patch discretizations, i.e. we make  the next  assumptions.
	
	\begin{assumption}
		\label{assumption_multi-patch}
		 Each patch parametrization $\f{F}_i$ fulfills Assumption \ref{assum_reg_mesh}. Moreover, for $i \neq j$ it is either $\overline{\o_i} \cap \overline{\o_j} = \emptyset$ or the two patches share a boundary edge or a single vertex. Further, we  assume  for each patch $\Omega_i$ the existence of an  edge $S^{(i)} \subset \partial \o_i $ s.t. $S^{(i)} \subset \partial \o$. 
	\end{assumption}
    	 For the discretization we further assume:  
    \begin{assumption}
    	\label{assumption_multi-patch_2}
    	For each patch we have the same underlying polynomial degree $p$, regularity parameter $r=0$ and the same parametric mesh. 
    \end{assumption}
	The multi-patch spaces are then defined by means of the single-patch spaces. For example,  we have the multi-patch Taylor-Hood pair
	\begin{align*}
		\f{V}_h^{TH,M} &\coloneqq \{ \f{v}  \in \f{H}^1(\o) \ | \ \f{v}_{|\o_i}  \in \f{V}_h^{TH,i},  \ \forall i\}, \\
				{V}_h^{TH,M} &\coloneqq \{ {v}  \in {H}^1(\o) \ | \ {v}_{|\o_i}  \in {V}_h^{TH,i},  \ \forall i \},
	\end{align*}
	where the upper index $i$ is used to identify the IGA spaces in the $i$-th patch, i.e. $\f{V}_h^{TH,i} \coloneqq (S_{p,p}^{r,r}  \times S_{p,p}^{r,r} ) \circ \f{F}_i^{-1}, \  $ $ V_{h,3}^{2,i} \coloneqq \widehat{V}_{h,3}^2 \circ \f{F}_i^{-1} $ and so on. The rest of the multi-patch spaces  $$ \t{V}_{h,2}^{w,M} \coloneqq \{ \f{\tau}  \in \t{H}(\o,\d) \ | \ \f{\tau}_{|\o_i}  \in \t{V}_{h,2}^{w,i},  \ \forall i \}, \ {V}_{h,3}^{2,M} \coloneqq \{ {q}  \in {L}^2(\o) \ | \ q_{|\o_i}  \in {V}_{h,3}^{2,i},  \ \forall i \},  \ \ \dots $$ are defined in an analogous manner.

	In this subsection we are again in the context of pure displacement boundary conditions. A combination of mixed boundary conditions with multi-patch parametrizations seems possible, but is not considered here; see Remark \ref{remark_gener}.
	\begin{figure}
	\end{figure}
	To obtain a well-posed method in the multi-patch framework, we have to check again the requirements in Lemma \ref{Lemma_Arnold_15}, where the appearing spaces have to be replaced by the multi-patch pendants. We do not  prove  the multi-patch  inf-sup stability in general, but with the additional restrictions determined by the above  assumptions we succeed. On the one hand we have the result:
	\begin{lemma}
		\label{Lemma_multi-patch_1}
		Let the multi-patch discretization be regular as described in Assumption \ref{assumption_multi-patch} and Assumption \ref{assumption_multi-patch_2}.  \\
		Then there exist mesh size independent constants  $ h_{\textup{max}}>0$, \ $C_S >0$ s.t. 
		\begin{align}
			\underset{q_h \in V_{h}^{TH,M} \textbackslash \{ 0\}}{\inf} \  \underset{\f{w}_h \in \f{V}_{h}^{TH,M} }{\sup}	\frac{ \langle \nabla \cdot \f{w}_h, q_h \rangle }{ \norm{\f{w}_h}_{\boldsymbol{H}^1} \ \norm{q_h}_{L^2} } \geq C_S ,  \ \ \forall h \leq h_{max}.
		\end{align}	
	\end{lemma}
	
	\begin{proof}
		Unfortunately, we can not apply Lemma  \ref{lemma_B2} directly, since the $H^1$-condition at the interfaces has to be taken into account. Therefore, we first consider the auxiliary space $$\tilde{\f{V}}_h \coloneqq  \{ \f{v}  \in \f{L}^2(\o) \ | \ \f{v}_{|\o_i} \in \f{V}_{h}^{TH,i} \cap \f{H}^1_{\Gamma_{D}^i}(\o_i) \},$$ 
		where $\f{V}_{h}^{TH,i}$ is just the space corresponding to the parametrization $\f{F}_i$ and with  $\Gamma_{D}^i \coloneqq \partial\o_i \textbackslash S^{(i)}$.
		Then, the global $H^1$-regularity is implied, i.e.  $\tilde{\f{V}}_h \subset \f{V}_{h}^{TH,M} $.
		Hence, we obtain  by means of  Lemma \ref{lemma_B2} constants  $C_i$ and mesh sizes $h_{\textup{max},i}$ such that
		\begin{align*}
			\underset{\f{w} \in \tilde{\f{V}}_h }{\sup}	\frac{ \langle \nabla \cdot \f{w}, q \rangle_{\o_i} }{ \norm{\f{w}}_{\boldsymbol{H}^1(\o_i)}  } \geq  \n{q}_{L^2(\o_i)} \geq C_i > 0,  \ \ \forall q \in V_{h}^{TH,i} , \ \ h \leq h_{max,i}.
		\end{align*}	
		Let now $ q \in V_{h}^{TH,M} \textbackslash \{0  \}$ arbitrary but fixed and $J \subset \{1, \dots ,m \}$ the set of all indices $j$ with $q_{|\o_j} \neq 0$. For $h \leq \min\{h_{\textup{max},i}  \}$ we obtain:
		\begin{align*}
			\underset{\f{w} \in \f{V}_{h}^{TH,M} }{\sup}	\frac{ \langle \nabla \cdot \f{w}, q \rangle_{\o} }{ \norm{\f{w}}_{\boldsymbol{H}^1(\o)}  } &\geq \sum_{j \in J}   \  \underset{\f{w}_j \in \f{V}_{h}^{TH,j} \cap \f{H}^1_{\Gamma_D^j}(\o_j) }{\sup}	\frac{ \langle \nabla \cdot \f{w}_j, q \rangle_{\o_j} }{ \norm{\f{w}_j}_{\boldsymbol{H}^1(\o_j)} \  }  \\
			&\geq \sum_{j \in J}  C_j \ \n{q}_{L^2(\o_j)} \geq \underset{i}{\textup{min}}\{C_i\}  \n{q}_{L^2(\o)}.
		\end{align*}
		Consequently, the inequality in the assertion is true for $C_S = \underset{i}{\textup{min}}\{C_i\}>0$ and $h_{\max}= \underset{i}{\textup{min}}\{h_{\max,i} \}$.
	\end{proof}

On the other hand one sees:
	
	\begin{lemma}
		\label{Lemma_multi_patch2}
		Under the same conditions as the previous Lemma \ref{Lemma_multi-patch_1}, the first inf-sup condition from Lemma \ref{Lemma_Arnold_15} is satisfied in the multi-patch setting, i.e.  there is a constant $0< C_P$ s.t.
		\begin{equation*}
				\underset{p_h \in V_{h,3}^{2,M} \textbackslash \{ 0\}}{\inf} \  \underset{\f{v}_h \in \f{V}_{h,2}^{2,M} }{\sup}	\frac{ \langle \nabla \cdot \f{v}_h, p_h \rangle }{ \norm{\f{v}_h}_{\boldsymbol{H}(\d)} \ \norm{p_h}_{L^2} } \geq C_P
		\end{equation*}
	\end{lemma}
	\begin{proof}
	We only give the proof idea since the approach is very similar to the steps in proof of Lemma \ref{Lemma_multi-patch_1}.\\
 By assumption,  we have for each patch an edge which is free, i.e. which is not part of an interface. Then the patch-wise application of  Lemma \ref{Lemma_mixed_boubdary_weak_sym} implies  the wanted inf-sup condition. 
	\end{proof}

	With a row-wise \textup{curl} application we further see easily $\textup{curl}(\f{V}_{h}^{TH,M}) \subset \t{V}_{h,2}^{w,M}$.   Thus,	
	using Lemmas \ref{Lemma_multi-patch_1} and  \ref{Lemma_multi_patch2}, we see with Lemma \ref{Lemma_Arnold_15}:
	\begin{corollary}
	Let Assumptions  \ref{assumption_multi-patch} and  \ref{assumption_multi-patch_2} hold. Then the straight-forward  adaption of the discrete mixed system \eqref{dis_weak_form_weak_symmetry_contiuous} to the multi-patch setting leads to a well-posed  problem provided that $h$ is small enough; see Lemma \ref{Lemma_multi-patch_1}. 
	\end{corollary}

	\begin{remark}
		\label{remark_gener}
		Obviously, it makes sense to combine the different generalization steps, meaning we have for example multi-patch parametrizations endowed with mixed boundary conditions. However, we restricted ourselves to special situations; see  Assumption \ref{assumption_boundary} and Assumptions \ref{assumption_multi-patch} and \ref{assumption_multi-patch_2}.  Although a strict proof  is missing, we think that  stability and convergence of the discretization method can be achieved also for more general parametrizations and mixed boundary conditions. 
	\end{remark}

	After we studied the mixed system with weakly imposed symmetry we face the more complicated situation of test spaces with strong symmetry.

	\section{Mixed formulation with strong symmetry}
	\label{section_strong}
	Here we want to show the existence of suitable  spline spaces for the case of planar linear elasticity with  strong symmetry. Later, the actual proofs  are within the scope of pure Dirichlet problems, meaning $\Gamma_D = \partial \o$.   But before we are able to define proper discrete spaces we first look at an auxiliary problem   endowed  with a particular mixed boundary condition.  Besides, in the whole section we consider only single-patch parametrizations.

	\subsection{An auxiliary mixed BC  problem }
	\label{section_dirichlet_Neumann}
	To begin with, we assume for the underlying problem:
	\begin{assumption}[Mixed boundary conditions]
			\label{eq_mixed_BC_strong_sym}
		We have the mixed boundary conditions 
		\begin{align}
			\f{\sigma} \cdot \f{n} = \f{0} \ \ \textup{on} \    \ \Gamma_t= \f{F}(\hat{\Gamma}_1), \ \ \f{u} = \f{u}_D \in \f{H}^{1/2}(\partial \o) \ \ \textup{on} \   \Gamma_D= \partial\Omega \textbackslash \Gamma_t.
		\end{align}
		For the labeling of the boundary edges we refer to Fig. \ref{Fig:boundaries} and we require $\f{F}^{-1} \in C^3(\overline{\o}), \ \f{F} \in C^3(\overline{\widehat{\o}})$.
	\end{assumption}
	In other words we consider the weak formulation:
	\emph{Find} $(\f{\sigma},\f{u})  \in  \t{H}_{\Gamma_1}(\o,\d,\mathbb{S}) \times \f{L}^2(\Omega) $ such that	
	\begin{alignat}{4}
		\label{weak_form_strong_symmetry_contiuous_BC}
		&\langle \f{A} \boldsymbol{\sigma} , \boldsymbol{\tau} \rangle \ + \ &&\langle \f{u}, \nabla \cdot \boldsymbol{\tau}  \rangle \  \color{black}  &&= \langle \boldsymbol{\tau} \cdot \f{n}; \f{u}_D \rangle_{\Gamma}, \ \hspace{0.4cm} &&\forall \boldsymbol{\tau} \in \boldsymbol{\mathcal{{H}}}_{\Gamma_1}(\Omega,\textup{div},\s), \nonumber\\
		&\langle \nabla \cdot \boldsymbol{\sigma} , \f{v} \rangle &&  &&= \langle \f{f},\f{v}\rangle, \ \hspace{0.4cm} && \forall \f{v} \in \boldsymbol{L}^2(\Omega),
	\end{alignat}
	with 
	\begin{align}
		\boldsymbol{\mathcal{{H}}}_{\Gamma_1}(\Omega,\textup{div},\s) \coloneqq    \overline{\{  \f{\tau} \in \boldsymbol{\mathcal{{H}}}(\Omega,\textup{div},\s) \cap C^1(\overline{\o}) \ | \   \f{\tau} \cdot \f{n}=\f{0} \ \textup{on} \ \Gamma_1  \}}^{\t{H}(\d)},
	\end{align}
	where  on right-hand side of the last line we have  the closure of the space in brackets w.r.t. the norm $\n{\cdot}_{\t{H}(\d)}$. For the sake of clarity  we write $\t{H}(\widehat{\o},\widehat{\d},\mathbb{S})$ and $\t{H}_{\hat{\Gamma}_1}(\widehat{\o},\widehat{\d},\mathbb{S})$ for the corresponding spaces on the parametric domain. For more information about such spaces with boundary conditions we refer to the articles \cite{Pauly_2022_2,Pauly_2022_1}. In  latter references the explanations are for the three-dimensional case, but they can be adapted easily to  planar domains.  
	To simplify the expressions we will use the \emph{Einstein summation convention}, where the different  indices take the values $1$ and $2$. Further, we use several times the dot "$\cdot$" to indicate a matrix multiplication or scalar multiplication. The usage of the dot sign should  clarify the appearing longer terms.
	
	Basic idea of our approach is to work in the parametric domain and to use in some sense proper pullback and push-forward operations to handle the situation in $\o$. Therefore,
	we  define the next  two transformations 
	\begin{align}
		\label{eq_def_gamma_2_1}
		\mathcal{Y}_{2,\Gamma_1}^s &\colon \t{L}^2(\Omega,\mathbb{S}) \rightarrow \t{L}^2(\widehat{\Omega},\mathbb{S}),   \\ &\f{S} \mapsto \textup{det}(\f{J})^2 \f{J}^{-1} (\f{S} \circ \f{F}) \f{J}^{-T} + \widehat{\textup{Airy}}({F}_n) \cdot \Big( \int_{0}^{\zeta_1}   (-1)^{n+1}  \   \textup{det}(\f{J}) \ J_{2l}^{-1}  \ (S_{\sigma(n)l} \circ \f{F})  \ d1\Big), \nonumber
	\end{align}
	with $\sigma(n) \coloneqq 3-n, \ \f{F}=(F_1,F_2)^T$ and 
	$$\widehat{\textup{Airy}}(\hat{\phi}) \coloneqq 
	\begin{pmatrix}
		\hat{\partial}_{22} \hat{\phi}  & -\hat{\partial}_{12} \hat{\phi} \\ -\hat{\partial}_{21} \hat{\phi} & \hat{\partial}_{11} \hat{\phi}
	\end{pmatrix}, \ \ \ \ \textup{Airy}({\phi}) \coloneqq
	\begin{pmatrix}
		{\partial}_{22} {\phi}  & -{\partial}_{12} {\phi} \\ -{\partial}_{21} {\phi} & {\partial}_{11} {\phi}
	\end{pmatrix}, \ \ \textup{respectively}. $$
	Above "$ \ d1$" means that we integrate w.r.t. to the first parametric coordinate, i.e. for some $g=g(\zeta_1,\zeta_2)$ we have $\int_{0}^{\zeta_1} g  \ d1 \coloneqq \int_{0}^{\zeta_1}g(\tau,\zeta_2) d\tau$.
	Besides, $\h_i, \ \partial_i$  denote the derivatives w.r.t. to the $i$-th parametric coordinate, $i$-th physical coordinate respectively, and we write $\h_{ij}, \ \partial_{ij}$ for $\h_i \h_j$ and $\partial_i \partial_j$.  We remark that $J_{ij}^{-1}= J_{ij}^{-1}(\f{x})$ stands for the entry  $(i,j)$ of the Jacobian $\f{J}^{-1}$ of $\f{F}^{-1}$.  
	With the relation $\f{F}(\zeta_1,\zeta_2)=\f{x}$ we interpret the integrand functions in the whole section as mappings defined on $\widehat{\o}$. For example $\int_{0}^{\zeta_1} J_{ij}^{-1}  \ d1 \coloneqq \int_{0}^{\zeta_1} J_{ij}^{-1} \circ \f{F}(\tau,\zeta_2) d\tau$.
	With the relation for the adjugate matrix $\textup{adj}(\f{J}) = \textup{det}(\f{J}) (\f{J}^{-1}\circ \f{F})$ and  using $\tilde{\f{J}} \coloneqq \textup{adj}(\f{J})$ we can write
	\begin{equation}
		\label{eq_def_gamma_1_E_2}
		\mathcal{Y}_{2,\Gamma_1}^s (\f{S}) =  \tilde{\f{J}} (\f{S} \circ \f{F}) \tilde{\f{J}}^{T} + \widehat{\textup{Airy}}({F}_n) \cdot \Big( \int_{0}^{\zeta_1}   (-1)^{n+1}  \  \tilde{{J}}_{2l}  \ (S_{\sigma(n)l} \circ \f{F})  \ d1 \Big).
	\end{equation} 
	One notes the simple structure of $\J$, namely $$\J\coloneqq \textup{adj}(\f{J}) = \begin{bmatrix}
		J_{22} & -J_{12} \\
		-J_{21} & J_{11}
	\end{bmatrix}.$$
	Secondly, we introduce the mapping 
	\begin{equation}
		\mathcal{Y}_3^s \colon \f{L}^2(\Omega) \rightarrow \f{L}^2(\widehat{\Omega}) \ , \ \f{v} \mapsto \textup{det}(\f{J}) \ \tilde{\f{J}} \cdot (\f{v}\circ \f{F}) + \hat{\partial}_1 \tilde{\f{J}} \cdot \int_{0}^{\zeta_1}    \textup{det}(\f{J}) \ (\f{v}\circ \f{F})    \ d1,
	\end{equation}
	where $\hat{\partial}_1 \tilde{\f{J}}$ means that each entry of $ \tilde{\f{J}}$ is differentiated w.r.t. the first parametric coordinate. 
	Later we will use the equivalent formulation  \begin{equation}
		\label{eq:gamma_2_E_equivalent_form}
		\mathcal{Y}_3^s(\f{v}) =  \h_1 \Big[  \J \cdot \int_{0}^{\zeta_1}    \textup{det}(\f{J}) \ (\f{v}\circ \f{F})    \ d1  \Big].
	\end{equation}
	
	Next we face  several lemmas involving the $\mathcal{Y}_{2,\Gamma_1}^s$ and $\mathcal{Y}_{3}^s$. Basically, these auxiliary results are needed to justify the appropriateness of the mentioned transformations for the discretization.

	We start with the the $L^2$-boundedness.

	\begin{lemma}[Boundedness of  the transformations]
		\label{Lemma_strong_sym_L2}
		The mappings $\mathcal{Y}_{2,\Gamma_1}^s$ and $\mathcal{Y}_3^s$ are bounded with respect to the $L^2$-norm.
	\end{lemma}
	\begin{proof}
		Since we assumed that $\f{F} \in C^3(\overline{\widehat{\o}})$ with $\f{F}^{-1} \in C^3(\overline{\widehat{\o}})$ and in view of the Cauchy-Schwarz inequality which yields  $$ |\int_{0}^{\zeta_1}  \hat{\phi}(\tau,\zeta_2) d\tau|^2 \leq \int_{0}^{\zeta_1}  |\hat{\phi}(\tau,\zeta_2)|^2 d \tau \cdot \int_{0}^{\zeta_1}  1 d\tau \leq \int_{0}^{1}  |\hat{\phi}(\tau,\zeta_2)|^2 d\tau  \ \ \textup{for} \  \ \hat{\phi} \in L^2(\widehat{\o}),$$
		one  sees easily the boundedness w.r.t. the $L^2$-norm.  
	\end{proof}

	\begin{lemma}[Divergence compatibility]
		\label{Lemmma_strong_sym_div}
		For $\f{S} \in \t{H}_{\Gamma_1}(\o,\d,\mathbb{S})$ we have the  relation
		$$\widehat{\nabla} \cdot \mathcal{Y}^{s}_{2,\Gamma_1}(\f{S}) = \mathcal{Y}^{s}_{3}(\nabla \cdot \f{S}).$$  
		In particular, with Lemma \ref{Lemma_strong_sym_L2} we get indeed $\mathcal{Y}^{s}_{2,\Gamma_1}\big( \t{H}_{\Gamma_1}(\o,\d,\mathbb{S}) \big) \subset \t{H}(\widehat{\o},\widehat{\d},\mathbb{S})  $.
	\end{lemma}
	
	\begin{proof}
		For the sake of clarity we moved the proof of the lemma to the appendix; see Section  \ref{Proof_1_appendix}. 
	\end{proof}

	\begin{lemma}[Invertibility of $\mathcal{Y}_{2,\Gamma_1}^s$]
		\label{Lemma_invertebility_Gamma_2_1}
		The mapping $\mathcal{Y}_{2,\Gamma_1}^s$ is invertible with \begin{equation}
			\label{eq_inv_gamma_1}
			(\mathcal{Y}_{2,\Gamma_1}^s)^{-1} (\tilde{\f{S}}) =  \tilde{\f{J}}^{-1} (\tilde{\f{S}} \circ \f{F}^{-1}) \tilde{\f{J}}^{-T} + \textup{Airy}(F_k^{-1}) \cdot \Big(\int_{0}^{\zeta_1}   (-1)^{k+1}  \    \ \tilde{S}_{2\sigma(k)}  \ d1 \Big) \circ \f{F}^{-1}.
		\end{equation}
	\end{lemma}
	\begin{proof}
		We refer to Section \ref{Lemma_proof_appendix_2} in the appendix for  the detailed proof steps. Since the proof is technical we omit it here.
	\end{proof}

	\begin{lemma}[Invertibility of $\mathcal{Y}_3^s$]
		\label{Lemma_invertibility_gamma_3}
		The mapping $\mathcal{Y}_3^s$ is invertible, where
		\begin{equation}
			\label{inv_gamma_3}
			(\mathcal{Y}_3^s)^{-1} (\tilde{\f{v}})= \textup{det}(\f{J}^{-1}) \ \tilde{\f{J}}^{-1} \cdot (\tilde{\f{v}} \circ \f{F}^{-1}) +  \Big( \frac{\hat{\partial}_1  \big[ \tilde{\f{J}}^{-1} \circ \f{F}\big]}{\textup{det}(\f{J})} \cdot \int_{0}^{\zeta_1}    \tilde{\f{v}}    \ d1 \Big) \circ \f{F}^{-1}.
		\end{equation}
	\end{lemma}
	
	\begin{proof}
		In this proof we write  $\tilde{\mathcal{Y}} = \tilde{\mathcal{Y}}(\tilde{\f{v}})$ for the mapping determined  by the right-hand side of \eqref{inv_gamma_3}. Further, we use the notation  $\hat{\f{v}} \coloneqq \f{v} \circ \f{F}$. To begin with, we look at $\tilde{\mathcal{Y}} \circ {\mathcal{Y}}_3^s$.
		\begin{align*}
			\tilde{\mathcal{Y}} \circ {\mathcal{Y}}_3^s(\f{v}) &= \f{v} + \Big(\underbrace{(\tilde{\f{J}}^{-1} \circ \f{F}) \cdot  \frac{\hat{\partial}_1 \tilde{\f{J}}}{\textup{det}(\f{J})} \cdot \int_{0}^{\zeta_1}   \tilde{g}    \ d1}_{\eqqcolon(L3,1)} \Big) \circ \f{F}^{-1} \\
			& +  \Big(\underbrace{\frac{\hat{\partial}_1  \big[ \tilde{\f{J}}^{-1} \circ \f{F}\big]}{\textup{det}(\f{J})} \cdot \int_{0}^{\zeta_1}   {\mathcal{Y}}_3^s(\f{v})   \ d1 }_{\eqqcolon(L3,2)}\Big) \circ \f{F}^{-1},
		\end{align*}
		where we set $\tilde{g} \coloneqq    \textup{det}(\f{J}) \ (\f{v}\circ \f{F})  $. Using the relations
		\begin{align*}
			\big(  {\mathcal{Y}}_3^s(\f{v}) \big)_j &= \textup{det}(\f{J}) \cdot \j_{jk} \cdot \hat{v}_k  + \hat{\partial}_1\j_{jl} \cdot  \int_{0}^{\zeta_1} \tilde{g}_l  \ d1,\\
			0&=\h_1\big[ (\j_{mk}^{-1} \circ \f{F}) \cdot \j_{kl}\big] = \h_1[\j_{mk}^{-1} \circ \f{F}] \cdot \j_{kl} + [\j_{mk}^{-1} \circ \f{F}] \cdot \h_1\j_{kl},
		\end{align*}
		we obtain for the $m$-th entry of $	(L3,1)+(L3,2)$:
		\begin{align*}
			\big[(L3,1)+(L3,2)\big]_m &= \frac{ \j_{mj}^{-1}\circ \f{F}}{\textup{det}(\f{J})} \cdot  \h_1\j_{jl} \cdot  \int_{0}^{\zeta_1} \tilde{g}_l  \ d1+ \frac{\h_1 [\j_{mj}^{-1}\circ \f{F}] }  {\textup{det}(\f{J})} \cdot  \int_{0}^{\zeta_1}   ({\mathcal{Y}}_3^s(\f{v}))_j   \ d1 \\
			&= -\frac{ \h_1[\j_{mj}^{-1} \circ \f{F}]}{\textup{det}(\f{J})} \cdot  \j_{jl} \cdot  \int_{0}^{\zeta_1} \tilde{g}_l  \ d1 \\  & \hspace{0.5cm} + \frac{\h_1 [\j_{mj}^{-1} \circ \f{F}] }{\textup{det}(\f{J})} \cdot     \int_{0}^{\zeta_1} \textup{det}(\f{J}) \cdot \j_{jk} \cdot \hat{v}_k    + \hat{\partial}_1\j_{jl} \cdot  \tilde{\tilde{g}} \ d1, \\
			& \ \ \ \textup{with} \  \tilde{\tilde{g}} \coloneqq \int_{0}^{\zeta_1} \tilde{g}_l \ d1. \ \textup{This implies } \\ 
			\big[(L3,1)+(L3,2)\big]_m	&= -\frac{ \h_1[\j_{mj}^{-1}\circ \f{F}]}{\textup{det}(\f{J})} \cdot  \j_{jl} \cdot  \int_{0}^{\zeta_1} \tilde{g}_l  \ d1 +\frac{ \h_1[\j_{mj}^{-1}\circ \f{F}]}{\textup{det}(\f{J})} \cdot  \j_{jl} \cdot  \int_{0}^{\zeta_1} \tilde{g}_l  \ d1 \\
			& \hspace{0.5cm} +\frac{ \h_1[\j_{mj}^{-1}\circ \f{F}]}{\textup{det}(\f{J})} \cdot   \int_{0}^{\zeta_1}  \textup{det}(\f{J}) \cdot \j_{jk} \cdot \hat{v}_k - \j_{jl} \cdot \tilde{g}_l  \ d1 \\
			&= 0 + \frac{ \h_1[\j_{mj}^{-1}\circ \f{F}]}{\textup{det}(\f{J})} \cdot   \int_{0}^{\zeta_1}  \underbrace{\textup{det}(\f{J}) \cdot \j_{jk} \cdot \hat{v}_k - \j_{jl} \cdot \textup{det}(\f{J}) \cdot  \hat{v}_l}_{=0}   \ d1 =0.
		\end{align*}
		Thus we showed that $\tilde{\mathcal{Y}}$ is the left inverse of $\mathcal{Y}_3^s$. In particular we showed  injectivity. Hence it remains   to show the surjectivity of $\mathcal{Y}_3^{s}$. \\
		Therefore consider
		\begin{align*}
				  {\mathcal{Y}}_3^s \circ 	\tilde{\mathcal{Y}}(\tilde{\f{v}}) &=  \tilde{\f{v}} + \Big(\underbrace{ \h_1\tilde{\f{J}}  \cdot   \int_{0}^{\zeta_1}  \textup{det}(\f{J}) \cdot   \big( \tilde{\mathcal{Y}}(\f{v}) \circ \f{F} \big)    \ d1}_{\eqqcolon(L3,3)} \Big)  \\
			& +  \ \ \ \underbrace{ \textup{det}(\f{J}) \cdot \J \cdot \frac{\hat{\partial}_1  \big[ \tilde{\f{J}}^{-1} \circ \f{F}\big]}{\textup{det}(\f{J})} \cdot \int_{0}^{\zeta_1}    \tilde{\f{v}}   \ d1 }_{\eqqcolon(L3,4)}.
		\end{align*}
	We can simplify 
	\begin{align*}
		[(L3,4)]_m &= -\h_1\j_{ml} \cdot   [\j_{lk}^{-1} \circ \f{F}] \cdot \int_{0}^{\zeta_1}    \tilde{{v}}_k   \ d1, \ \ \textup{since} \ \  \h_1\j_{ml} \cdot   [\j_{lk}^{-1} \circ \f{F}] = - \j_{ml} \cdot   \h_1[\j_{lk}^{-1} \circ \f{F}], \\
			\textup{and expl}&\textup{oiting the product rule we obtain }\\
		[(L3,3)]_m &= \h_1\j_{ml} \cdot   \int_{0}^{\zeta_1} \textup{det}(\f{J}) \cdot   \big( \tilde{\mathcal{Y}}(\f{v}) \circ \f{F} \big)_l    \ d1  \\ 
		& =  \h_1\j_{ml} \cdot \int_{0}^{\zeta_1}   {[\j^{-1}_{lk} \circ \f{F}] \cdot \tilde{v}_k}   \ d1 -  \h_1\j_{ml} \cdot \int_{0}^{\zeta_1}   [\j^{-1}_{lk} \circ \f{F}] \cdot \tilde{v}_k    \ d1 \\ & \ \ \ +  \h_1\j_{ml} \cdot [\j_{lk}^{-1} \circ \f{F}] \cdot   \int_{0}^{\zeta_1}    \tilde{{v}}_k   \ d1 \\
		& = - [(L3,4)].
	\end{align*}
	 Consequently, $\mathcal{Y}_3^s$ is bijective with inverse mapping $\tilde{\mathcal{Y}}$.
	\end{proof}

	\begin{lemma}[Boundary condition compatibility]
		\label{Lemma_compatibility_BC_strong_sym}
		$\mathcal{Y}_{2,\Gamma_1}^s$ preserves  the traction boundary condition for our choice in Assumption \ref{eq_mixed_BC_strong_sym}. More precisely, we have
		$$\mathcal{Y}_{2,\Gamma_1}^s(\t{H}_{\Gamma_1}(\o,\d,\mathbb{S})) = \t{H}_{\hat{\Gamma}_1}(\widehat{\o},\widehat{\d},\mathbb{S}).$$
	\end{lemma}
	\begin{proof}
		Let us assume  $\f{S} \in \t{H}_{\Gamma_1}(\o,\d,\mathbb{S}) \cap C^1(\overline{\o})$. Then we obtain  $\mathcal{Y}_{2,\Gamma_1}^s(\f{S})  \in \t{H}(\widehat{\o},\widehat{\d},\mathbb{S}) \cap C^1(\overline{\widehat{\o}})$.  Writing $\hat{\f{n}}_1$ for the outer unit normal to the edge $\hat{\Gamma}_1$,  it is due to  the divergence preserving  transform 
		 $\f{v} \mapsto \textup{det}(\f{J}) \f{J}^{-1} (\f{v}\circ \f{F})$: \
		$  (\textup{det}(\f{J}) \ \f{J}^{-1} \cdot (\f{S}_i \circ \f{F})) \cdot \hat{\f{n}}_1= -\j_{1j} \cdot \hat{S}_{ij}=0 $ on $\hat{\Gamma}_1$. Note $\hat{\f{S}} \coloneqq \f{S} \circ \f{F}$. Consequently, 
		$$ \big( \mathcal{Y}_{2,\Gamma_1}^s(\f{S})  \big)_m \cdot \hat{\f{n}}_1= -\j_{1j} \cdot \hat{S}_{ij} \cdot \j_{mi}=0 \cdot \j_{mi}=0 \ \  \textup{on} \ \hat{\Gamma}_1.$$
		One observes that the integral terms in the definition of $\mathcal{Y}_{2,\Gamma_1}^s$ (see \ref{eq_def_gamma_2_1}) vanish on $\hat{\Gamma}_1$. Hence $\mathcal{Y}_{2,\Gamma_1}^s(\f{S})  \in \t{H}_{\hat{\Gamma}_1}(\widehat{\o},\widehat{\d},\mathbb{S})$. Due to Lemma \ref{Lemmma_strong_sym_div}, Lemma \ref{Lemma_strong_sym_L2} and a density argument we have indeed   $\mathcal{Y}_{2,\Gamma_1}^s(\t{H}_{\Gamma_1}(\o,\d,\mathbb{S})) \subset \t{H}_{\hat{\Gamma}_1}(\widehat{\o},\widehat{\d},\mathbb{S})$.\\
		On the other hand let us assume $\tS \in \t{H}_{\hat{\Gamma}_1}(\widehat{\o},\widehat{\d},\mathbb{S}) \cap C^1(\overline{\widehat{\o}})$. One shows
		\begin{align*}
			\big( (\mathcal{Y}_{2,\Gamma_1}^s)^{-1}(\tilde{\f{S}})  \big)_m \cdot {\f{n}}_1 &= 	 \big[ \j_{mi}^{-1} \cdot \big(\J^{-1} \cdot (\tS^i \circ \f{F}^{-1}) \big)^T\big] \cdot {\f{n}}_1 \\
			&= \j_{mi}^{-1}  \cdot \big[ 	 \big(\J^{-1} \cdot (\tS^i \circ \f{F}^{-1}) \big)^T \cdot {\f{n}}_1	\big]	= \j_{mi}^{-1} \cdot 0 =0 \ \ \textup{on} \ \ \Gamma_1.
		\end{align*}
		In the last lines we wrote ${\f{n}}_1$ for the outer unit normal to the edge $\Gamma_1= \f{F}(\hat{\Gamma}_1)$ and we use the fact $ \big(\J^{-1} \cdot (\tS^i \circ \f{F}^{-1}) \big)^T \cdot {\f{n}}_1	=0$ which is a consequence of the transformation rule for normal vectors. 
		 This implies  $(\mathcal{Y}_{2,\Gamma_1}^s)^{-1}(\tilde{\f{S}}) \in \t{H}_{\Gamma_1}(\o,\d,\mathbb{S}) \cap C^1(\overline{\o})$.  The general case of $\tS \in \t{H}_{\hat{\Gamma}_1}(\widehat{\o},\widehat{\d},\mathbb{S})$ follows by a density argument.
	\end{proof}
	
	If we use the Lemma \ref{Lemma_compatibility_BC_strong_sym}  we see that we have the commuting diagram
	\begin{figure}[h!]
		\centering
		\begin{tikzpicture}
			\node at (0.8,1.4) {$\f{L}^2(\o)$};

			\draw[->] (-0.8,1.4) to (-0.72+0.8,1.4);		
			
			\node at (-0.4,1.65) {$\nabla \cdot$};			
			\node at (-2.25,1.4) {$\t{H}_{\Gamma_1}({\o},\d,\mathbb{S})$};			
			
			\draw[->] (-2.25,1.1) -- (-2.25,0.2);			
			\draw[->] (0.95,1.1) -- (0.95,0.2);	
			\node at (0.8,-0.1) {$\f{L}^2(\widehat{\o})$};			
			
			\draw[->] (-0.8,-0.1) to (-0.72+0.8,-0.1);
			
			\node at (-0.4,0.15) {$\widehat{\nabla} \cdot $};			
			\node at (-2.25,-0.1) {$\t{H}_{\hat{\Gamma}_1}(\widehat{\o},\widehat{\d},\mathbb{S})$};			
			\node[left] at (1,0.7) { \small $\mathcal{Y}_{3}^{s}$};	
			\node[left] at (-2.2,0.7) { \small $\mathcal{Y}_{2,\Gamma_1}^{s}$};
		\end{tikzpicture}
		\caption{The diagram commutes.}
		\label{Fig:com_diagram_5}
	\end{figure}

	Next we look at the classical planar Dirichlet problem and use the results derived in Section \ref{section_dirichlet_Neumann}.

	\begin{remark}
		The results of this subsection imply the well-posedness of the weak form \eqref{weak_form_strong_symmetry_contiuous_BC}.  Namely, due to the commutativity of the diagram in Fig. \ref{Fig:com_diagram_5}, Lemma \eqref{Lemmma_strong_sym_div} -  \eqref{Lemma_invertibility_gamma_3}  it is enough to prove the well-posedness for $\f{F} = \textup{id}$. In fact,  the latter can be shown   with an ansatz similar to the one in the proof of Lemma \ref{Lemma_exact_dis}. To be more precise, in view of the choice \eqref{eq_construction_1}-\eqref{eq_construction_3} for a general $ \tilde{\f{v}} \in  \f{L}^2(\widehat{\o})$ it is easy to check that for all such $\tilde{\f{v}}$ there is a $\tilde{\f{\tau}} \in \t{H}_{\hat{\Gamma}_1}(\widehat{\o},\widehat{\d},\mathbb{S})$ s.t. $\widehat{\nabla} \cdot \tilde{\f{\tau}} = \tilde{\f{v}} $ and $\n{\tilde{\f{\tau}}}_{\t{L}^2} \leq \n{\tilde{\f{v}}}_{\f{L}^2}$. This leads to  the classical  Brezzi stability conditions for the saddle-point problem  $\eqref{weak_form_strong_symmetry_contiuous_BC}$. Since we want to study the case with pure displacement boundary conditions, we omit here further details. 
	\end{remark}
	
	\subsection{The Dirichlet problem}
	\label{section_dirichlet_problem}
	In this section we suppose $\Gamma_D = \partial \o$.
	Further, for the moment we assume   always $\f{S} \in C^1(\overline{\o},\s) \coloneqq  \t{L}^2(\o,\s)  \cap C^1(\overline{\o}) $ and
	$\tilde{\f{S}} \in C^1(\overline{\widehat{\o}},\s) \coloneqq \t{L}^2(\widehat{\o},\s) \cap C^1(\overline{\widehat{\o}})$. 
	 The  smoothness assumption is made to guarantee  the well-definedness of  subsequent  calculations. Analyzing the proof of Lemma \ref{Lemmma_strong_sym_div} we see that there is only one step where the underlying boundary conditions are needed, namely the rearrangement from line  \eqref{eq_Lemma_BC_line1} to \eqref{eq_Lemma_BC_line2}. Hence, if we do not require the boundary conditions in Assumption \ref{eq_mixed_BC_strong_sym}, we get an additional term. More precisely, with $\hat{\S} \coloneqq \S \circ \f{F}$ it is
	\begin{align*}
		\widehat{\nabla} \cdot \mathcal{Y}^{s}_{2,\Gamma_1}(\f{S}) = \mathcal{Y}^{s}_{3}(\nabla \cdot \f{S}) + \begin{bmatrix}
			\h_1 \j_{1i} \cdot \j_{1l}(0, \cdot) \cdot \hat{S}_{il}(0,\cdot) \\
			\h_1 \j_{2i} \cdot \j_{1l}(0, \cdot) \cdot \hat{S}_{il}(0,\cdot)
		\end{bmatrix}.
	\end{align*} 
	Above and in the following we write $g(0,\cdot)$ to set the first parametric coordinate of a mapping $g$ to zero, i.e. for example $\int_{0}^{\zeta_1} g(0,\cdot) d1 \coloneqq \int_{0}^{\zeta_1} g(0,\zeta_2) d\tau$. 
	Thus we define 
	\begin{align}
		\label{eq_def_Gamma_2s}
		\mathcal{Y}^{s}_{2}(\f{S})  &\coloneqq \mathcal{Y}^{s}_{2, \Gamma_1}(\f{S}) - \mathcal{Y}^{s}_{2,A}(\f{S}), \ \  \forall \f{S} \in C^1(\overline{\o},\s),  \ \textup{where} \\
		\mathcal{Y}^{s}_{2,A}(\f{S})  &\coloneqq \begin{bmatrix}
			\int_{0}^{\zeta_1} \h_1 \j_{1i} \cdot \j_{1l}(0, \cdot) \cdot \hat{S}_{il}(0,\cdot)  \ d1 & 0 \\
			0 &  \int_{0}^{\zeta_2} \h_1 \j_{2i} \cdot \j_{1l}(0, \cdot) \cdot \hat{S}_{il}(0,\cdot) \ d2 
		\end{bmatrix}.\nonumber
	\end{align}
	
	Obviously, we have the next Lemma.
	\begin{lemma}[Divergence compatibility]
		\label{Lemma_compatibility_strong}
		It holds $\widehat{\nabla} \cdot \mathcal{Y}_{2}^{s}(\f{S}) = \mathcal{Y}^{s}_{3}(\nabla \cdot \f{S}), \ \ \forall \S  \in C^1(\overline{\o},\s)$.
	\end{lemma}	
	Another point we want to mention is the fact that $\mathcal{Y}_{2}^{s}(\f{S}) \in C^1(\overline{\widehat{\o}},\s)$ if $\f{S} \in C^1(\overline{\o},\s)$. This follows easily by the defintion of the mappings.	
	Furthermore we set 
	\begin{align*}
		\tilde{\mathcal{Y}}_2^{s}(\tS) &\coloneqq (\mathcal{Y}_{2,\Gamma_1}^s)^{-1} (\tS) - \tilde{\mathcal{Y}}_{2,A}^{s}(\tS), \ \tS \in C^1(\overline{\widehat{\o}},\s),
	\end{align*}
	with 
	\begin{align*}
		\tilde{\mathcal{Y}}_{2,A}^{s}(\tS) &\coloneqq (\mathcal{Y}^s_{2,\Gamma_1})^{-1}   \Big( \begin{bmatrix}
			\int^{\zeta_1}_{0}  \big[\mathcal{Y}^s_{3}(T(\tS)) \big]_1   \ d1 & 0 \\
			0 & \int^{\zeta_2}_{0}  \big[\mathcal{Y}^s_{3}(T(\tS)) \big]_2  \ d2
		\end{bmatrix}\Big),\\
		T(\tS) &\coloneqq 
		\begin{bmatrix}
			\partial_k \tilde{J}^{-1}_{1i} \cdot \j^{-1}_{k1} \cdot (\tilde{S}_{i1}(0,\cdot) \circ \f{F}^{-1}) \\
			\partial_k \tilde{J}^{-1}_{2i} \cdot \j^{-1}_{k1} \cdot (\tilde{S}_{i1}(0,\cdot) \circ \f{F}^{-1})
		\end{bmatrix}.
	\end{align*}
	Again one can verify the well-definedness of $\tilde{\mathcal{Y}}_{2}^{s}(\tS)$ and that $\tilde{\mathcal{Y}}_{2}^{s}(\tS) \in C^1(\overline{\o},\s)$ in case of $\tS \in C^1(\overline{\widehat{\o}},\s)$. Next we show the invertibility of the introduced map $\mathcal{Y}_2^{s}$. 
	
	\begin{lemma}[Invertibility of $\mathcal{Y}_2^{s}$]
		\label{Lemma_invertebility_Gamma_2_2}
		The mapping $\mathcal{Y}_2^{s} \colon C^1(\overline{\o},\s) \rightarrow C^1(\overline{\widehat{\o}},\s)$ determined by \eqref{eq_def_Gamma_2s} is invertible with 
		$$(\mathcal{Y}_2^{s})^{-1} = \tilde{\mathcal{Y}}_2^{s}.$$
	\end{lemma}
	
	\begin{proof}
		A detailed proof for the assertion can be found in the appendix; see Section \ref{Lemma_proof_appendix_3}.
	\end{proof}

	\begin{lemma}[Boundedness w.r.t. $H^1$-norm]
		\label{Lemma_Boundedness_strong}
		The mappings ${\mathcal{Y}}_2^{s} \colon C^1(\overline{\o},\s) \rightarrow C^1(\overline{\widehat{\o}},\s)$ and $\tilde{\mathcal{Y}}_2^{s} \colon C^1(\overline{\widehat{\o}},\s) \rightarrow C^1(\overline{{\o}},\s)$ are bounded in the sense  \begin{center}$\n{{\mathcal{Y}}_2^{s}(\S) }_{\t{H}(\widehat{\d})} \leq C \n{\S}_{\t{H}^1}$  and $\n{\tilde{\mathcal{Y}}_2^{s}(\tS) }_{\t{H}({\d})} \leq C \n{\tS}_{\t{H}^1}$.  \end{center}
		In particular, ${\mathcal{Y}}_2^{s}$ and $\tilde{\mathcal{Y}}_2^{s}$ have unique bounded extensions on $\t{H}^1(\o) \cap \t{L}^2(\o,\s)$, $\t{H}^1(\widehat{\o}) \cap \t{L}^2(\widehat{\o},\s)$ respectively. Hence, it is justified to write also  \begin{align*}
			{\mathcal{Y}}_2^{s} &\colon \t{H}^1(\o) \cap \t{L}^2(\o,\s) \rightarrow  \t{H}(\widehat{\o},\widehat{\d},\mathbb{S}), \\
			\tilde{\mathcal{Y}}_2^{s} &\colon \t{H}^1(\widehat{\o}) \cap \t{L}^2(\widehat{\o},\s) \rightarrow  \t{H}({\o},{\d},\mathbb{S}).
		\end{align*}
	\end{lemma}
	
	\begin{proof}
		We show the assertion only for ${\mathcal{Y}}_2^{s}$ since the second part can be seen with a similar reasoning.\\
		Due to Lemma \ref{Lemma_strong_sym_L2} and Lemma \ref{Lemma_compatibility_strong} it is enough to check  $ \n{{\mathcal{Y}}_{2,A}^s(\S)}_{\t{L}^2} \leq C \n{\S}_{\t{H}^1}$. 
		Since we assume $\f{F}, \f{F}^{-1}$ to be three times continuously differentiable we have $|\h_1\j_{mi}| \in C^0(\overline{\widehat{\o}})$ and in view of the Cauchy-Schwarz inequality it is for some constants $C<\infty$, i.e. $C$ may vary at different occurrences,
		
		\begin{align*}
			\Big|\int_{0}^{\zeta_1} \h_1 \j_{1i} \cdot \j_{1l}(0, \cdot) \cdot \hat{S}_{il}(0,\cdot)  \ d1\Big|^2 \leq C \int_{0}^{\zeta_1} \big|    \hat{\f{S}}(0,\cdot) \big|^2   \ d1, \ \ \ \hat{\S}\coloneqq \S \circ \f{F} .
		\end{align*}
		Then applying a standard trace estimate one sees 
		\begin{align*}
			\int_{0}^{1}	\Big|\int_{0}^{\zeta_1} \h_1 \j_{1i} \cdot \j_{1l}(0, \cdot) \cdot \hat{S}_{il}(0,\cdot)  \ d1\Big|^2  d2&\leq C   \int_{0}^{1} \big|    \hat{\f{S}}(0,\cdot) \big|^2   \ d2\  \\
			& \leq C \n{\hat{\S}}_{\t{H}^1}^2 
			\leq C \n{\S}_{\t{H}^1}^2 .
		\end{align*}
		And in an analogous manner we get 
		\begin{align*}
			\Big|\int_{0}^{\zeta_2} \h_1 \j_{2i} \cdot \j_{1l}(0, \cdot) \cdot \hat{S}_{il}(0,\cdot) \ d2\Big|^2 \leq  C \n{\S}_{\t{H}^1}^2.
		\end{align*}
		Latter two estimates yield the boundedness result for ${\mathcal{Y}}_{2,A}^{s}$. 
	\end{proof}

\subsection{Discretization}
Let us apply the  $\mathcal{Y}_i^s$ for the discretization of the mixed elasticity system with strong symmetry.
Our starting point is the next definition of  spline spaces:
\begin{align*}
	\widehat{\t{V}}_{h,2}^s &= \widehat{\t{V}}_{h,2}^s(p,r) \coloneqq   \textup{SYM}(	S_{p+1,p-1}^{r+1,r-1},S_{p,p}^{r,r}, S_{p-1,p+1}^{r-1,r+1})   , \\
	\ \ \ \widehat{\f{V}}_{h,3}^s&=\widehat{\f{V}}_{h,3}^s(p,r) \coloneqq \big(S_{p,p-1}^{r,r-1} \times S_{p-1,p}^{r-1,r}\big)^T\textup{with} \ \ p>r\geq 1.
\end{align*}
We note that with $r \geq 1$ we get  directly $\widehat{\t{V}}_{h,2}^s \subset \t{H}(\widehat{\o},\widehat{\d},\mathbb{S}) \cap \t{H}^1(\widehat{\o})$. And the isogeometric test spaces are defined by means of  $\tilde{\mathcal{Y}}^s_2, \ (\mathcal{Y}^s_3)^{-1}$, namely
\begin{align}
	{\t{V}}_{h,2}^s \coloneqq \tilde{\mathcal{Y}}_2^s \big(\widehat{\t{V}}_{h,2}^s\big), \ \ \ \ {\f{V}}_{h,3}^s \coloneqq (\mathcal{Y}_3^s)^{-1} \big(\widehat{\f{V}}_{h,3}^s\big).
\end{align}
Exploiting the piecewise smoothness of splines one obtains the next properties.
\begin{lemma}
	\label{lemma_sch1}
	It holds ${\mathcal{Y}}_2^{s} \circ \tilde{\mathcal{Y}}_2^{s}( \hat{\f{\tau}}_h) = \hat{\f{\tau}}_h, \ \forall \hat{\f{\tau}}_h \in \widehat{\t{V}}_{h,2}^s$ and $(\widehat{\nabla} \cdot) \circ  {\mathcal{Y}}_2^{s} = {\mathcal{Y}}_3^{s} \circ {\nabla} \cdot\ $ on ${\t{V}}_{h,2}^s$.
\end{lemma}
\begin{proof}
	Above in Section \ref{section_dirichlet_problem} we assumed several times the mappings $\S$ and $\tS$ to be smooth enough, namely $C^1$-regular. By the piecewise smoothness of spline functions and the fact that $\widehat{\t{V}}_{h,2}^s \subset \t{H}(\widehat{\o},\widehat{\d},\mathbb{S}) \cap \t{H}^1(\widehat{\o})$ and ${\t{V}}_{h,2}^s \subset \t{H}({\o},{\d},\mathbb{S}) \cap \t{H}^1({\o})$ the proof steps for Lemma \ref{Lemma_invertebility_Gamma_2_2} (see Section \ref{Lemma_proof_appendix_3}) and the  divergence relation from    Lemma  \ref{Lemma_compatibility_strong} are still valid for our choice of test functions.
\end{proof}
We get the next useful result:
\begin{lemma}
	\label{Lemma_exact_dis}
	The diagram in Fig. \ref{Fig:com_diagram_6} commutes and it is exact meaning  the right column corresponds to the  image of the left column w.r.t. the divergence operator. Moreover we have $\widehat{\nabla} \cdot \big(\widehat{\t{V}}_{h,2}^s \cap \t{H}_{\hat{\Gamma}_1}(\widehat{\Omega},\widehat{\d},\s)\big) = \widehat{\f{V}}_{h,3}^s$.
\end{lemma}
\begin{proof}
	We show the last equation of the assertion and  $\widehat{\nabla} \cdot \big(\widehat{\t{V}}_{h,2}^s \big) \subset \widehat{\f{V}}_{h,3}^s$. Then the commutativity and exactness of the diagram in Fig. \ref{Fig:com_diagram_6} follows by Lemma  \ref{lemma_sch1}. \\ Since   \begin{align*}
		\widehat{\nabla} \cdot \big[S_{p+1,p-1}^{r+1,r-1} \times S_{p,p}^{r,r} \big]&= \h_1\big(S_{p+1,p-1}^{r+1,r-1}\big) + \h_2\big(S_{p,p}^{r,r}\big)\subset S_{p,p-1}^{r,r-1} , \\\widehat{\nabla} \cdot \big[S_{p,p}^{r,r} \times S_{p-1,p+1}^{r-1,r+1} \big]&= \h_1\big(S_{p,p}^{r,r}\big) + \h_2\big(S_{p-1,p+1}^{r-1,r+1}\big) \subset  S_{p-1,p}^{r-1,r},
	\end{align*}
	we obtain  indeed $\widehat{\nabla} \cdot \big(\widehat{\t{V}}_{h,2}^s \big) \subset \widehat{\f{V}}_{h,3}^s$.\\
	Now let $\tilde{\f{v}} \coloneqq (\tilde{v}_1,\tilde{v}_2)^T \in \widehat{\f{V}}_{h,3}^s$ arbitrary but fixed. Then define $\tilde{\f{\tau}}= (\tilde{\tau}_{ij})$ through \begin{align} \label{eq_construction_1}
		\tilde{\tau}_{11} &\coloneqq \int_{0}^{\zeta_1} \tilde{v}_1 \ d1 \in \int_{0}^{\zeta_1} S_{p,p-1}^{r,r-1} \ d1 \subset S_{p+1,p-1}^{r+1,r-1} , \\ \ \tilde{\tau}_{22} &\coloneqq \label{eq_construction_2} \int_{0}^{\zeta_2} \tilde{v}_2 \ d2 \in \int_{0}^{\zeta_2} S_{p-1,p}^{r-1,r} \ d2 \subset S_{p-1,p+1}^{r-1,r+1} , \\ \ \tilde{\tau}_{12}&=\tilde{\tau}_{21} \coloneqq 0. \label{eq_construction_3}
	\end{align}
	Clearly, we have $\tilde{\f{\tau}} \in \widehat{\t{V}}_{h,2}^s \cap \t{H}_{\hat{\Gamma}_1}(\widehat{\Omega},\widehat{\d},\s)$ and $\widehat{\nabla} \cdot \tilde{\f{\tau}} = \tilde{\f{v}}$. The arbitrariness of $\tilde{\f{v}}$ finishes the proof. 
\end{proof}
\begin{figure}[h!]
	\centering
	\begin{tikzpicture}
		\node at (0,1.4) {${\f{V}}_{h,3}^s$};

		\draw[->] (-1.6,1.4) to (-0.72,1.4);		
		
		\node at (-1.2,1.65) {$\nabla \cdot$};			
		\node at (-2.25,1.4) {${\t{V}}_{h,2}^s$};			
		
		\draw[->] (-2.25,1.1) -- (-2.25,0.2);			
		\draw[->] (0.95-0.8,1.1) -- (0.95-0.8,0.2);	
		\node at (0,-0.2) {$\widehat{\f{V}}_{h,3}^s$};			
		
		\draw[->] (-1.6,-0.2) to (-0.72,-0.2);
		
		\node at (-1.2,0.15) {$\widehat{\nabla} \cdot $};			
		\node at (-2.25,-0.2) {$\widehat{\t{V}}_{h,2}^s$};			
		\node[left] at (-2.2,0.7) { \small $\mathcal{Y}_{2}^{s}$};
		\node[left] at (0.15,0.7) { \small $\mathcal{Y}_3^{s}$};	
		
	\end{tikzpicture}
	\caption{The diagram commutes and it is exact.}
	\label{Fig:com_diagram_6}
\end{figure}

\begin{theorem}[Well-posedness of the mixed formulation with strong symmetry]
	\label{theorem:well-posedness}
	The weak formulation \eqref{weak_form_strong_symmetry_contiuous_1} with strong symmetry  is well-posed. In other words,  there exists a unique solution and the underlying  inf-sup condition is satisfied. 
\end{theorem}

\begin{proof}
	Follows directly by the explanations in   \cite[Section 8.8]{ArnoldBook} and \cite[Section 4.4.2]{ArnoldBook}. 
\end{proof}

Finally we can prove now that the weak formulation  with Dirichlet BC and  strong symmetry can be discretized properly by means of the spaces ${\t{V}}_{h,2}^s, \ {\f{V}}_{h,3}^s$. 

\begin{theorem}[Discrete method in the strong symmetry case]
	\label{theorem_ discretization_strong_sym}
	Let $(\f{\sigma},\f{u})$ be the exact solution of the weak form with strong symmetry \eqref{weak_form_strong_symmetry_contiuous_1} equipped with Dirichlet data $\f{u}_D \in \f{H}^{1/2}(\Gamma)$. The discretized weak form that seeks for $\f{\sigma}_h \in {\t{V}}_{h,2}^s, \ \f{u}_h \in {\f{V}}_{h,3}^s $ s.t.
	\begin{alignat}{4}
		\label{dis_weak_form_strong_symmetry_contiuous}
		&\langle \p{A} \boldsymbol{\sigma}_h , \boldsymbol{\tau}_h \rangle \ + \ &&\langle \f{u}_h , \nabla \cdot \boldsymbol{\tau}_h  \rangle \  \color{black}  &&= \langle \boldsymbol{\tau}_h \cdot \p{n} ; \f{u}_D\rangle_{\Gamma}, \ \hspace{0.4cm} &&\forall \boldsymbol{\tau}_h \in {\t{V}}_{h,2}^s, \nonumber\\
		&\langle \nabla \cdot \boldsymbol{\sigma}_h , \f{v}_h \rangle &&  &&= \langle \f{f},\f{v}_h\rangle, \ \hspace{0.4cm} && \forall \f{v}_h \in {\f{V}}_{h,3}^s,
	\end{alignat}
	is inf-sup stable. More precisely, there are constants $0<\gamma_1, \gamma_2$ independent from $h$ satisfying 
	\begin{align}
		\underset{\f{v}_h \in {\f{V}}_{h,3}^s\textbackslash \{0\}}{\inf}  \  \ \underset{\f{\tau}_h \in {\t{V}}_{h,2}^s}{\sup} \	\frac{ \langle \nabla \cdot \f{\tau}_h, \f{v}_h \rangle }{ \norm{\f{\tau}_h}_{\t{H}(\d)} \norm{\f{v}_h}_{\f{L}^2} } \geq \gamma_1 > 0 
	\end{align}
	and 
	\begin{align*}
		\langle \p{A} \boldsymbol{\tau}_h , \boldsymbol{\tau}_h \rangle &\geq  \gamma_2  \norm{\f{\tau}_h}_{\t{H}(\d)}^2, \ \ \forall \f{\tau}_h \in B_h, \ \textup{with} \ \\
		B_h&\coloneqq \{ \f{\tau}_h \in {\t{V}}_{h,2}^s \ | \ \langle \nabla \cdot \f{\tau}_h, \f{v}_h \rangle = 0 , \ \forall \f{v}_h \in {\f{V}}_{h,3}^s \}.
	\end{align*}

\end{theorem}

\begin{proof}
	Due to Lemma \ref{Lemma_exact_dis} it is $\norm{{\f{s}}}_{\t{H}({\d})}= \norm{{\f{s}}}_{\t{L}^2}$ for all ${\f{s}} \in B_h$.  Then it follows easily $$\langle \p{A}{\f{s}} , {\f{s}} \rangle \geq  \gamma_2  \norm{\f{s}}_{\t{L}^2}^2 =  \gamma_2 \norm{\f{s}}_{\t{H}({\d})}^2, \ \ \forall \f{s} \in B_h,$$
	for some $\p{A}$-dependent constant $\gamma_2$. \\
	Now  we show the inf-sup condition first in the parametric domain, i.e. if $\f{F}= \textup{id}$. \\
	Let $\tilde{\f{v}} \in \widehat{\f{V}}_{h,3}^s \textbackslash \{0 \}$  arbitrary but fixed. With the choice \eqref{eq_construction_1} - \eqref{eq_construction_3} in proof of the mentioned lemma we see that there is a $\tilde{\f{\tau}}=\tilde{\f{\tau}}(\tilde{\f{v}})  \in \widehat{\t{V}}_{h,2}^s \cap \t{H}_{\hat{\Gamma}_1}(\widehat{\Omega},\widehat{\d},\s)$ s.t. $\widehat{\nabla} \cdot \tilde{\f{\tau}} = \tilde{\f{v}}$ and $\n{\tilde{\f{\tau}}}_{\t{L}^2} \leq \n{\tilde{\f{v}}}_{\f{L}^2}$.    And one obtains
	\begin{align*}
		\underset{\tilde{\f{s}} \in \widehat{\t{V}}_{h,2}^s }{\sup}	\frac{ \langle \nabla \cdot \tilde{\f{s}}, \tilde{\f{v}} \rangle }{ \norm{\tilde{\f{s}}}_{\t{H}(\widehat{\d})}  } \geq \frac{ \langle \nabla \cdot \tilde{\f{\tau}}, \tilde{\f{v}} \rangle }{ \norm{\tilde{\f{\tau}}}_{\t{H}(\widehat{\d})}  } \geq \frac{1}{\sqrt{2}} \frac{ \n{\tilde{\f{v}}}_{\f{L}^2}^2}{ \norm{\tilde{\f{v}}}_{\f{L}^2}  } \geq \frac{1}{\sqrt{2}}  \norm{\tilde{\f{v}}}_{\f{L}^2}.
	\end{align*}
	The arbitrariness of $\tilde{\f{v}} $ shows us the validity of the inf-sup stability result in the parametric case.
	
	The general case is now shown utilizing  Lemma \ref{Lemma_compatibility_BC_strong_sym} and Lemma \ref{Lemmma_strong_sym_div}.
	Let $\f{v} \in {\f{V}}_{h,3}^s \textbackslash \{0 \}$ and $\tilde{\f{v}} \coloneqq \mathcal{Y}^s_3(\f{v})$ arbitrary but fixed. Then if  $\tilde{\f{\tau}}$ is chosen as above and setting $\f{\tau} \coloneqq (\mathcal{Y}_{2,\Gamma_1}^s)^{-1} \tilde{\f{\tau}}$ we get $ \nabla \cdot \f{\tau}= {\f{v}}$. One observes $\tilde{\f{\tau}} \in \t{H}_{\hat{\Gamma}_1}(\widehat{\Omega},\widehat{\d},\s)$. By Lemma \ref{Lemma_strong_sym_L2} and the invertibility of $\mathcal{Y}_{2,\Gamma_1}^s$ there are constants $C >0$ s.t.  $$ \norm{\f{\tau} }_{\t{L}^2} \leq C \norm{\tilde{\f{\tau}} }_{\t{L}^2} \leq  C   \norm{\tilde{\f{v}} }_{\f{L}^2} \leq  C   \norm{{\f{v}} }_{\f{L}^2}  .  $$ Consequently,
	\begin{align*}
		\underset{{\f{s}} \in {\t{V}}_{h,2}^s}{\sup}	\frac{ \langle \nabla \cdot {\f{s}}, {\f{v}} \rangle }{ \norm{{\f{s}}}_{\t{H}(\d)}  } \geq \frac{ \langle \nabla \cdot {\f{\tau}}, {\f{v}} \rangle }{ \norm{{\f{\tau}}}_{\t{H}(\d)}  } \geq \frac{\norm{{\f{v}}}_{\f{L}^2}^2 }{C \ \sqrt{2} \ \norm{{\f{v}}}_{\f{L}^2}} \geq \frac{1}{C \ \sqrt{2}}\norm{{\f{v}}}_{\f{L}^2} \ \ .
	\end{align*}
	Latter estimate yields the inf-sup stability result.
\end{proof}

\subsection{A first error estimate}
In this last theory section we look at an error estimate for the Dirichlet problem with strong symmetry. We use a standard approach involving spline projections and orient ourselves towards \cite{Buffa2011IsogeometricDD}. 

\begin{definition}
	For $p_1-1 \geq r_1 \geq -1, \ p_2-1 \geq r_2 \geq 0$ and $p_3-1 \geq r_3 \geq 1$ we set
	\begin{align}
		\widehat{\Pi}_{p_1} &\colon L^2((0,1)) \rightarrow S_{p_1}^{r_1}, \  v \mapsto  \sum_i (\lambda_i^{p_1}v)\widehat{B}_{i,p_1},  \ \ \textup{analogous to }(3.6)  \ \textup{in \cite{Buffa2011IsogeometricDD}} \nonumber, \\
		\widehat{\Pi}_{p_2-1}^{c,1} &\colon L^2((0,1)) \rightarrow S_{p_2-1}^{r_2-1}, \ v \mapsto \frac{d }{d \zeta} \widehat{\Pi}_{p_2} \int_{0}^{\zeta_1} v(\tau) \ \textup{d} \tau,  \ \ \textup{analogous to } (3.9) \ \textup{in \cite{Buffa2011IsogeometricDD}}, \nonumber \\
		\widehat{\Pi}_{p_3-2}^{c,2} &\colon L^2((0,1)) \rightarrow S_{p_3-2}^{r_3-2}, \ v \mapsto \frac{d }{d \zeta} \widehat{\Pi}_{p_3-1}^{c,1} \int_{0}^{\zeta_1} v(\tau) \ \textup{d} \tau  = \frac{d^2}{d \zeta^2} \widehat{\Pi}_{p_3} \int_{0}^{\zeta_1} \int_{0}^{\tau} v(r) \ \textup{d}r\, \textup{d} \tau. \nonumber	
	\end{align}
	Above $\lambda_i^p$ denote the canonical dual basis functionals corresponding to the spline basis functions, i.e. $\lambda_i^p( \widehat{B}_{j,p}) = \delta_{ij}$. For more information we refer to \cite[Section 4.6]{schumaker_2007}.
\end{definition}
Below, we assume for the underlying regularity parameter always $r_i=r$.
Useful for our purposes are the subsequent properties.
\begin{lemma}
	\label{lemma:univariate_spline_proj_prop}
	The above  projections of the univariate case satisfy the following  properties, where we assume $r \geq -1, \ p \geq 0$ for the first two lines and $p -1 \geq r \geq 0$ otherwise and let $I \coloneqq (\psi_i,\psi_{i+1})$ denote an arbitrary sub-interval induced by the discretization.
	\begin{align}
		\widehat{\Pi}_{p}s &= s  & \forall s \in S^{r}_{p} , \label{eq:lemma_uni_spline_prop_0}\\		
		|\widehat{\Pi}_{p} v|_{H^l(I)} &\leq C \ 	|v|_{H^l(\tilde{I})}  & \forall v \in H^l((0,1)), \ 0 \leq l \leq p+1, \nonumber \\
		\label{eq:lemma_uni_spline_prop_1}
		\widehat{\Pi}_{p-1}^{c,1}s &= s  & \forall s \in S^{r-1}_{p-1} , \\
		\label{eq:lemma_uni_spline_prop_2}
		\widehat{\Pi}_{p-1}^{c,1} \partial_{\zeta} v &= \partial_{\zeta} \widehat{\Pi}_{p}  v  & \forall v \in H^1((0,1)), \\
		\label{eq:lemma_uni_spline_prop_3}
		|\widehat{\Pi}_{p-1}^{c,1} v|_{H^l(I)} &\leq C \ 	| v|_{H^l(\tilde{I})}  & \forall v \in H^l((0,1)), \ 0 \leq l \leq p,
	\end{align}
	for some constant $C$ independent of mesh refinement.\\
	And if it is $r \geq 1, \ p \geq 2$ we  further have 
	\begin{align}
		\label{eq:lemma_uni_spline_prop_4}
		\widehat{\Pi}_{p-2}^{c,2}s &= s  & \forall s \in S^{r-2}_{p-2} , \\
		\label{eq:lemma_uni_spline_prop_5}
		\widehat{\Pi}_{p-2}^{c,2} \partial_{\zeta} v &= \partial_{\zeta} \widehat{\Pi}_{p-1}^{c,1}  v & \forall v \in H^1((0,1)), \\	
		|\widehat{\Pi}_{p-2}^{c,2} v|_{H^l(I)} &\leq C \ 	| v|_{H^l(\tilde{I})}  & \forall v \in H^l((0,1)), \ 0 \leq l \leq p-1. \label{eq:lemma_uni_spline_prop_6}
	\end{align}
	Above $|\cdot|_{H^l}$ denotes  the standard Sobolev seminorm.
\end{lemma}
\begin{proof}
	The first five statements correspond to the properties (3.7),\ (3.8) and (3.10)-(3.12) in \cite{Buffa2011IsogeometricDD} if we keep the regular mesh assumption in mind.
	Hence we only check the points \eqref{eq:lemma_uni_spline_prop_4}-\eqref{eq:lemma_uni_spline_prop_6}.\\
	In the following we use (3.4) of \cite{Buffa2011IsogeometricDD}, which gives immediately that $\partial_{\zeta}^2 S_{p}^{r} = S_{p-2}^{r-2}$. Firstly, with property \eqref{eq:lemma_uni_spline_prop_1} we have
	\begin{align*}
		\widehat{\Pi}_{p-2}^{c,2}s =  \frac{d }{d \zeta} \widehat{\Pi}_{p-1}^{c,1} \underbrace{\int_{0}^{\zeta_1} s(\tau) \ \textup{d} \tau}_{\in S_{p-1}^{r-1}} =  \frac{d }{d \zeta}  {\int_{0}^{\zeta_1} s(\tau) \ \textup{d} \tau} =  s.
	\end{align*}
	In view of \eqref{eq:lemma_uni_spline_prop_0} and \eqref{eq:lemma_uni_spline_prop_2} one gets 
	\begin{align*}
		\widehat{\Pi}_{p-2}^{c,2} \partial_{\zeta} v &= \frac{d^2}{d \zeta^2} \widehat{\Pi}_p \int_{0}^{\zeta_1}  v(\tau) + a \ \textup{d}\tau\, = \frac{d^2}{d \zeta^2} \widehat{\Pi}_p \int_{0}^{\zeta_1}   v(\tau) \ \textup{d}\tau = \frac{d}{d \zeta} \widehat{\Pi}_{p-1}^{c,1} v.
	\end{align*}
	In the last line we wrote $a \in \mathbb{R}$ for a suitable constant of integration. Hence \eqref{eq:lemma_uni_spline_prop_5} follows.\\ 
	The last point can be proven in a similar fashion like equation 3.12 in \cite{Buffa2011IsogeometricDD} using the fact that if $v \in H^l((0,1)) $ then $w(\zeta)  \coloneqq \int_{0}^{\zeta_1} v \textup{d} \tau  \in H^{l+1}((0,1))$ and using estimates \eqref{eq:lemma_uni_spline_prop_3}, \eqref{eq:lemma_uni_spline_prop_5}. It is
	$$|\widehat{\Pi}_{p-2}^{c,2} v|_{H^l(I)} \leq |\widehat{\Pi}_{p-1}^{c,1} w|_{H^{l+1}(I)} \leq C \, |w|_{H^{l+1}(\tilde{I})}  \leq C \, |v|_{H^{l}( \tilde{I})} .$$
\end{proof}

\begin{lemma}[Projections on the parametric domain]
	We define the next projections on the parametric domain 
	\begin{align}
		\widehat{\Pi}_{h,2}^s = \widehat{\Pi}_{h,2}^s(p,r) &\coloneqq \textup{SYM}(\widehat{\Pi}_{p+1} \otimes \widehat{\Pi}_{p-1}^{c,2}, \  \widehat{\Pi}_{p}^{c,1} \otimes \widehat{\Pi}_{p}^{c,1} ,  \ \widehat{\Pi}_{p-1}^{c,2} \otimes \widehat{\Pi}_{p+1}) ,\\
		\widehat{\Pi}_{h,3}^s = \widehat{\Pi}_{h,3}^s(p,r) & \coloneqq \big(\widehat{\Pi}_{p}^{c,1} \otimes \widehat{\Pi}_{p-1}^{c,2}, \ \widehat{\Pi}_{p-1}^{c,2} \otimes \widehat{\Pi}_{p}^{c,1} \big)^T.
	\end{align}
	It holds the commutativity relation 
	$$(\widehat{\nabla} \cdot ) \circ \widehat{\Pi}_{h,2}^s = \widehat{\Pi}_{h,3}^s \circ (\widehat{\nabla} \cdot )  \ \  \textup{on} \ \t{H}(\widehat{\o},\widehat{\d},\s). $$
	If $\hat{\f{\tau}} \coloneqq\textup{SYM}(\hat{\tau}_{11},\hat{\tau}_{12},\hat{\tau}_{22}) \in {\t{H}}(\widehat{\o},\widehat{\d}, \mathbb{S}) \cap \t{H}^k(\widehat{\o})$  and $\hat{\f{v}} \in \f{H}^k(\widehat{\o})$ we have for some constant $C$ independent of $h$  the estimates
	\begin{align}
		\norm{\hat{\f{\tau}}- \widehat{\Pi}_{h,2}^s\hat{\f{\tau}}}_{\t{H}^1}  &\leq C h^{k-1}    \norm{\hat{\f{\tau}}}_{\t{H}^k} , p \geq k \geq 1, \ p> r \geq 1, \\
		\norm{\hat{\f{v}}-\widehat{\Pi}_{h,3}^s\hat{\f{v}}}_{\f{L^2}} &\leq C  h^k   \norm{\hat{\f{v}}}_{\f{H}^k}  , \ p \geq  k \geq 0.
	\end{align}
\end{lemma}
\begin{proof}
	We begin with the commutativity relation. Basically, the proof steps are the same as of Lemma 4.3 in \cite{Buffa2011IsogeometricDD}, but for reasons of completeness, we sketch the proof steps. Therefore, we first assume that we have a smooth mapping  $\hat{\f{\tau}} \in \t{H}(\widehat{\o},\widehat{\d},\s)$. We exploit now  Lemma \ref{lemma:univariate_spline_proj_prop} and the tensor product structure of the multivariate spline projections. For the latter we refer to the explanations in Section 4.1. of the  mentioned paper \cite{Buffa2011IsogeometricDD}. This makes the next equality chain reasonable.
	\begin{align*}
		\Big[(\widehat{\nabla} \cdot ) \circ \widehat{\Pi}_{h,2}^s \hat{\f{\tau}}  \Big]_1 &= \h_1\big( \widehat{\Pi}_{p+1} \otimes \widehat{\Pi}_{p-1}^{c,2}\big)\hat{\tau}_{11}  + \h_2 \big(\widehat{\Pi}_{p}^{c,1} \otimes \widehat{\Pi}_{p}^{c,1} \big) \hat{\tau}_{12} \\
		&= \big( \h_1\widehat{\Pi}_{p+1} \otimes \widehat{\Pi}_{p-1}^{c,2}\big)\hat{\tau}_{11}  +  \big(\widehat{\Pi}_{p}^{c,1} \otimes \h_2\widehat{\Pi}_{p}^{c,1} \big) \hat{\tau}_{12} \\
		&=\big( \widehat{\Pi}_{p}^{c,1}\h_1 \otimes \widehat{\Pi}_{p-1}^{c,2}\big)\hat{\tau}_{11}  +  \big(\widehat{\Pi}_{p}^{c,1} \otimes \widehat{\Pi}_{p-1}^{c,2}\h_2 \big) \hat{\tau}_{12} \\
		&= \big( \widehat{\Pi}_{p}^{c,1} \otimes \widehat{\Pi}_{p-1}^{c,2}\big)\h_1\hat{\tau}_{11}  +  \big(\widehat{\Pi}_{p}^{c,1} \otimes \widehat{\Pi}_{p-1}^{c,2} \big) \h_2\hat{\tau}_{12} = \big( \widehat{\Pi}_{p}^{c,1} \otimes \widehat{\Pi}_{p-1}^{c,2}\big) (\widehat{\nabla} \cdot [\hat{\tau}_{11},\hat{\tau}_{12}]^T).
	\end{align*}
	A similar calculation can be applied for the second entry of $(\widehat{\nabla} \cdot ) \circ \widehat{\Pi}_{h,2}^s \hat{\f{\tau}} $. Consequently, the commutativity is clear.\\
	The estimates  for the projection are based on the boundedness and spline-preserving properties of the univariate projections in Lemma  \ref{lemma:univariate_spline_proj_prop}. The proof is analogous to the proof of  Lemma 5.1 in \cite{Buffa2011IsogeometricDD} and we state briefly the underlying idea. Let $\hat{\f{\tau}} \in {\t{H}}(\widehat{\o},\widehat{\d}, \mathbb{S}) \cap \t{H}^k(\widehat{\o}), \ p \geq k \geq 1, \ p > r \geq 1$. One uses Lemma 3.1 of \cite{IGA3} which gives us the existence of a spline mapping $\f{s} \in \widehat{\t{V}}_{h,2}^s$ such that $$\norm{\hat{\f{s}}-\hat{\f{\tau}}}_{\t{H}^1} \leq C  h^{k-1}  \norm{\hat{\f{\tau}}}_{\t{H}^k}, \ \ \ \norm{\hat{\f{s}}-\hat{\f{\tau}}}_{\t{L}^2} \leq C  h^{k}  \norm{\hat{\f{\tau}}}_{\t{H}^k},$$
	for a suitable  constant $C$. Since the univariate projections are spline preserving and the tensor-product structure of the multivariate projections fits to the product structure of the spline spaces we get directly the spline preserving property of $\widehat{\Pi}_{h,2}^s$ and $\widehat{\Pi}_{h,3}^s$. Thus,  an application of the triangle inequality and a standard inverse estimate for polynomials yields
	\begin{align*}
		\norm{\hat{\f{\tau}}- \widehat{\Pi}_{h,2}^s\hat{\f{\tau}}}_{\t{H}^1}  & \leq 	\norm{\hat{\f{\tau}}- \hat{\f{s}}}_{\t{H}^1}  + 	\norm{\widehat{\Pi}_{h,2}^s \big(\hat{\f{s}}- \hat{\f{\tau}} \big)}_{\t{H}^1}   \\
		& \leq C  h^{k-1}  \norm{\hat{\f{\tau}}}_{\t{H}^k} + \frac{C}{h} \norm{ \hat{\f{s}}- \hat{\f{\tau}} }_{\t{L}^2}\\
		& \leq C  h^{k-1}  \norm{\hat{\f{\tau}}}_{\t{H}^k}.
	\end{align*}
	Above we used the boundedness of the projection w.r.t. to the $\t{L}^2$-norm following from the boundedness properties in Lemma \ref{lemma:univariate_spline_proj_prop}.
\end{proof}

\begin{lemma}
	\label{lemma:improved_estimate}
	If $\f{\tau} \in \t{H}({\o},{\d},\s) \cap \t{H}^{k+1}(\o)$ and $\f{v} \in \f{H}^k(\o), \ k\geq 0$, then it holds 
	\begin{align*}
		\norm{\mathcal{Y}_2^s(\f{\tau})}_{\t{H}^k}  &\leq C    \norm{{\f{\tau}}}_{\t{H}^{k+1}} , \\
		\norm{\mathcal{Y}_3^s(\f{{v}})}_{\f{H}^k} &\leq C  \norm{{\f{v}}}_{\f{H}^k},
	\end{align*}
	assuming a regular enough $\f{F}$; e.g. $\f{F} \in C^{k+2}(\overline{\widehat{\o}})$ and $  \f{F}^{-1} \in C^{k+2}({\overline{\o}})$. The constant $C$ depends on the parametrization.
\end{lemma}

\begin{proof}
	Again we concentrate on the first estimate, since the second one is proven with a similar approach. Further, we first assume $\f{\tau} \in  C^{k+1}(\overline{\o}) \cap \t{L}^2(\o,\s)$ and start with the case $k=1$. Note that the assertion is clear for $k=0$; see proof of Lemma \ref{Lemma_Boundedness_strong}. With the definition \eqref{eq_def_Gamma_2s} of $\mathcal{Y}_2^s$, problems may arise due to  the term $\mathcal{Y}_{2,A}^s$ because of the evaluation on the boundary part $\hat{\Gamma}_1$. Applying derivatives $\h_1, \h_2$ to $\mathcal{Y}_{2,A}^s(\f{\tau})$ we obtain expressions involving values and first derivatives $\partial_m \f{\tau} \circ \f{F}(0,\cdot)$ on the boundary edge $\hat{\Gamma}_1$ which are  multiplied by geometry dependent terms. By the trace theorem (see \cite[Chapter 2]{steinbach}) we obtain \begin{align*}
		\norm{{\tau_{ij}} \circ \f{F}}_{{L}^2(\hat{\Gamma}_1)} &\leq C \norm{{\tau_{ij}} \circ \f{F}}_{{H}^1(\widehat{\o})} \leq C \norm{{\tau_{ij}}}_{{H}^1({\o})}, \\  \norm{\partial_m{\tau_{ij}}\circ \f{F}}_{{L}^2(\hat{\Gamma}_1)} &\leq C \norm{\partial_m{\tau_{ij}}\circ \f{F}}_{{H}^1(\widehat{\o})} \leq C \norm{{\tau_{ij}}}_{{H}^2(\o)}.
	\end{align*}
	Consequently, if the geometry function $\f{F}$ is regular enough,  one can  see that $$\norm{\mathcal{Y}_{2,A}^s(\f{\tau})}_{\t{H}^1(\widehat{\o})} \leq C \norm{\f{\tau}}_{\t{H}^2({\o})}$$ and clearly $\norm{\mathcal{Y}_{2}^s(\f{\tau})}_{\t{H}^1(\widehat{\o})} \leq C \norm{\f{\tau}}_{\t{H}^2({\o})}$.
	Since the smooth mappings are dense in $\t{H}^2(\o)$ the inequality is still valid for a general element in $\t{H}^2(\o)$.
	A similar reasoning can be used for the cases $k\geq 2$.

\end{proof}

\begin{lemma}[Projections on the physical domain]
	\label{lemma:projection_physical}
	Utilizing the mappings $\mathcal{Y}_2^s, \ \mathcal{Y}_3^s$ we define the projections onto  the actual test function spaces
	\begin{align*}
		{\Pi}_{h,2}^s &\colon \t{H}(\o,\d,\s) \cap \t{H}^1(\o) \rightarrow \t{V}_{h,2}^s  \ , \  \f{\tau}  \mapsto \tilde{\mathcal{Y}}_2^s\circ \widehat{\Pi}_{h,2}^s \circ \mathcal{Y}_2^s(\f{\tau}), \\ {\Pi}_{h,3}^s &\colon \f{L}^2(\o) \rightarrow \f{V}_{h,3}^s \ , \ \f{v} \mapsto (\mathcal{Y}_3^s)^{-1} \circ \widehat{\Pi}_{h,3}^s \circ \mathcal{Y}_3^s (\f{v}).
	\end{align*}
	Let $\f{\tau} \in  {\t{H}}({\o},\d, \mathbb{S}) \cap \t{H}^{k+1}(\o), \ \f{v} \in \f{H}^k(\o)$ , where $1\leq k \leq p, \ 1 \leq r < p$. Then the next inequalities hold:
	\begin{align*}
		\norm{{\f{\tau}}- {\Pi}_{h,2}^s{\f{\tau}}}_{\t{H}(\d)}  &\leq C  h^{k-1}     \norm{{\f{\tau}}}_{\t{H}^{k+1}} , \\
		\norm{\f{{v}}-{\Pi}_{h,3}^s{\f{v}}}_{\f{L^2}} &\leq C  h^k  \norm{{\f{v}}}_{\f{H}^k}.
	\end{align*}
\end{lemma}
\begin{proof}
	We show the inequality for $\f{\tau}$. Set $\tilde{\f{\tau}} \coloneqq \mathcal{Y}_2^s(\f{\tau})$. Combining Lemma \ref{Lemma_invertebility_Gamma_2_2}, Lemma \ref{Lemma_Boundedness_strong}  and Lemma \ref{lemma:improved_estimate}  one can write
	\begin{align*}
		\norm{{\f{\tau}}- {\Pi}_{h,2}^s{\f{\tau}}}_{\t{H}(\d)} & = \norm{\tilde{\mathcal{Y}}_2^s \circ \mathcal{Y}_2^s({\f{\tau}})- \tilde{\mathcal{Y}}_2^s \circ\widehat{\Pi}_{h,2}^s \tilde{\f{\tau}}}_{\t{H}(\d)}\\
		& \leq  C \norm{\tilde{\f{\tau}}- \widehat{\Pi}_{h,2}^s\tilde{\f{\tau}}}_{\t{H}^1}  \\
		& \leq C h^{k-1} \norm{\tilde{\f{\tau}}}_{\t{H}^k} \\ 
		& \leq C h^{k-1} \norm{\f{\tau}}_{\t{H}^{k+1}}.
	\end{align*}
	The second estimate follows by similar steps.
\end{proof}

\begin{theorem}[Convergence for the strong symmetry case]
	It holds the quasi-optimal error estimate 
	\begin{align*}
		\norm{\f{\sigma}- \f{\sigma}_h}_{\t{H}(\d)} + \norm{\f{u}-\f{u}_h}_{\f{L}^2} \leq C \underset{\f{\tau}_h \in {\t{V}}_{h,2}^s, \ \f{v}_h \in {\f{V}}_{h,3}^s}{\inf} \hspace{0.1cm}	\norm{\f{\sigma}- \f{\tau}_h}_{\t{H}(\d)} + \norm{\f{u}-\f{v}_h}_{\f{L}^2} 
	\end{align*} 
	between the exact  $(\f{\sigma},\f{u})$ and  discrete  $(\f{\sigma}_h,\f{u}_h)$ solution of the mixed formulation with  strong symmetry and boundary condition $\f{u}=\f{u}_D $ on $\partial \o$. In case of  sufficiently smooth exact solution, meaning  $ \f{\sigma}  \in {\t{H}}({\o},{\d}, \mathbb{S}) \cap \t{H}^{k+1}(\o), \ \f{u} \in \f{H}^k({\o}), \ 1 \leq k \leq p, \ 1 \leq r < p $, we can specify the convergence behavior  via the estimate 
	\begin{align}
		\label{theorem_eq_conv}
		\norm{\f{\sigma}- \f{\sigma}_h}_{\t{H}(\d)} + \norm{\f{u}-\f{u}_h}_{\f{L^2}} \leq C  h^{k-1}  \big(   \norm{\f{\sigma}}_{\t{H}^{k+1}} + \norm{\f{u}}_{\f{H}^k} \big).
	\end{align}
	Again $C< \infty$ is independent of the mesh size $h$.
\end{theorem}

\begin{proof} 
	In view of classical results for mixed problems (see e.g. \cite{Brezzi}) the quasi-optimal error estimate is a consequence of the inf-sup stability of the discrete method, see Theorem \ref{theorem_ discretization_strong_sym}, and the well-posedness of the continuous problem; see Theorem \ref{theorem:well-posedness}. Therefore it is enough to show the convergence rates w.r.t. $h$. But this is clear since we can apply the result of the previous Lemma \ref{lemma:projection_physical} to the quasi-optimal error estimate, namely 
	\begin{align*}
		\norm{\f{\sigma}- \f{\sigma}_h}_{\t{H}(\d)} + \norm{\f{u}-\f{u}_h}_{\f{L}^2} &\leq C \Big(  \norm{\f{\sigma}- {\Pi}_{h,2}^s\f{\sigma}}_{\t{H}(\d)} + \norm{\f{u}-{\Pi}_{h,3}^s\f{u}}_{\f{L}^2} \Big) \\
		&\leq  C  h^{k-1}   \big(   \norm{\f{\sigma}}_{\t{H}^{k+1}} + \norm{\f{u}}_{\f{H}^k} \big).
	\end{align*}

\end{proof}

	\begin{remark}
			The error estimate \eqref{theorem_eq_conv} of the last theorem shows an order reduction and further we have to require a relatively smooth exact solution. Hence, the approximation estimate is certainly not entirely satisfactory. We hope that this first estimate is too pessimistic and it is advisable to study numerical examples in order to check the actual approximation performance. Problems within our estimate derivation are caused by the boundedness result for the mapping $\mathcal{Y}_2^s$ (see Lemma \ref{Lemma_Boundedness_strong}) where the $\t{H}^1$-norm appears instead of the wanted $\t{H}(\d)$-norm. At least in the case $\h_1\J=\f{0}$ one can prove boundedness w.r.t. the $\t{H}(\d)$-norm. The latter implies an improved error estimate of the form
		\begin{align}
			\label{eq_improved_estimate}
			\norm{\f{\sigma}- \f{\sigma}_h}_{\t{H}(\d)} + \norm{\f{u}-\f{u}_h}_{\f{L^2}} \leq C  h^{k}  \big(   \norm{\f{\sigma}}_{\t{H}^{k}(\d)} + \norm{\f{u}}_{\f{H}^k} \big).
		\end{align}
		We obtain such an improved estimate also in case of the mixed boundary conditions from Assumption \ref{eq_mixed_BC_strong_sym}. In other words, one can verify that in situations where the additional term $\mathcal{Y}^{s}_{2,A}$ in \eqref{eq_def_Gamma_2s} vanishes, estimates of the form  \eqref{eq_improved_estimate} are valid.
	\end{remark}

	\section{Numerical examples for the weak symmetry case}
	\label{sec_numerics}
	Here we consider different numerical examples  for the discretization method from Section \ref{section_weak_form_weak_symmetry} with weakly imposed symmetry.  The  case of strong symmetry  requires more effort and  certainly deserves a detailed numerical study. Therefore we postpone numerical experiments for the method of Section \ref{section_strong}  to a further article. 
	
	In the following we assume a  homogeneous, isotropic body, i.e. the Lam\'{e} coefficients $\lambda$ and $\mu$ characterize the material. Further, all meshes in the parametric domain are uniform in the sense that for mesh size $h=1/n$ the parametric mesh partitions the domain $\widehat{\o}= (0,1)^2$  into $n^2$  equal squares; see e.g. Fig. \ref{Fig:boundaries} for $h=1/6$. Up to a constant scaling factor latter parametric mesh size is equivalent to the actual mesh size in $\o$. \\
	All numerical computations below have been carried out utilizing MATLAB \cite{MATLAB:2020} together with the GeoPDEs package \cite{geopdes,geopdes3.0}. The plots and figures are created by means of the mentioned software, too. 
	
	\subsection{Convergence tests}
	We start with problems having smooth exact solutions $(\f{\sigma},\f{u},\f{p})$ to check the convergence rates suggested by the theory. We face three different examples. 
	First we look at the single-patch case with homogeneous displacement boundary conditions. The computational domain is  displayed  in Fig. \ref{Fig1} (a) and we define the  right-hand side, i.e. the source function $\f{f}$,   in such a way that the exact solution for the displacement $\f{u}= (u_1,u_2)^T$ is $$u_1 = \big( \sin(\pi \zeta_1)  \sin(\pi \zeta_2)  \big) \circ \f{F}^{-1}, \ \ \ u_2 = -u_1 \ \ \textup{and} \ \ \f{F}(\zeta_1,\zeta_2) \coloneqq (\zeta_1 , \zeta_1^2 + \zeta_2)^T.$$  Besides we set $\lambda=2, \mu=1$.
	\begin{figure}[H]
		\hspace{-1.37cm}
		\begin{minipage}{4.5cm}
			\begin{tikzpicture}
				\node[inner sep=0pt] (ring) at (0,0)
				{\includegraphics[height=5.17cm,width=7.15cm]{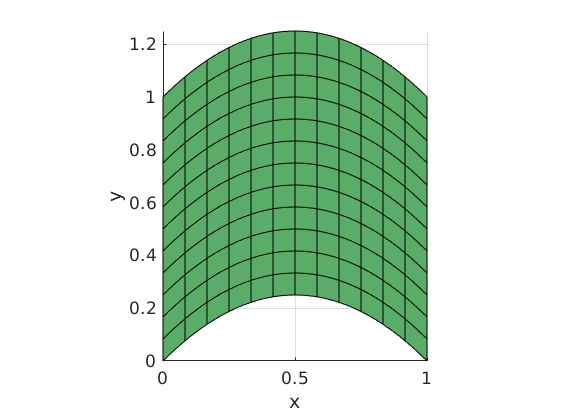}};
				\node at (1.1,2.3) {$\Gamma_4$};
				\node at (0.2,-1.4) {$\Gamma_3$};
			\end{tikzpicture}
			\hspace*{1.25cm}\centering \footnotesize  (a) Mesh for $h=1/12$.
		\end{minipage}
		\hspace{0.65cm}
		\begin{minipage}{4.5cm}
			\begin{tikzpicture}
				\node[inner sep=0pt] (ring) at (0,0)
				{\includegraphics[height=5.17cm,width=7.15cm]{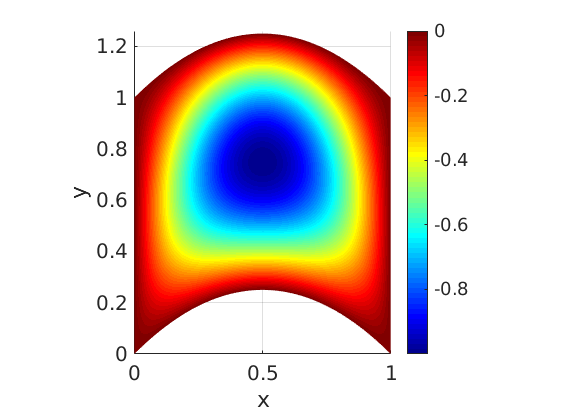}};
				\node at (-3.21,0.1) {$\Gamma_2$};
			\end{tikzpicture}
			\centering
			\hspace*{0.8cm}	\footnotesize (b) $y$-displacement.
		\end{minipage}
		\hspace{1.13cm}
		\begin{minipage}{4.5cm}
			\includegraphics[height=5.17cm,width=7.15cm]{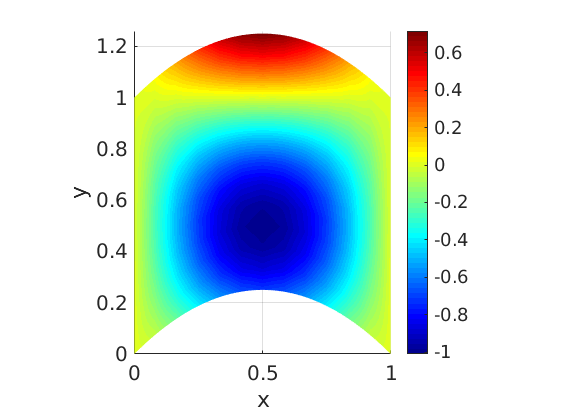} \centering \hspace*{0.8cm} \footnotesize    (c) $y$-displacement.
		\end{minipage}
		\caption{An example with a curved boundary domain. On the left we see the underlying mesh for $h=1/12$. Using  this mesh for  the first test example with homogeneous displacement BC, we obtain the $y$-displacement  shown in the middle ($r=1, p=3$). The $y$-displacement for the second example with traction boundary conditions is displayed on the right.  }
		\label{Fig1}
	\end{figure}
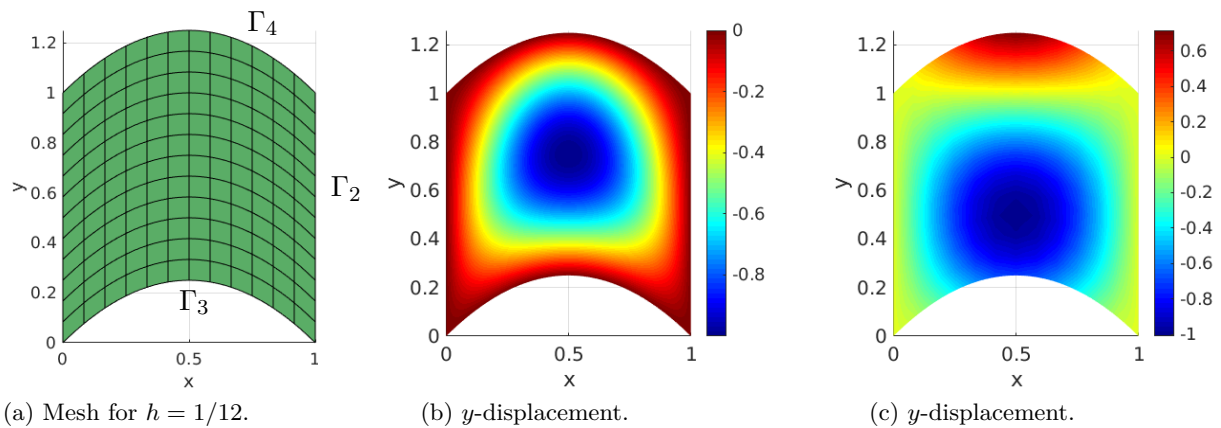
	Applying the weak symmetry method with spaces from Section \ref{section_discretization} we obtain for polynomial degrees $p=2,3,4, $ regularity $r=0$ and for $r=p-2$ (see Section \ref{subsub_sec_regu}) the errors  shown in Fig. \ref{Fig2} and Fig. \ref{Fig3}. For both  cases we see a good accordance with the theoretically predicated convergence orders $\mathcal{O}(h^p)$; see Theorem \ref{Theorem_dis_error_2D}.

	\begin{figure}[H]
		\begin{minipage}{4.5cm}
			\includegraphics[height=5cm,width=5.5cm]{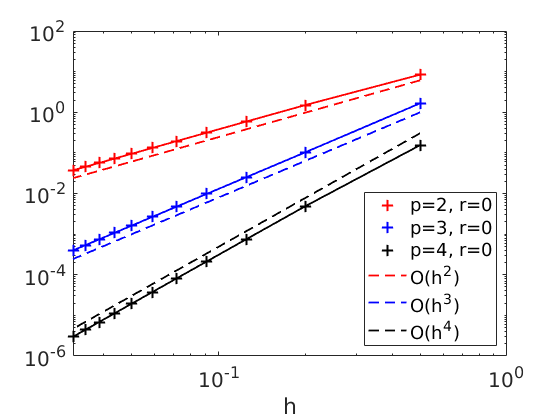}
			\footnotesize \centering \hspace*{-0.8cm} (a)  $\n{\f{\sigma}- \f{\sigma}_h}_{\t{H}(\d)}.$
		\end{minipage}
		\hspace{0.42cm}
		\begin{minipage}{4.5cm}
			\includegraphics[height=5cm,width=5.5cm]{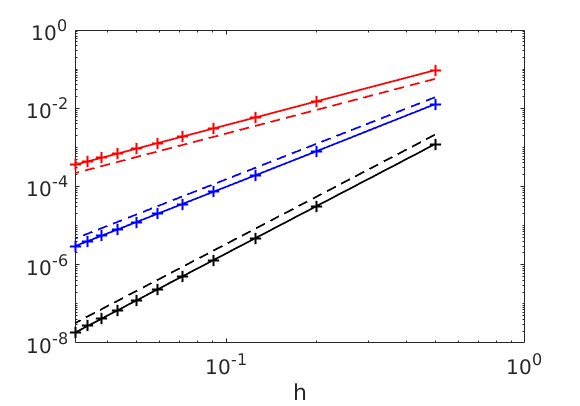}
			\footnotesize \centering \hspace*{-0.8cm} (b)   $\n{\f{u}- \f{u}_h}_{\f{L}^2}.$
		\end{minipage}
		\hspace{0.42cm}
		\begin{minipage}{4.5cm}
			\includegraphics[height=5cm,width=5.5cm]{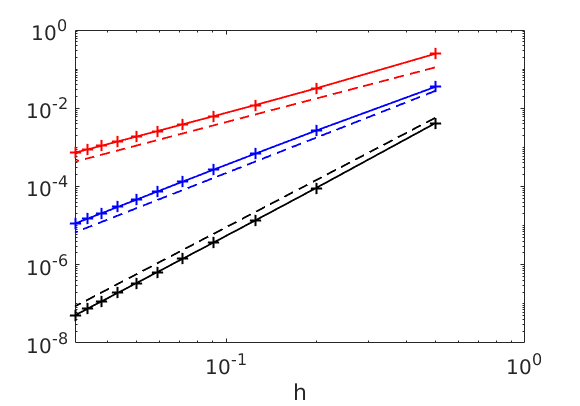}
			\footnotesize \centering  \hspace*{-0.8cm} (c) $\n{\f{p}- \f{p}_h}_{\t{L}^2}.$
		\end{minipage}
		\caption{Error decrease for the case $r=0$ and the first convergence example.}
		\label{Fig2}
	\end{figure}

	\begin{figure}[ht]
		\begin{minipage}{4.5cm}
			\includegraphics[height=5cm,width=5.5cm]{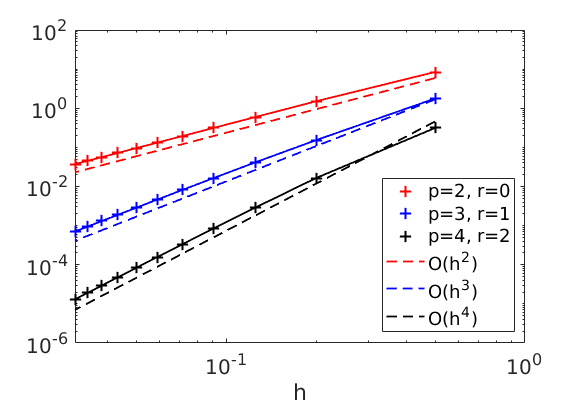} \centering
			\footnotesize \hspace*{-0.8cm} (a)  $\n{\f{\sigma}- \f{\sigma}_h}_{\t{H}(\d)}.$
		\end{minipage}
		\hspace{0.41cm}
		\begin{minipage}{4.5cm}
			\includegraphics[height=5cm,width=5.5cm]{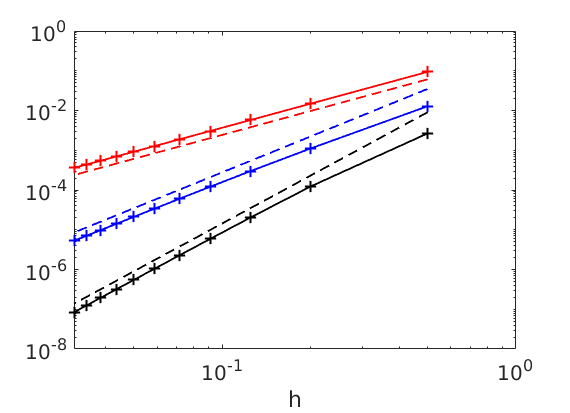}  \centering
			\footnotesize \hspace*{-0.8cm}  (b) $\n{\f{u}- \f{u}_h}_{\f{L}^2}.$
		\end{minipage}
		\hspace{0.41cm}
		\begin{minipage}{4.5cm}
			\includegraphics[height=5cm,width=5.5cm]{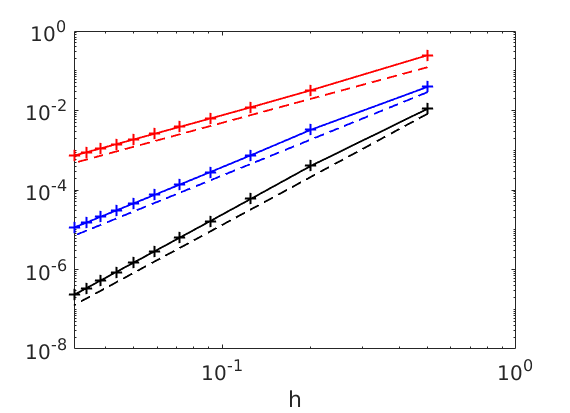}  \centering
			\footnotesize \hspace*{-0.8cm} (c)  $\n{\f{p}- \f{p}_h}_{\t{L}^2}.$
		\end{minipage}
		\caption{Convergence behavior for the first test example with increased spline regularity, meaning   $r=p-2$.}
		\label{Fig3}
	\end{figure}

	\begin{figure}[ht]
		\begin{minipage}{4.5cm}
			\includegraphics[height=5cm,width=5.5cm]{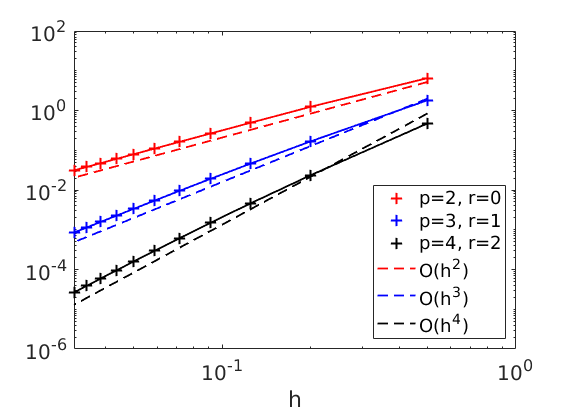} \centering
			\footnotesize \hspace*{-0.8cm} (a)  $\n{\f{\sigma}- \f{\sigma}_h}_{\t{H}(\d)}.$
		\end{minipage}
		\hspace{0.41cm}
		\begin{minipage}{4.5cm}
			\includegraphics[height=5cm,width=5.5cm]{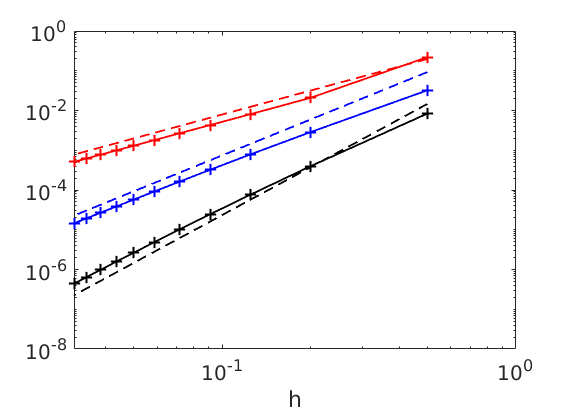} \centering
			\footnotesize \hspace*{-0.8cm} (b)  $\n{\f{u}- \f{u}_h}_{\f{L}^2}.$
		\end{minipage}
		\hspace{0.41cm}
		\begin{minipage}{4.5cm}
			\includegraphics[height=5cm,width=5.5cm]{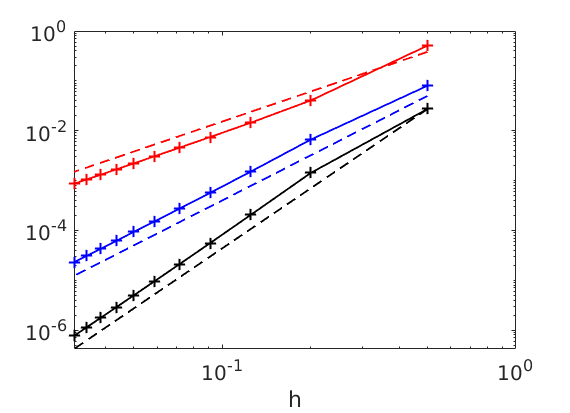} \centering
			\footnotesize \hspace*{-0.8cm} (c)  $\n{\f{p}- \f{p}_h}_{\t{L}^2}.$
		\end{minipage}
		\caption{The errors for the second test example with mixed boundary conditions.}
		\label{Fig5}
	\end{figure}
	
	Then we modify the last test, namely we have a new source function corresponding to  the  exact solution  $$u_1 =  \sin(\pi x)  \sin(\pi y)  , \ \ \ u_2 = -u_1,$$ together with  traction boundary conditions on the right, bottom and top edge of the domain, i.e. $\Gamma_2, \Gamma_3, \Gamma_4$ in Fig. \ref{Fig1} (a). The exact normal components $\f{\sigma} \cdot \f{n}$ are approximated by a $L^2$-projection onto  the respective boundary spline spaces. In particular we reduce the problem virtually to the case of zero traction  on the $\Gamma_t$ part. The plots in  Fig. \ref{Fig5} show the different errors for the choice $r=p-2$ and the error decrease is comparable to the homogeneous displacement BC situation.

	Next we face a numerical example within the framework of multi-patch parametrizations. We divide the square $[-1,1]^2$ into four patches which have  curved interfaces, see Fig. \ref{Figmp}.
	As manufactured solution we use again   $$u_1 = \big( \sin(\pi x) \ \sin(\pi y)  \big) , \ \ \ u_2 = -u_1 \ \ \textup{and} \ \ \lambda = 2, \mu = 1.$$ 
	Thus, we can  suppose a homogeneous displacement boundary condition.
	We refer to Fig. \ref{Fig4} for the approximation errors. One notes the error reduction of about $\mathcal{O}(h^p)$ which might indicate a good performance of  the method for  multi-patch parametrizations.

	\begin{figure}
		\begin{minipage}{4.5cm}
			\includegraphics[height=5cm,width=5.5cm]{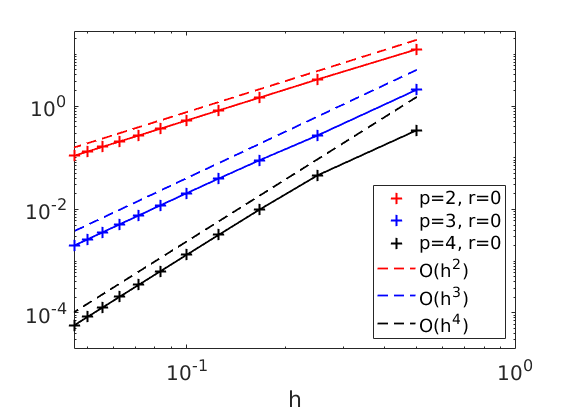}
			\centering
			\footnotesize \hspace*{-0.8cm} (a)  $\n{\f{\sigma}- \f{\sigma}_h}_{\t{H}(\d)}.$
		\end{minipage}
		\hspace{0.41cm}
		\begin{minipage}{4.5cm}
			\includegraphics[height=5cm,width=5.5cm]{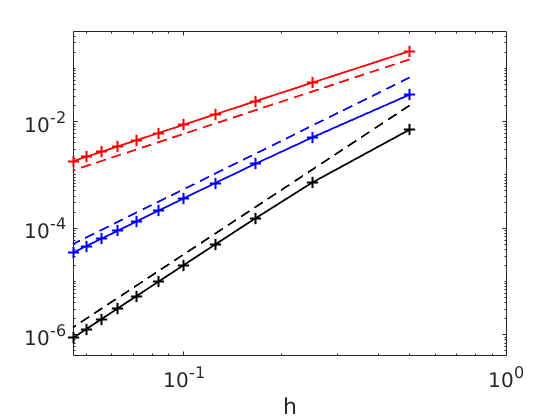}
			\centering
			\footnotesize \hspace*{-0.8cm} (b) $\n{\f{u}- \f{u}_h}_{\f{L}^2}.$
		\end{minipage}
		\hspace{0.41cm}
		\begin{minipage}{4.5cm}
			\includegraphics[height=5cm,width=5.5cm]{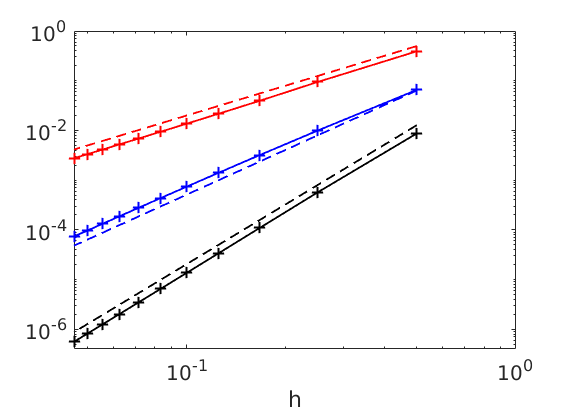}
			\centering
			\footnotesize \hspace*{-0.8cm} (c)  $\n{\f{p}- \f{p}_h}_{\t{L}^2}.$
		\end{minipage}
		\caption{The  method with weak symmetry can be applied to the multi-patch setting and the convergence behavior is similar to the single-patch case. }
		\label{Fig4}
	\end{figure}

	\begin{figure}[H]
		\hspace{1cm}
		\begin{minipage}{4.5cm}
			\includegraphics[height=5cm,width=6.5cm]{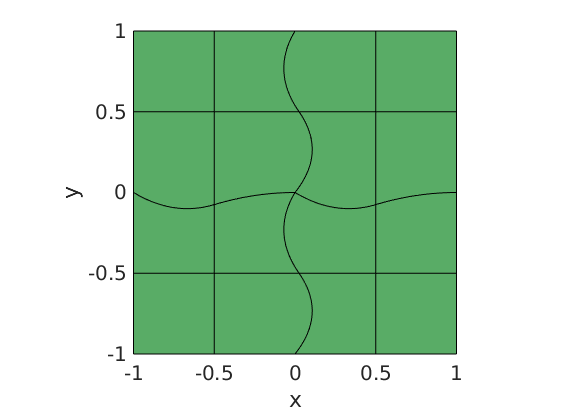}
		\end{minipage}
		\hspace{1.8cm}
		\begin{minipage}{4.5cm}
			\includegraphics[height=5cm,width=6.5cm]{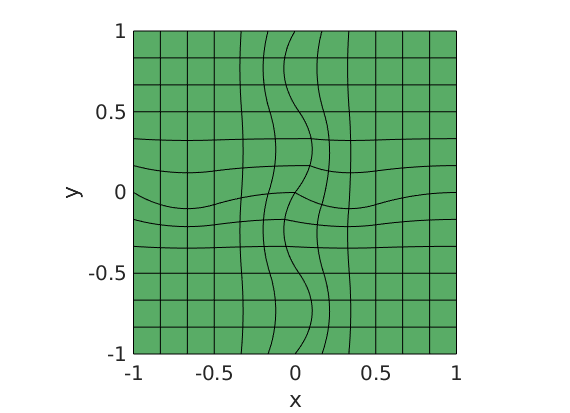}
		\end{minipage}
		\caption{On the left we have the initial mesh for our four-patch example ($h=1/2$), where the patch interfaces are given by the curved lines. On the right we see the mesh for $h=1/6$.  }
		\label{Figmp}
	\end{figure}

	\subsection{Quasi-Incompressible case}
	Here we want to study a stability test with respect to the  Lamé coefficient $\lambda$. Therefore, we apply the \emph{Quasi-Incompressible Regime} example from \cite[Section 7.2]{Rettung} to the single-patch and multi-patch parametrizations from above. More precisely, we  set $\mu =1$ and $\lambda = 10^{10}$ and  adapt the source functions such that we obtain the exact solution 
	\begin{align*}
		u_1 &= \big( \cos (2\pi x)-1\big)\sin(2\pi y) + \frac{ \sin(\pi x) \sin(\pi y)}{1+ \lambda},\\
		u_2 &=  \big( 1-\cos (\pi y)\big)\sin(2\pi x) + \frac{2 \sin(\pi x) \sin(\pi y)}{1+ \lambda}.
	\end{align*}
	
	On the one hand we compute numerical solutions for the domain in Fig. \ref{Fig1}, together with mixed boundary conditions. Apart  from the left edge on the $y$-axis we have traction boundary conditions.  And on the other hand we utilize the four-patch parametrization (see Fig. \ref{Figmp}) with zero displacement BC but equal source function $\f{f}$.  
	See Fig. \ref{Fig7} for the single-patch results and Fig. \ref{Fig6} for the multi-patch domain. For both tests we use regularity $r=0$. We observe an convergence comparable with the compressible regime ($\lambda =2, \mu = 1$) shown in \ref{Fig6} (d) - (f) and \ref{Fig7} (d) - (f). 
		The   single-patch parametrization has the underlying Jacobian 
	$$\f{J}(\zeta_1,\zeta_2)= \begin{bmatrix}
		1 & 0 \\ 2\zeta_1 & 1
	\end{bmatrix}, $$
	and thus from the theoretical point we should have stability w.r.t. $\lambda$ for all  $p\geq 2 $; see Theorem  \ref{Theorem_2D_stability} and Remark \ref{remark_stability}. 
	In the multi-patch example the patches are parameterized by means of globally $C^1$-regular bi-quadratic splines. Consequently, the convergence is stable for $p \geq 2$, too. Our numerical study confirms this point.

	\begin{figure}
		\begin{minipage}{4.5cm}
			\includegraphics[height=5cm,width=5.5cm]{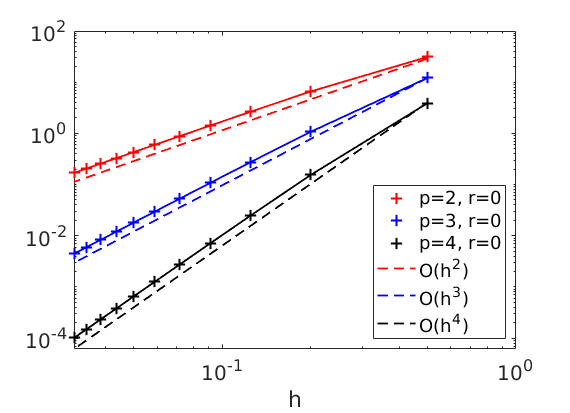}
			\footnotesize \centering \hspace*{-0.8cm}  (a)  $\n{\f{\sigma}- \f{\sigma}_h}_{\t{H}(\d)}.$
		\end{minipage}
		\hspace{0.41cm}
		\begin{minipage}{4.5cm}
			\includegraphics[height=5cm,width=5.5cm]{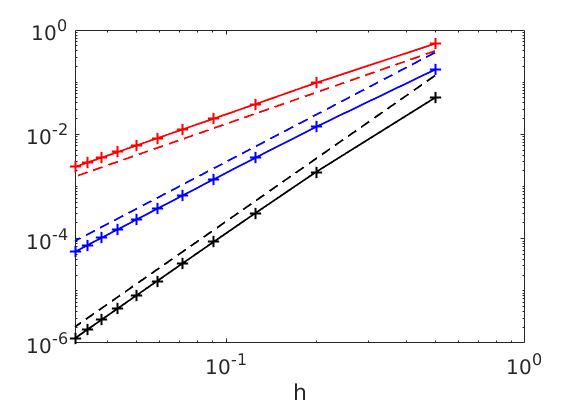}
			\footnotesize \centering \hspace*{-0.8cm}  (b)  $\n{\f{u}- \f{u}_h}_{\f{L}^2}.$
		\end{minipage}
		\hspace{0.41cm}
		\begin{minipage}{4.5cm}
			\includegraphics[height=5cm,width=5.5cm]{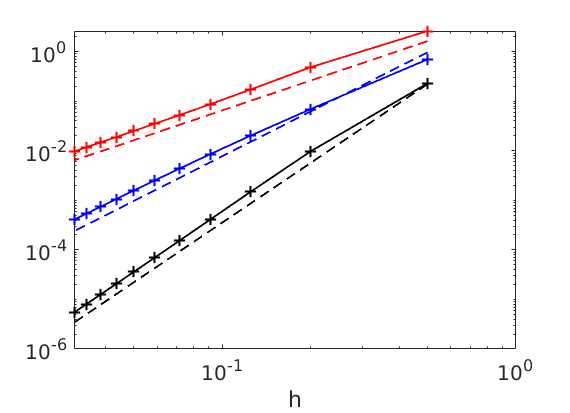}
			\footnotesize \centering \hspace*{-0.8cm}  (c)  $\n{\f{p}- \f{p}_h}_{\t{L}^2}.$
		\end{minipage}

		\begin{minipage}{4.5cm}
			\includegraphics[height=5cm,width=5.5cm]{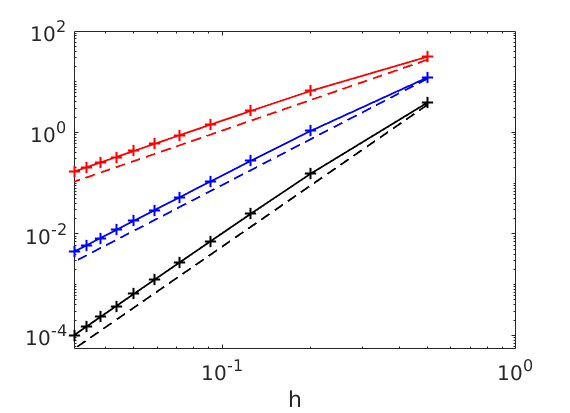}
			\footnotesize \centering \hspace*{-0.8cm}  (d) $\n{\f{\sigma}- \f{\sigma}_h}_{\t{H}(\d)}.$
		\end{minipage}
		\hspace{0.41cm}
		\begin{minipage}{4.5cm}
			\includegraphics[height=5cm,width=5.5cm]{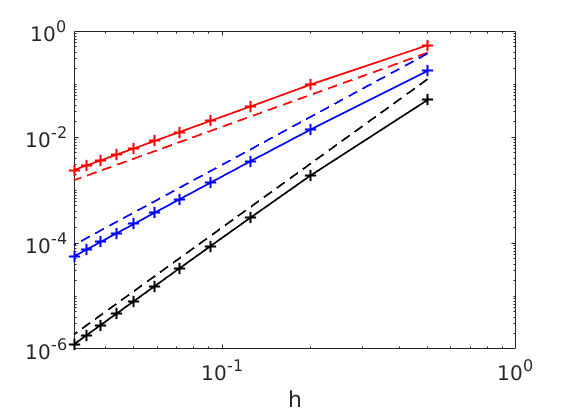}
			\footnotesize \centering \hspace*{-0.8cm}  (e)  $\n{\f{u}- \f{u}_h}_{\f{L}^2}.$
		\end{minipage}
		\hspace{0.41cm}
		\begin{minipage}{4.5cm}
			\includegraphics[height=5cm,width=5.5cm]{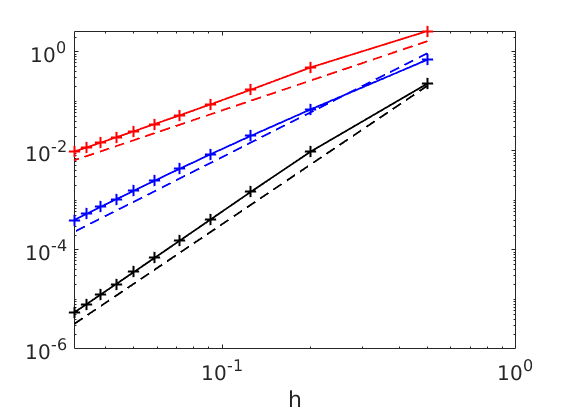}
			\footnotesize \centering  \hspace*{-0.8cm}  (f)  $\n{\f{p}- \f{p}_h}_{\t{L}^2}.$
		\end{minipage}
		
		\caption{ In the top row we see the convergence plots for the domain in Fig. (a) and a \emph{Quasi-Incompressible Regime} example similar to \cite[Section 7.2]{Rettung};($\lambda = 10^{10}, \mu = 1$). In the bottom row we have the results for the compressible case with $\lambda=2, \mu =1$. }
		\label{Fig7}
	\end{figure}

	\begin{figure}
		\begin{minipage}{4.5cm}
			\includegraphics[height=5cm,width=5.5cm]{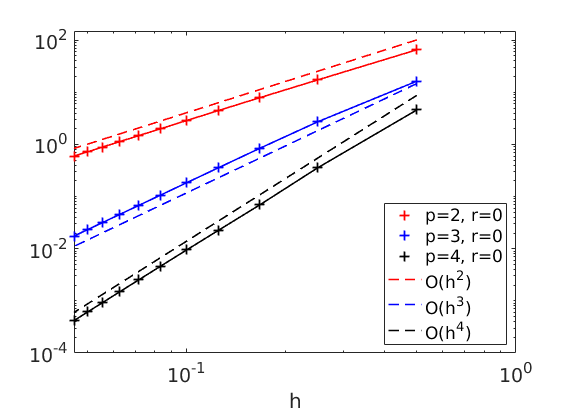}
			\footnotesize \centering \hspace*{-0.8cm} (a)    $\n{\f{\sigma}- \f{\sigma}_h}_{\t{H}(\d)}.$
		\end{minipage}
		\hspace{0.41cm}
		\begin{minipage}{4.5cm}
			\includegraphics[height=5cm,width=5.5cm]{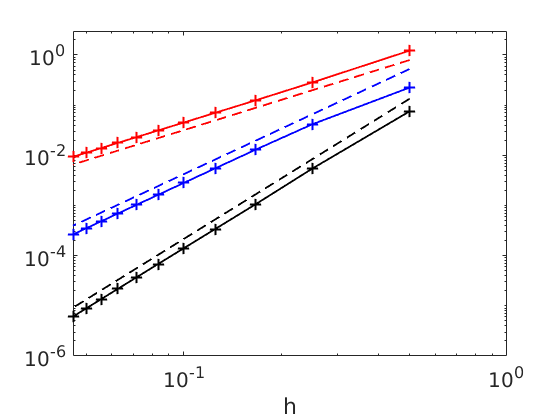}
			\footnotesize  \centering \hspace*{-0.8cm} (b)    $\n{\f{u}- \f{u}_h}_{\f{L}^2}.$
		\end{minipage}
		\hspace{0.41cm}
		\begin{minipage}{4.5cm}
			\includegraphics[height=5cm,width=5.5cm]{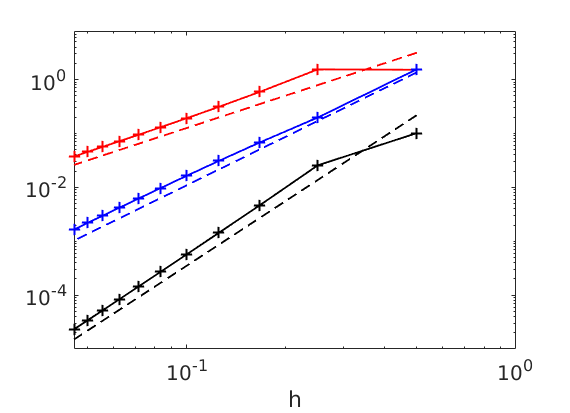}
			\footnotesize  \centering \hspace*{-0.8cm} (c)    $\n{\f{p}- \f{p}_h}_{\t{L}^2}.$
		\end{minipage}

		\begin{minipage}{4.5cm}
			\includegraphics[height=5cm,width=5.5cm]{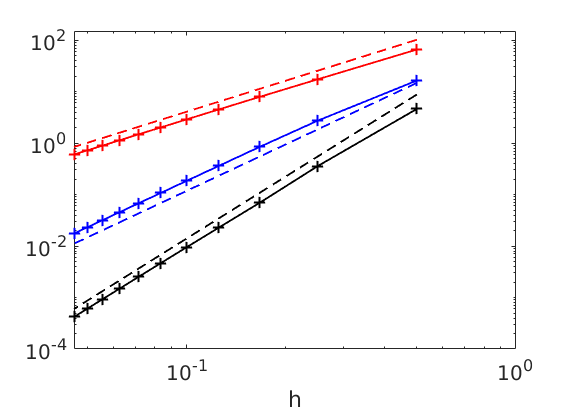}
			\footnotesize  \centering \hspace*{-0.8cm} (d)    $\n{\f{\sigma}- \f{\sigma}_h}_{\t{H}(\d)}.$
		\end{minipage}
		\hspace{0.41cm}
		\begin{minipage}{4.5cm}
			\includegraphics[height=5cm,width=5.5cm]{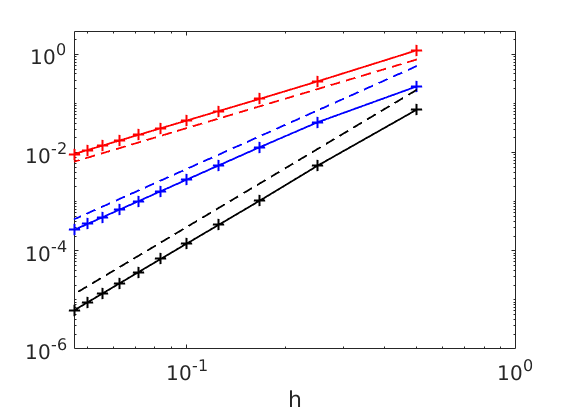}
			\footnotesize  \centering \hspace*{-0.8cm} (e)   $\n{\f{u}- \f{u}_h}_{\f{L}^2}.$
		\end{minipage}
		\hspace{0.41cm}
		\begin{minipage}{4.5cm}
			\includegraphics[height=5cm,width=5.5cm]{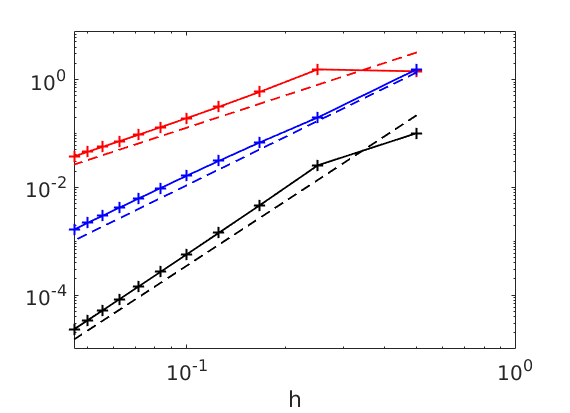}
			\footnotesize  \centering \hspace*{-0.8cm} (f)   $\n{\f{p}- \f{p}_h}_{\t{L}^2}.$
		\end{minipage}
		\caption{Above we have the error decrease for the \emph{Quasi-Incompressible Regime} example  in the four-patch square domain; see Fig \ref{Figmp}. In the top row we see the results for Lam\'{e} coefficients $\lambda = 10^{10}, \ \mu = 1$ and in the bottom we show for comparison the errors for the choice $\lambda = 2, \ \mu = 1$. The similarity of both rows indicate the stability w.r.t. to $\lambda$ also in the multi-patch framework.}
		\label{Fig6}
	\end{figure}

	\subsection{Loaded disk}
	\label{sec_disk}
We show an example for which we combine mixed boundary conditions   with another multi-patch situation.   To be more precise, we have a disk domain composed of $5$ patches, where on the top half $\Gamma_t $ of the disk  a traction force $\f{t}_n = (0,t_y)^T, \ \ t_y(x) = -0.1((2-x)^2(2+x)^2) $ is applied and on the lower boundary half $\Gamma_D$ we have a zero displacement condition; see  Fig \ref{Fig8} (a). For  $p=3,  r=0,  \lambda = 100, \mu =1 $  the  displacement components are given in Fig.  \ref{Fig8} (b) - (c).
	Strictly speaking, 
	 the parametrization does not meet Assumption \ref{assumption_multi-patch}, since there is an interior patch. Nevertheless, 
	   we want to demonstrate here that the weak symmetry method can be used also for more general situations.
	
	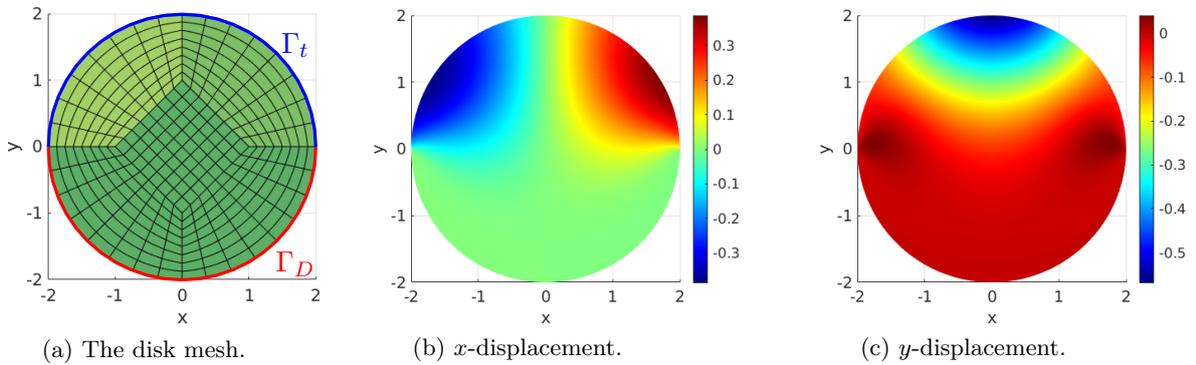
\begin{figure}
		\hspace{-1cm}
		\begin{minipage}{5cm}
			\begin{tikzpicture}
				\node[inner sep=0pt] (ring) at (0,0)
				{\includegraphics[height=4.49cm,width=6.21cm]{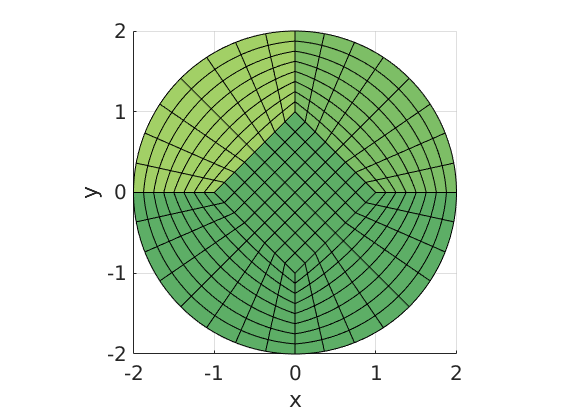}};
				\draw [blue,very thick,domain=0:180] plot ({1.76*cos(\x)+0.11}, {1.76*sin(\x)+0.15});
				\draw [red,very thick,domain=0:-180] plot ({1.76*cos(\x)+0.11}, {1.76*sin(\x)+0.15});
				\node[red] at (1.6,-1.4) {$\Gamma_D$};
				\node[blue] at (1.6,1.5) {$\Gamma_t$};
			\end{tikzpicture}
			\hspace*{0.3cm}	\footnotesize \centering (a) The disk mesh.
		\end{minipage}
		\begin{minipage}{5cm}
			\includegraphics[width=1.23\textwidth]{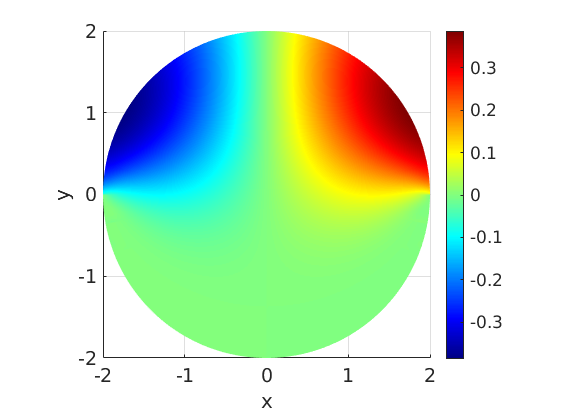}
			\footnotesize \centering  (b) $x$-displacement.
		\end{minipage}
		\hspace{0.61cm}
		\begin{minipage}{5cm}
			\includegraphics[width=1.23\textwidth]{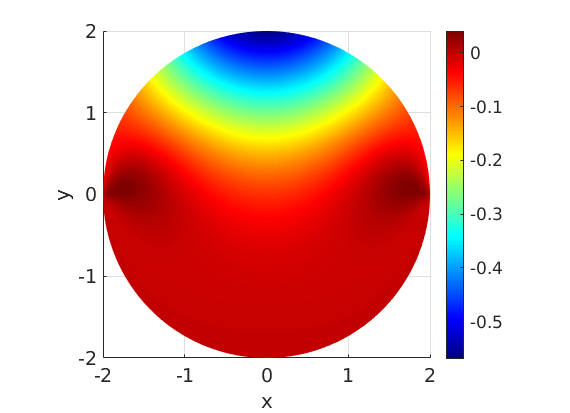}
			\footnotesize \centering  (c)  $y$-displacement.
		\end{minipage}
		\caption{ On $\Gamma_t$ a traction force is applied leading to a deformation. Utilizing the benefits of IGA we can handle a variety of geometries such as this disk shape.}
		\label{Fig8}
	\end{figure}

		\begin{figure}[h!]
		\centering
		\begin{minipage}{5cm}
			\includegraphics[width=1.23\textwidth]{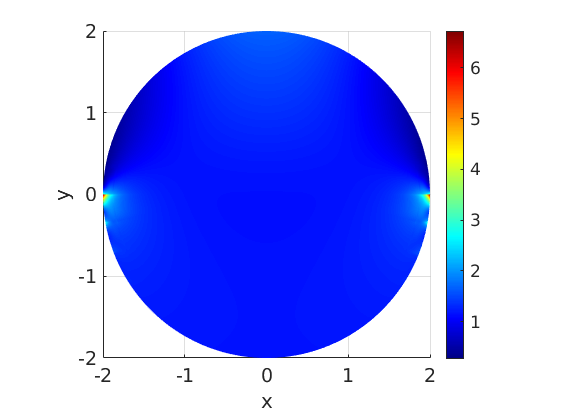}
			\hspace*{0.6cm}	\footnotesize \centering (a)  The stress magnitudes \hspace*{1.4cm} for the multi-patch disk.
		\end{minipage}
		\hspace{1.2cm}
		\begin{minipage}{7cm}
			\includegraphics[width=1\textwidth]{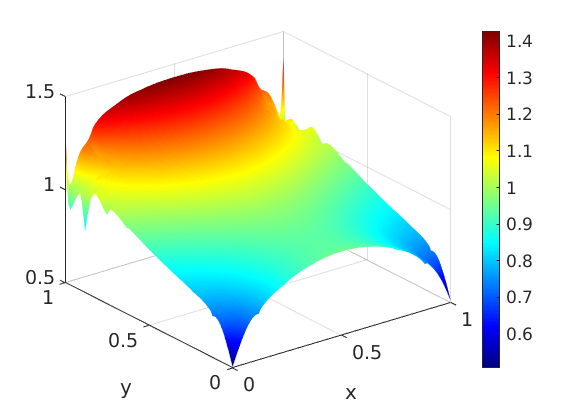}
			\footnotesize \hspace*{1.8cm} (b) Stress magnitudes.
		\end{minipage}

		\begin{minipage}{5cm}
			\includegraphics[width=1.26\textwidth]{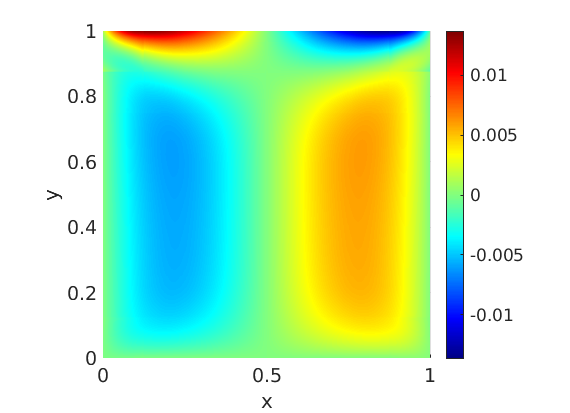}
			\footnotesize \centering  (c)  $x$-displacement.
		\end{minipage}
		\hspace{1.9cm}
		\begin{minipage}{5cm}
			\includegraphics[width=1.26\textwidth]{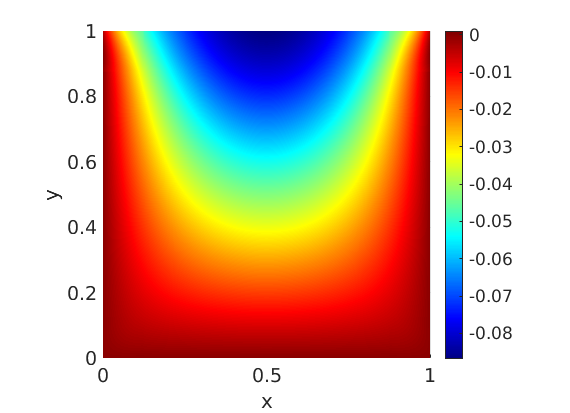}
			\footnotesize \centering  (d)  $y$-displacement.
		\end{minipage}
		\caption{In case of mixed boundary conditions we observe local stress oscillations near points in  $\overline{\Gamma_D} \cap \overline{\Gamma_t}$. In the top left figure  one can note the large stress magnitudes near the points $(-2,0)$ and $(0,2)$. If we have a square body with a uniform traction force applied on the top edge $(0,1)-(1,1)$ we see again local oscillations; compare the top right figure. As displayed in the bottom row, we still get reasonable displacements despite stress overshoots. }
		\label{Fig10}
	\end{figure}

	\subsection{Local stress oscillations}
	\label{sec_stress_o}
	One point we observed in numerical experiments and we want to mention here is the appearance of local overshoots and oscillations of the stress components in case of mixed boundary conditions. For example, if we come back to the multi-patch disk case from above (Section \ref{sec_disk}) with applied traction force, then we obtain for the  stress magnitudes, meaning the values $\sqrt{\sigma_{11}^2 +  2\sigma_{12}^2 + \sigma_{22}^2}$, the numerical solution in Fig. \ref{Fig10} (a).  Near the points $(-2,0)$ and $(2,0)$ where the  boundary part $\Gamma_D$ meets the traction boundary $\Gamma_t$ one can see  comparatively large stress values and an oscillatory behavior. In the general,  this  behavior can be observed in case of mixed-boundary conditions independent of the parametrization regularity, single- or multi-patch situation  and independent of the global spline regularity. 	
	To demonstrate this, we look at the  simplest single-patch square geometry, meaning $\f{F}= \textup{id}$, and enforce at the top edge a constant load $\f{t}_n = (-1,0)^T$. For polynomial degree $p=3$, $r=0$, a uniform mesh ($h=1/8$) and $\lambda =10, \ \mu =1$ we get the stress magnitudes in  Fig. \ref{Fig10} (b). In particular, we obtain again  local stress oscillations near the points $(0,1)$ and $(1,1)$. If we decrease $h$ or increase $p$ this stress outliers move closer to the mentioned boundary points where $\Gamma_D$ and $\Gamma_t$ meet. 
	One reason for the stress issue might be changed role of  boundary conditions for the mixed weak form compared to the classical primal weak elasticity formulation. In our experiments the traction BCs are explicitly enforced by prescribing  the respective stress components, but the displacement boundary conditions are only implemented in a weak sense. Thus, at points where Dirchlet and traction sides meet we specify the traction force while at the same time the displacement is enforced weakly on $\Gamma_D$ according to $\f{u}_D$. As a result  it seems that a non-smooth behavior of the solution  is triggered near such points. Nevertheless, even if we obtain such local stress wiggles  within numerical computations, we still get adequate solutions for the displacement; compare Fig. \ref{Fig10} (c) -(d). 
	However,  we do not have a strict proof or explanation to clarify this  issue with the stresses  at hand. Hence, a more detailed study of the stress  solutions  in the situation of mixed boundary conditions is advisable.

	\section{Conclusion}
	\label{sec_conc}
	
	In this article we defined discrete spaces utilizing B-splines for the approximation of the mixed weak form of planar linear elasticity with weakly imposed symmetry as well as for the  mixed formulation implied by the Hellinger-Reissner functional. The latter  requires symmetric stress fields and the  derivation of  suitable finite-dimensional test spaces needed more effort. For both approaches our choice of discretization spaces satisfies the Brezzi inf-sup condition and we obtain well-posed saddle-point problems together with a quasi-optimality estimate. Moreover, we showed the convergence of the methods at least if the underlying geometry and the exact solution are regular enough. And if we use B-spline parametrizations $\f{F} \in (S_{p,p}^{r,r})^2$, then the corresponding  inf-sup  constant for the weak symmetry case can be chosen independent of  $\lambda $   which indicates stability for the quasi-incompressible regime. Several numerical examples within the scope of weakly imposed symmetry  confirmed the  theoretical  error estimate.  However, for the ansatz with strong symmetry a numerical study is missing and might be a starting point of a further paper. Unfortunately, the space definitions in Section \ref{section_strong} are quite involved and thus symmetry preserving computations might be expensive. Another point that deserves a clarification is the issue of local stress oscillations in case of mixed boundary conditions; compare Section \ref{sec_stress_o}.  The main advantage of the proposed methods is the exploitation of results from IGA.  This is why we can  consider computational domains with curved boundaries which is not directly possible if classical FEM is used.

	\section{Appendix}
	
	For the sake of clarity  we add here the  missing proofs for the Lemmas \ref{Lemmma_strong_sym_div}, \ref{Lemma_invertebility_Gamma_2_1} and \ref{Lemma_invertebility_Gamma_2_2}.
	
	\subsection{Proof of Lemma \ref{Lemmma_strong_sym_div} \textmd{(Divergence compatibility of $\mathcal{Y}_{2,\Gamma_1}$)}}
	\label{Proof_1_appendix}	
	\begin{proof}
		Because of Lemma \ref{Lemma_strong_sym_L2} we can assume for this proof that $\f{S} \in \t{H}_{\Gamma_1}(\o,\d,\mathbb{S}) \cap C^1(\overline{\o})$. The general statement follows by a density argument. Firstly it is $$\mathcal{Y}^s_{2,\Gamma_1} (\f{S}) =   \underbrace{\tilde{\f{J}} (\f{S} \circ \f{F}) \tilde{\f{J}}^{T}}_{ \eqqcolon(L1,1)} + \underbrace{ \widehat{\textup{Airy}}({F}_n) \cdot \int_{0}^{\zeta_1} h_n  \ d1}_{\eqqcolon(L1,2)},$$ where we set $h_n \coloneqq    (-1)^{n+1}  \  \tilde{{J}}_{2l}  \ (S_{\sigma(n)l} \circ \f{F})$. And let  us define $\hat{\f{S}} \coloneqq \f{S} \circ \f{F}$. With the introduced notation we have  for the $m$-th entry of $\widehat{\nabla} \cdot (L1,1)$:
		\begin{align*}
			[\widehat{\nabla} \cdot (L1,1)]_m &=  \widehat{\nabla} \cdot \big(  \J  \cdot \hat{\f{S}} \cdot (\J_m)^T\big) \\
			&=	\widehat{\nabla} \cdot \big( \j_{mi} \ \J \cdot \hat{\f{S}}^i\big) \\
			&= (\widehat{\nabla}\j_{mi})^T \cdot \big[ \J \cdot \hat{\f{S}}_i\big] + \j_{mi} \ \widehat{\nabla} \cdot \big( \J \cdot \hat{\f{S}}_i \big).
		\end{align*}
		Since $\J = \textup{det}(\f{J}) \ \f{J}^{-1} $ we get with \eqref{eq_trans_compa}:
		\begin{equation}
			\label{eq_Lemma-div_compatibel_auxiliary_1}
			[\widehat{\nabla} \cdot (L1,1)]_m = \underbrace{\h_k \j_{mi} \cdot \j_{kl}  \cdot \hat{S}_{il}}_{\eqqcolon(L1,3)} + \big[\textup{det}(\f{J}) \ \J \cdot \big(\nabla \cdot \S \big) \circ \f{F} \big]_m.
		\end{equation}
		If  we keep the boundary condition in mind then the term $(L1,3)$ can be written as 
		\begin{align}
			\label{eq_Lemma_BC_line1}
			(L1,3) &= \h_1 \j_{mi} \cdot \j_{1l}  \cdot \hat{S}_{il}+ \h_2 \j_{mi} \cdot \j_{2l}  \cdot \hat{S}_{il}  \\ 	\label{eq_Lemma_BC_line2}
			&= \h_1 \j_{mi} \cdot  \int_{0}^{\zeta_1} \h_1[\j_{1l}  \cdot \hat{S}_{il}]  \ d1+ \h_2 \j_{mi} \cdot \j_{2l}  \cdot \hat{S}_{il} .
		\end{align}
		Note, by assumption $ \S  \cdot \f{n}= \f{0}$ on $\Gamma_1$  which implies $
		(1,0) \cdot (\J \cdot (\S_j\circ \f{F}))   = {0} , \ j \in \{1,2\} $  on $\hat{{\Gamma}}_1$. And thus $ \j_{1i} \cdot \hat{S}_{ij}(0, \zeta_2)=0 , \forall j $. \\
		On the other hand we have
		\begin{align*}
			[\widehat{\nabla} \cdot (L1,2)]_m &=    (\widehat{\textup{Airy}}(F_n))_m \cdot \Big(\widehat{\nabla}\int_{0}^{\zeta_1}h_n \ d1\Big) + \Big(\int_{0}^{\zeta_1}h_n \ d1 \Big)\cdot \big(  \underbrace{\widehat{\nabla}  \cdot \widehat{\textup{Airy}}(F_n)}_{=0} \big) \\
			&=  h_n \cdot (\widehat{\textup{Airy}}(F_n))_{m1} +  \Big(\int_{0}^{\zeta_1} \h_2h_n  \ d1 \Big) \cdot (\widehat{\textup{Airy}}(F_n))_{m2} \\
			&=    \underbrace{\tilde{{J}}_{2l}  \cdot \hat{S}_{2l} \cdot (\widehat{\textup{Airy}}(F_1))_{m1}-  \tilde{{J}}_{2l} \cdot  \hat{S}_{1l} \cdot (\widehat{\textup{Airy}}(F_2))_{m1}}_{  \eqqcolon (L1,4) }  \\
			& \ \ \ + \int_{0}^{\zeta_1} \h_2[\tilde{{J}}_{2l}  \cdot \hat{S}_{2l}]  \ d1 \cdot (\widehat{\textup{Airy}}(F_1))_{m2} - \int_{0}^{\zeta_1} \h_2[\tilde{{J}}_{2l}  \cdot \hat{S}_{1l}]  \ d1 \cdot (\widehat{\textup{Airy}}(F_2))_{m2} . 
		\end{align*}
		Hence, for $m=1$ we obtain for $(L1,4)$ from the second last line in detail:
		\begin{align*}
			[(L1,4)]_1 &=     \tilde{{J}}_{21}  \cdot \hat{S}_{21} \cdot (\widehat{\textup{Airy}}(F_1))_{11}-  \tilde{{J}}_{21} \cdot  \hat{S}_{11} \cdot (\widehat{\textup{Airy}}(F_2))_{11} +   \tilde{{J}}_{22}  \cdot \hat{S}_{22} \cdot (\widehat{\textup{Airy}}(F_1))_{11} \\ &  \ \ \ -  \tilde{{J}}_{22} \cdot  \hat{S}_{12} \cdot (\widehat{\textup{Airy}}(F_2))_{11} \\
			&= \tilde{{J}}_{21}  \cdot \hat{S}_{21} \cdot \h_2J_{12}-  \tilde{{J}}_{21} \cdot  \hat{S}_{11} \cdot \h_2J_{22} +   \tilde{{J}}_{22}  \cdot \hat{S}_{22} \cdot \h_2J_{12} \\ & \ \ \ -  \tilde{{J}}_{22} \cdot  \hat{S}_{12} \cdot \h_2J_{22} \\
			&=-\tilde{{J}}_{21}  \cdot \hat{S}_{21} \cdot \h_2\j_{12}-  \tilde{{J}}_{21} \cdot  \hat{S}_{11} \cdot \h_2\j_{11} -   \tilde{{J}}_{22}  \cdot \hat{S}_{22} \cdot \h_2\j_{12} \\ & \ \ \ -  \tilde{{J}}_{22} \cdot  \hat{S}_{12} \cdot \h_2\j_{11} \\
			&= -\h_2\j_{1i}  \cdot \hat{S}_{1i}  \cdot \j_{21}-  \h_2\j_{1i}  \cdot \hat{S}_{2i} \cdot  \j_{22} \\
			&= - \h_2\j_{1i}  \cdot \hat{S}_{il}  \cdot \j_{2l} \ .
		\end{align*}
		And for $m=2$ it is
		\begin{align*}
			[(L1,4)]_2 
			&= -\tilde{{J}}_{21}  \cdot \hat{S}_{21} \cdot \h_2J_{11}+  \tilde{{J}}_{21} \cdot  \hat{S}_{11} \cdot \h_2J_{21} -   \tilde{{J}}_{22}  \cdot \hat{S}_{22} \cdot \h_2J_{11} \\ &  \ \ \ +  \tilde{{J}}_{22} \cdot  \hat{S}_{12} \cdot \h_2J_{21}  \\
			&=-\tilde{{J}}_{21}  \cdot \hat{S}_{21} \cdot \h_2\j_{22}-  \tilde{{J}}_{21} \cdot  \hat{S}_{11} \cdot \h_2\j_{21} -   \tilde{{J}}_{22}  \cdot \hat{S}_{22} \cdot \h_2\j_{22} \\ &  \ \ \ -  \tilde{{J}}_{22} \cdot  \hat{S}_{12} \cdot \h_2\j_{21}  \\
			&= - \h_2\j_{2i}  \cdot \hat{S}_{il}  \cdot \j_{2l}.
		\end{align*}
		Comparing  $[(L1,4)]_m$ and  \eqref{eq_Lemma_BC_line2} we see that if we add \eqref{eq_Lemma_BC_line2} and $[\widehat{\nabla} \cdot (L1,2)]_m$ then the terms without integrals vanish and we get
		\begin{align*}
			[\widehat{\nabla} &\cdot (L1,2)]_m + [ (L1,3)]_m \\
			& = \int_{0}^{\zeta_1} \h_2[\tilde{{J}}_{2l}  \cdot \hat{S}_{2l}]  \ d1 \cdot (\widehat{\textup{Airy}}(F_1))_{m2} - \int_{0}^{\zeta_1} \h_2[\tilde{{J}}_{2l}  \cdot \hat{S}_{1l}] \ d1  \cdot (\widehat{\textup{Airy}}(F_2))_{m2}  \\
			& \ \ \ \ +\h_1 \j_{mi} \cdot  \int_{0}^{\zeta_1} \h_1[\j_{1l}  \cdot \hat{S}_{il}]  \ d1 .
		\end{align*}
		For $m=1$ we get for the last two lines 
		\begin{align*}
			[\widehat{\nabla} &\cdot (L1,2)]_1 + [ (L1,3)]_1 \\
			& =-\int_{0}^{\zeta_1} \h_2\tilde{{J}}_{2l}   \cdot \hat{S}_{2l}  + \tilde{{J}}_{2l}   \cdot \h_2\hat{S}_{2l} \ d1 \cdot \h_1J_{12} \\ &
			\ \  \  + \int_{0}^{\zeta_1} \h_2\tilde{{J}}_{2l}  \cdot \hat{S}_{1l} +  \tilde{{J}}_{2l}  \cdot \h_2\hat{S}_{1l}   \ d1 \cdot \h_1J_{22}  \\
			& \ \  \ +  \int_{0}^{\zeta_1} \h_1\j_{1l}  \cdot \hat{S}_{il}  \ d1  \cdot \h_1 \j_{1i} + \int_{0}^{\zeta_1} \j_{1l}  \cdot \h_1 \hat{S}_{il}  \ d1  \cdot \h_1 \j_{1i} \\
			&=  \underbrace{-\int_{0}^{\zeta_1} \h_2\tilde{{J}}_{2l}   \cdot \hat{S}_{2l}    \ d1 \cdot \h_1J_{12} 
				+ \int_{0}^{\zeta_1} \h_2\tilde{{J}}_{2l}  \cdot \hat{S}_{1l}   \ d1 \cdot \h_1J_{22}  +  \int_{0}^{\zeta_1} \h_1\j_{1l}  \cdot \hat{S}_{il}  \ d1  \cdot \h_1 \j_{1i} }_{\eqqcolon(L1,5)} \\
			&  \ \ \  \underbrace{ -\int_{0}^{\zeta_1} \tilde{{J}}_{2l}   \cdot \h_2\hat{S}_{2l}    \ d1 \cdot \h_1J_{12} 
				+ \int_{0}^{\zeta_1} \tilde{{J}}_{2l}  \cdot \h_2\hat{S}_{1l}   \ d1 \cdot \h_1J_{22}  +  \int_{0}^{\zeta_1} \j_{1l}  \cdot \h_1\hat{S}_{il}  \ d1  \cdot \h_1 \j_{1i}}_{\eqqcolon (L1,6)}
		\end{align*}
		With the relation $\h_1\j_{1l}=-\h_2\j_{2l}$ we get that $(L1,5)$ is zero since 
		\begin{align*}
			-& \int_{0}^{\zeta_1} \h_2\tilde{{J}}_{2l}   \cdot \hat{S}_{2l}   \ d1  \cdot \h_1J_{12} 
			+ \int_{0}^{\zeta_1} \h_2\tilde{{J}}_{2l}  \cdot \hat{S}_{1l}    \ d1 \cdot \h_1J_{22}  + \int_{0}^{\zeta_1} \h_1\j_{1l}  \cdot \hat{S}_{il} \ d1 \cdot \h_1 \j_{1i} \\
			& = \int_{0}^{\zeta_1}\h_2\tilde{{J}}_{2l}   \cdot \hat{S}_{2l}     \ d1\cdot \h_1\j_{12} 
			+ \int_{0}^{\zeta_1} \h_2\tilde{{J}}_{2l}  \cdot \hat{S}_{1l}    \ d1 \cdot \h_1\j_{11}  + \int_{0}^{\zeta_1} \h_1\j_{1l}  \cdot \hat{S}_{il}   \ d1  \cdot \h_1 \j_{1i}\\
			&= \int_{0}^{\zeta_1}\h_2\tilde{{J}}_{2l}   \cdot \hat{S}_{il}     \ d1 \cdot \h_1\j_{1i} 
			+  \int_{0}^{\zeta_1} \h_1\j_{1l}  \cdot \hat{S}_{il}   \ d1  \cdot \h_1 \j_{1i}\\
			&=0.
		\end{align*}
		Consequently, 
		\begin{align}
			[\widehat{\nabla} &\cdot (L1,2)]_1 + [ (L1,3)]_1 = (L1,6) \nonumber\\&=  \int_{0}^{\zeta_1} \tilde{{J}}_{2l}   \cdot \h_2\hat{S}_{2l}    \ d1 \cdot \h_1\j_{12} 
			+ \int_{0}^{\zeta_1} \tilde{{J}}_{2l}  \cdot \h_2\hat{S}_{1l}   \ d1 \cdot \h_1\j_{11}  +  \int_{0}^{\zeta_1} \j_{1l}  \cdot \h_1\hat{S}_{il}  \ d1  \cdot \h_1 \j_{1i}  \nonumber \\
			&= \int_{0}^{\zeta_1} \tilde{{J}}_{2l}   \cdot \h_2\hat{S}_{il}    \ d1 \cdot \h_1\j_{1i} 
			+  \int_{0}^{\zeta_1} \j_{1l}  \cdot \h_1\hat{S}_{il}  \ d1  \cdot \h_1 \j_{1i}   \nonumber\\
			&= \int_{0}^{\zeta_1} \tilde{{J}}_{kl}   \cdot \h_k\hat{S}_{il}    \ d1 \cdot \h_1\j_{1i}  = \int_{0}^{\zeta_1} \tilde{{J}}_{kl}   \cdot [\partial_j{S}_{il} \circ \f{F}] \cdot J_{jk}   \ d1 \cdot \h_1\j_{1i} \nonumber \\
			&= \int_{0}^{\zeta_1} \textup{det}(\f{J}) \cdot{{J}}_{kl}^{-1}   \cdot [\partial_j{S}_{il} \circ \f{F}] \cdot J_{jk}   \ d1 \cdot \h_1\j_{1i} \nonumber \\
			& = \int_{0}^{\zeta_1} \delta_{jl} \cdot  \textup{det}(\f{J})  \cdot [\partial_j{S}_{il}\circ \f{F}]   \ d1 \cdot \h_1\j_{1i} \nonumber \\
			&= \int_{0}^{\zeta_1}  \textup{det}(\f{J})   \cdot [\partial_l{S}_{il} \circ \f{F}]   \ d1 \cdot \h_1\j_{1i} = \int_{0}^{\zeta_1}  \textup{det}(\f{J})   \cdot [(\nabla \cdot \S) \circ \f{F}]_i   \ d1 \cdot \h_1\j_{1i}.
			\label{eq_gamma_1_div_term2_1}
		\end{align}
		In the fourth row above we used the chain rule and further we write $\delta_{ij}$ for the Kronecker delta. And one observes the relation $\textup{det}(\f{J}) \cdot{{J}}_{kl}^{-1}= \j_{kl} \circ \f{F}$.\\
		For reasons of completeness we show the proof steps also for the second entry, i.e. for  $m=2$, although the approach is the same:
		\begin{align*}
			[\widehat{\nabla} &\cdot (L1,2)]_2 + [ (L1,3)]_2 \\
			& =\int_{0}^{\zeta_1} \h_2\tilde{{J}}_{2l}   \cdot \hat{S}_{2l}  + \tilde{{J}}_{2l}   \cdot \h_2\hat{S}_{2l} \ d1 \cdot \h_1J_{11} \\ &
			\ \  \ - \int_{0}^{\zeta_1} \h_2\tilde{{J}}_{2l}  \cdot \hat{S}_{1l} +  \tilde{{J}}_{2l}  \cdot \h_2\hat{S}_{1l}   \ d1 \cdot \h_1J_{21}  \\
			& \ \  \ +  \int_{0}^{\zeta_1} \h_1\j_{1l}  \cdot \hat{S}_{il}  \ d1  \cdot \h_1 \j_{2i} + \int_{0}^{\zeta_1} \j_{1l}  \cdot \h_1 \hat{S}_{il}  \ d1  \cdot \h_1 \j_{2i} \\
			&=  \underbrace{\int_{0}^{\zeta_1} \h_2\tilde{{J}}_{2l}   \cdot \hat{S}_{2l}    \ d1 \cdot \h_1\j_{22} 
				+ \int_{0}^{\zeta_1} \h_2\tilde{{J}}_{2l}  \cdot \hat{S}_{1l}   \ d1 \cdot \h_1\j_{21}  +  \int_{0}^{\zeta_1} \h_1\j_{1l}  \cdot \hat{S}_{il}  \ d1  \cdot \h_1 \j_{2i} }_{\eqqcolon(L1,7)} \\
			& \ \ + \underbrace{ \int_{0}^{\zeta_1} \tilde{{J}}_{2l}   \cdot \h_2\hat{S}_{2l}    \ d1 \cdot \h_1\j_{22} 
				+ \int_{0}^{\zeta_1} \tilde{{J}}_{2l}  \cdot \h_2\hat{S}_{1l}   \ d1 \cdot \h_1\j_{21}  +  \int_{0}^{\zeta_1} \j_{1l}  \cdot \h_1\hat{S}_{il}  \ d1  \cdot \h_1 \j_{2i}}_{\eqqcolon (L1,8)}.
		\end{align*}
		The term of the second last row is zero, due to 
		\begin{align*}
			\int_{0}^{\zeta_1} \h_2\tilde{{J}}_{2l}  & \cdot \hat{S}_{2l}    \ d1 \cdot \h_1\j_{22} 
			+ \int_{0}^{\zeta_1} \h_2\tilde{{J}}_{2l}  \cdot \hat{S}_{1l}   \ d1 \cdot \h_1\j_{21}  +  \int_{0}^{\zeta_1} \h_1\j_{1l}  \cdot \hat{S}_{il}  \ d1  \cdot \h_1 \j_{2i} \\
			&= \int_{0}^{\zeta_1} \h_2\tilde{{J}}_{2l}   \cdot \hat{S}_{il}  +   \h_1\j_{1l}  \cdot \hat{S}_{il}    \ d1 \cdot \h_1 \j_{2i} \\
			&= \int_{0}^{\zeta_1} 0  \ d1 \cdot \h_1 \j_{2i} =0.
		\end{align*}
		And the term $(L1,8)$ can be rewritten similar to the term $(L1,6)$ as:
		\begin{align}
			(L1,8)& =\int_{0}^{\zeta_1} \tilde{{J}}_{2l}   \cdot \h_2\hat{S}_{2l}    \ d1 \cdot \h_1\j_{22} 
			+ \int_{0}^{\zeta_1} \tilde{{J}}_{2l}  \cdot \h_2\hat{S}_{1l}   \ d1 \cdot \h_1\j_{21} +  \int_{0}^{\zeta_1} \j_{1l}  \cdot \h_1\hat{S}_{il}  \ d1  \cdot \h_1 \j_{2i}  \nonumber \\
			&= \int_{0}^{\zeta_1} \tilde{{J}}_{kl}   \cdot \h_k\hat{S}_{il}    \ d1 \cdot \h_1\j_{2i}=\int_{0}^{\zeta_1} \tilde{{J}}_{kl}   \cdot J_{jk} \cdot [\partial_j{S}_{il} \circ \f{F} ]    \ d1 \cdot \h_1\j_{2i}  \nonumber\\
			&= \int_{0}^{\zeta_1}  \textup{det}(\f{J})   \cdot [\partial_l{S}_{il} \circ \f{F}]   \ d1 \cdot \h_1\j_{2i} = \int_{0}^{\zeta_1}  \textup{det}(\f{J})   \cdot [(\nabla \cdot \S) \circ \f{F}]_i   \ d1 \cdot \h_1\j_{2i}.
			\label{eq_gamma_1_div_term2_2}
		\end{align}
		Now, in view of the equations \eqref{eq_Lemma-div_compatibel_auxiliary_1}, \eqref{eq_gamma_1_div_term2_1} and \eqref{eq_gamma_1_div_term2_2} we get finally
		\begin{align*}
			[\widehat{\nabla} \cdot (L1,1)]_m &+ [\widehat{\nabla} \cdot (L1,2)]_m \\ &=  \big[\textup{det}(\f{J}) \ \J \cdot \big((\nabla \cdot \S) \circ \f{F} \big) \big]_m +  \h_1\j_{mi} \cdot \int_{0}^{\zeta_1}  \textup{det}(\f{J})   \cdot [(\nabla \cdot \S) \circ \f{F}]_i   \ d1 
		\end{align*}
		which finishes the proof.
	\end{proof}

	\subsection{Proof of Lemma \ref{Lemma_invertebility_Gamma_2_1} \textmd{(Invertibility of $\mathcal{Y}_{2,\Gamma_1}^s$)} }
	\label{Lemma_proof_appendix_2}
	\begin{proof}
		Here we write  $\tilde{\mathcal{Y}}= \tilde{\mathcal{Y}}(\S)$ for the  mapping defined by the expression on  right-hand side of \eqref{eq_inv_gamma_1}. Then we show first that $\tilde{\mathcal{Y}} \circ {\mathcal{Y}}_{2,\Gamma_1}^s= \textup{id}$. One sees easily
		\begin{align*}
			\tilde{\mathcal{Y}} \circ {\mathcal{Y}}_{2,\Gamma_1}^s(\f{S}) &= \f{S} +  \underbrace{\tilde{\f{J}}^{-1} (\widehat{\textup{Airy}}(F_n) \circ \f{F}^{-1}) \tilde{\f{J}}^{-T} \cdot \Big(\int_{0}^{\zeta_1}  h_n  \ d1 \Big) \circ \f{F}^{-1} }_{\eqqcolon (L2,1)}  \\
			&  \ \ \ + \underbrace{\textup{Airy}(F_k^{-1})
				\cdot \Big(\int_{0}^{\zeta_1}   (-1)^{k+1}  \    \ ({\mathcal{Y}}_{2,\Gamma_1}^s(\f{S}))_{2\sigma(k)}  \ d1 \Big) \circ \f{F}^{-1}}_{\eqqcolon(L2,2)},	\end{align*}
		again with the abbreviation $h_n \coloneqq    (-1)^{n+1}  \  \tilde{{J}}_{2l}  \ (S_{\sigma(n)l} \circ \f{F})$.
		Further we set \\ $\tilde{h}_k \coloneqq  (-1)^{k+1}  \    \ ({\mathcal{Y}}_{2,\Gamma_1}^s(\f{S}))_{2\sigma(k)}$ and $\hat{\f{S}} \coloneqq \f{S} \circ \f{F}$.
		
		Now we have to check that the terms $(L2,1)$ and $(L2,2)$ add up to zero. To see this, we use the  chain rule for the Hessian operator. It holds
		\begin{equation*}
			0={\nabla}^2(F_n \circ \f{F}^{-1})= \f{J}^{-T}  (\widehat{\nabla}^2F_n \circ \f{F}^{-1})  \f{J}^{-1} + \nabla^2F_k^{-1} \cdot (J_{nk} \circ \f{F}^{-1}),
		\end{equation*}
		which implies 
		the modified relation
		\begin{equation}
			\label{eq_chain_hessian_2}
			\J^{-1}  (\widehat{\textup{Airy}}(F_n) \circ \f{F}^{-1})  \J^{-T} = - \textup{Airy}(F_k^{-1}) \cdot (J_{nk} \circ \f{F}^{-1}).
		\end{equation}
		And with latter equation we can write 
		\begin{align}
			\label{eq_gamma_1_inv_1}
			(L2,1) = - \textup{Airy}(F_k^{-1}) \cdot (J_{nk} \circ \f{F}^{-1}) \cdot \Big(\int_{0}^{\zeta_1}  h_n  \ d1 \Big) \circ \f{F}^{-1}.
		\end{align}
		Next we look at the $\tilde{h}_k$. More precisely, utilizing the product rule we have
		\begin{align}
			\tilde{h}_1&= ({\mathcal{Y}}_{2,\Gamma_1}^s(\f{S}))_{22} = \j_{2i} \cdot \hat{S}_{ij} \cdot \j_{2j} + (\widehat{\textup{Airy}}(F_n))_{22} \cdot \int_{0}^{\zeta_1}{h}_n  \ d1 \nonumber \\
			&= \j_{2i} \cdot \hat{S}_{ij} \cdot \j_{2j} + \h_1J_{n1} \cdot \int_{0}^{\zeta_1}{h}_n  \ d1 \nonumber \\
			&= \j_{2i} \cdot \hat{S}_{ij} \cdot \j_{2j}- J_{n1} \cdot h_n + \h_1 \big[J_{n1} \cdot \int_{0}^{\zeta_1}  h_n  \ d1 \big]. \label{eq_L2_10}
		\end{align}
		One observes the equality chain
		\begin{align}
			J_{n1} \cdot h_n &= J_{11} \cdot \hat{S}_{2l}\cdot  \j_{2l} - J_{21} \cdot \hat{S}_{1l}\cdot  \j_{2l} \nonumber \\
			&= \j_{22} \cdot \hat{S}_{2l}\cdot  \j_{2l} + \j_{21} \cdot \hat{S}_{1l}\cdot  \j_{2l} \nonumber \\
			&= \j_{2i} \cdot \hat{S}_{il} \cdot \j_{2l} \label{eq_Lem2_11}.
		\end{align}
		Hence,  \eqref{eq_L2_10} and \eqref{eq_Lem2_11} yield
		\begin{align}
			\label{eq_gamma_1_inv_2}
			\int_{0}^{\zeta_1}\tilde{h}_1   \ d1 = J_{n1} \cdot \int_{0}^{\zeta_1}{h}_n      \ d1.
		\end{align}
		And on the other hand for $k=2$ we get 
		\begin{align*}
			\tilde{h}_2&= -({\mathcal{Y}}_{2,\Gamma_1}^s(\f{S}))_{21} =-\j_{2i} \cdot \hat{S}_{ij} \cdot \j_{1j} - (\widehat{\textup{Airy}}(F_n))_{12} \cdot \int_{0}^{\zeta_1}{h}_n  \ d1 \nonumber \\
			&= -\j_{2i} \cdot \hat{S}_{ij} \cdot \j_{1j} + \h_1J_{n2} \cdot \int_{0}^{\zeta_1}{h}_n  \ d1 \\
			&= -\j_{2i} \cdot \hat{S}_{ij} \cdot \j_{1j}- J_{n2} \cdot h_n + \h_1 \big[J_{n2} \cdot \int_{0}^{\zeta_1}  h_n  \ d1 \big]. \nonumber
		\end{align*} 
		Thus it is 
		\begin{align}
			\label{eq_gamma_1_inv_3}
			\int_{0}^{\zeta_1}\tilde{h}_2   \ d1 = J_{n2} \cdot \int_{0}^{\zeta_1}{h}_n      \ d1.
		\end{align}
		The last line is clear due to
		\begin{align*}
			J_{n2} \cdot h_n &= - \j_{12} \cdot  \hat{S}_{2l}\cdot  \j_{2l} - \j_{11}  \cdot \hat{S}_{1l}\cdot  \j_{2l} = -\j_{1i} \cdot  \hat{S}_{il}\cdot  \j_{2l}.
		\end{align*}
		In view of \eqref{eq_gamma_1_inv_1}, \eqref{eq_gamma_1_inv_2} and \eqref{eq_gamma_1_inv_3} we get
		\begin{align*}
			(L2,1)+(L2,2) &= - \textup{Airy}(F_k^{-1}) \cdot (J_{nk} \circ \f{F}^{-1}) \cdot \Big(\int_{0}^{\zeta_1}  h_n  \ d1 \Big) \circ \f{F}^{-1} \\
			& \ \ \ + \textup{Airy}(F_k^{-1}) \cdot (J_{nk} \circ \f{F}^{-1}) \cdot \Big(\int_{0}^{\zeta_1}  h_n  \ d1 \Big) \circ \f{F}^{-1}=0. 
		\end{align*}
		And thus $\tilde{\mathcal{Y}}$ is the left inverse of ${\mathcal{Y}}_{2,\Gamma_1}^s$. \\
		We continue to show that $\tilde{\mathcal{Y}}$ is also the right inverse.\\
		Writing $\tilde{\tilde{h}}_n \coloneqq  (-1)^{n+1}  \    \ \tilde{S}_{2\sigma(n)}$ one has
		\begin{align*}
			{\mathcal{Y}}_{2,\Gamma_1}^s \circ \tilde{\mathcal{Y}}(\tilde{\f{S}}) &= \tilde{\f{S}} +  \underbrace{\tilde{\f{J}} (\textup{Airy}(F_n^{-1}) \circ \f{F}) \tilde{\f{J}}^{T} \cdot \Big(\int_{0}^{\zeta_1}  \tilde{\tilde{h}}_n  \ d1 \Big)  }_{\eqqcolon (L2,3)}  \\
			&  \ \ \ + \underbrace{\widehat{\textup{Airy}}(F_k)
				\cdot \Big(\int_{0}^{\zeta_1}   (-1)^{k+1}  \    \  \big[(\tilde{\mathcal{Y}}(\tilde{\f{S}}))_{\sigma(k)l} \circ \f{F}\big]  \cdot \tilde{J}_{2l}  \ d1  \Big) }_{\eqqcolon(L2,4)}.
		\end{align*}
		Similar to \eqref{eq_chain_hessian_2} one can apply the chain rule and obtains
		\begin{align}
			\label{eq_L2_1}
			(L2,3)&= -\widehat{\textup{Airy}} (F_k) \cdot (J^{-1}_{nk} \circ \f{F}) \cdot \int_{0}^{\zeta_1}\tilde{\tilde{h}}_n  \ d1.
		\end{align}
		Below we write several times $\j^{-1}_{ij}$ instead of correctly $\j^{-1}_{ij} \circ \f{F}$ to shorten the expressions. In view of this remark  we can rearrange
		\begin{align*}
			\big[(\tilde{\mathcal{Y}}(\tilde{\f{S}}))_{\sigma(1)l} \circ \f{F}\big]  \cdot \tilde{J}_{2l} &= \big[(\tilde{\mathcal{Y}}(\tilde{\f{S}}))_{2l} \circ \f{F}\big]  \cdot \tilde{J}_{2l}    \\&= \j^{-1}_{2i} \cdot \tilde{S}_{ij} \cdot \underbrace{\j^{-1}_{lj} \cdot \j_{2l}}_{= \delta_{2j}}  + \tilde{J}_{2l} \cdot (\textup{Airy}(F_n^{-1}))_{2l} \cdot \int_{0}^{\zeta_1}\tilde{\tilde{h}}_n  \ d1 \\
			&= \j^{-1}_{2i} \cdot \tilde{S}_{i2} + \tilde{J}_{2l} \cdot (\textup{Airy}(F_n^{-1}))_{2l} \cdot \int_{0}^{\zeta_1}\tilde{\tilde{h}}_n  \ d1 \\
			&= \j^{-1}_{2i} \cdot \tilde{S}_{i2} - \tilde{J}_{21} \cdot \partial_1J^{-1}_{n2} \cdot \int_{0}^{\zeta_1}\tilde{\tilde{h}}_n  \ d1 + \tilde{J}_{22} \cdot \partial_1J^{-1}_{n1} \cdot \int_{0}^{\zeta_1}\tilde{\tilde{h}}_n  \ d1 \\
			&= \j^{-1}_{2i} \cdot \tilde{S}_{i2} + {J}_{i1} \cdot \underbrace{\partial_1J^{-1}_{ni}}_{=\partial_iJ^{-1}_{n1}} \cdot \int_{0}^{\zeta_1}\tilde{\tilde{h}}_n  \ d1  \\
			&= \j^{-1}_{2i} \cdot \tilde{S}_{i2} + \h_1\big[ J_{n1}^{-1}\circ \f{F}\big] \cdot \int_{0}^{\zeta_1}\tilde{\tilde{h}}_n  \ d1  \\
			&=\j^{-1}_{2i} \cdot \tilde{S}_{i2} - [ J_{n1}^{-1}\circ \f{F}\big] \cdot \tilde{\tilde{h}}_n + \h_1 \big[(J_{n1}^{-1}\circ \f{F}) \cdot \int_{0}^{\zeta_1}\tilde{\tilde{h}}_n  \ d1 \big].
		\end{align*}
		And with 
		\begin{align*}
			[ J_{n1}^{-1}\circ \f{F}\big] \cdot \tilde{\tilde{h}}_n = \tilde{J}^{-1}_{22} \cdot \tilde{S}_{22} + \tilde{J}_{21}^{-1} \cdot \tilde{S}_{21} =  \j^{-1}_{2i} \cdot \ts_{2i},
		\end{align*}
		we obtain
		\begin{align}
			\label{eq_L2_2}
			\int_{0}^{\zeta_1}\big[(\tilde{\mathcal{Y}}(\tilde{\f{S}}))_{\sigma(1)l} \circ \f{F}\big]  \cdot \tilde{J}_{2l}  \ d1&= (J_{n1}^{-1}\circ \f{F}) \cdot \int_{0}^{\zeta_1}\tilde{\tilde{h}}_n  \ d1.
		\end{align}
		With analogous steps  one has
		\begin{align}
			\big[(\tilde{\mathcal{Y}}(\tilde{\f{S}}))_{\sigma(2)l} \circ \f{F}\big]  \cdot \tilde{J}_{2l} &= \big[(\tilde{\mathcal{Y}}(\tilde{\f{S}}))_{1l} \circ \f{F}\big]  \cdot \tilde{J}_{2l}  \nonumber  \\&= \j^{-1}_{1i} \cdot \tilde{S}_{ij} \cdot \j^{-1}_{lj} \cdot \j_{2l} + \tilde{J}_{2l} \cdot (\textup{Airy}(F_n^{-1}))_{1l} \cdot \int_{0}^{\zeta_1}\tilde{\tilde{h}}_n  \ d1 \nonumber \\
			&= \j^{-1}_{1i} \cdot \tilde{S}_{i2} + \tilde{J}_{2l} \cdot (\textup{Airy}(F_n^{-1}))_{1l} \cdot \int_{0}^{\zeta_1}\tilde{\tilde{h}}_n  \ d1 \nonumber \\
			&= \j^{-1}_{1i} \cdot \tilde{S}_{i2} + \tilde{J}_{21} \cdot \partial_2J^{-1}_{n2} \cdot \int_{0}^{\zeta_1}\tilde{\tilde{h}}_n  \ d1 - \tilde{J}_{22} \cdot \partial_1J^{-1}_{n2} \cdot \int_{0}^{\zeta_1}\tilde{\tilde{h}}_n  \ d1 \nonumber \\
			&= \j^{-1}_{1i} \cdot \tilde{S}_{i2} - \h_1\big[ J_{n2}^{-1}\circ \f{F}\big] \cdot \int_{0}^{\zeta_1}\tilde{\tilde{h}}_n  \ d1 \nonumber \\
			&=\j^{-1}_{1i} \cdot \tilde{S}_{i2} + [ J_{n2}^{-1}\circ \f{F}\big] \cdot \tilde{\tilde{h}}_n - \h_1 \big[(J_{n2}^{-1}\circ \f{F}) \cdot \int_{0}^{\zeta_1}\tilde{\tilde{h}}_n  \ d1 \big]. \label{eq_L2_12}
		\end{align}
		Using the equality chain 
		\begin{align*}
			[ J_{n2}^{-1}\circ \f{F}\big] \cdot \tilde{\tilde{h}}_n &= - \j_{12}^{-1} \cdot \ts_{22}-\j_{11} \cdot \ts_{12} = -\j_{1i} \cdot \ts_{i2},
		\end{align*}
		and with \eqref{eq_L2_12} it is 
		\begin{align}
			\label{eq_L2_3}
			\int_{0}^{\zeta_1}	\big[(\tilde{\mathcal{Y}}(\tilde{\f{S}}))_{\sigma(2)l} \circ \f{F}\big]  \cdot \tilde{J}_{2l}   \ d1 &= -(J_{n2}^{-1}\circ \f{F}) \cdot \int_{0}^{\zeta_1}\tilde{\tilde{h}}_n  \ d1.
		\end{align}
		Looking at \eqref{eq_L2_1}, \eqref{eq_L2_2} and \eqref{eq_L2_3} one observes $(L2,3)= -(L2,4)$ and this finishes the proof.
	\end{proof}

	\subsection{Proof of Lemma \ref{Lemma_invertebility_Gamma_2_2} \textmd{(Invertibility of $\mathcal{Y}_2^s$)}}
	\label{Lemma_proof_appendix_3}
	\begin{proof}
		Using $\hat{\S} \coloneqq \f{S} \circ \f{F}$ we have 
		\begin{align}
			\label{eq:Lem17_1}
			\tilde{\mathcal{Y}}_2^{s} \circ {\mathcal{Y}}_2^{s}(\S)&= \S   - ({\mathcal{Y}}_{2,\Gamma_1}^{s})^{-1} \circ {\mathcal{Y}}_{2,A}^s(\S) - \tilde{\mathcal{Y}}_{2,A}^{s} \circ {\mathcal{Y}}_{2}^s(\S) \nonumber\\
			&= \S \nonumber \\ & \ \ \ - ({\mathcal{Y}}_{2,\Gamma_1}^{s})^{-1}  \Big(  \begin{bmatrix}
				\int_{0}^{\zeta_1} \h_1 \j_{1i} \cdot \j_{1l}(0, \cdot) \cdot \hat{S}_{il}(0,\cdot)  \ d1 & 0 \\
				0 &  \int_{0}^{\zeta_2} \h_1 \j_{2i} \cdot \j_{1l}(0, \cdot) \cdot \hat{S}_{il}(0,\cdot) \ d2
			\end{bmatrix}\Big) \nonumber\\
			& \ \ \ - ({\mathcal{Y}}_{2,\Gamma_1}^{s})^{-1} \Big(\begin{bmatrix}
				\int^{\zeta_1}_{0}  \big[\mathcal{Y}^s_{3}\big( T({\mathcal{Y}}_2^{s}[\S])\big) \big]_1   \ d1 & 0 \\
				0 & \int^{\zeta_2}_{0}  \big[\mathcal{Y}^s_{3}\big( T({\mathcal{Y}}_2^{s}[\S])\big) \big]_2  \ d2
			\end{bmatrix} \Big).
		\end{align}
		The $m$-th component of $T({\mathcal{Y}}_2^{s}[\S])$ is 
		\begin{align*}
			\big[T({\mathcal{Y}}_2^{s}[\S]) \big]_m&= \partial_k \tilde{J}^{-1}_{mi} \cdot \j^{-1}_{k1} \cdot (({\mathcal{Y}}_2^{s}[\S])_{i1}(0,\cdot) \circ \f{F}^{-1}).
		\end{align*}
		Note, we have 
		\begin{align*}
			(({\mathcal{Y}}_2^{s}[\S])_{i1}(0,\cdot) =  \j_{il}(0,\cdot) \cdot \hat{S}_{lj}(0,\cdot) \cdot \j_{1j}(0,\cdot).
		\end{align*}
		All the other terms vanish on $\hat{\Gamma}_1$.
		Thus, using \eqref{eq:gamma_2_E_equivalent_form}, we can write 
		\begin{align*}
			\Big[ \mathcal{Y}^s_{3}&\big(T({\mathcal{Y}}_2^{s}[\S]) \big) \Big]_m  \\
			&=  \h_1 \Big[  \j_{mn} \cdot \int_{0}^{\zeta_1}    \textup{det}(\f{J}) \ (\partial_k \tilde{J}^{-1}_{ni} \circ \f{F}) \cdot (\j^{-1}_{k1} \circ \f{F}) \cdot \j_{il}(0,\cdot) \cdot \hat{S}_{lj}(0,\cdot) \cdot \j_{1j}(0,\cdot)  \ d1  \Big].
		\end{align*}
		Since $\textup{det}(\f{J}) \cdot  (\j^{-1}_{k1} \circ \f{F}) = J_{k1}$ and due to 
		$$\h_1\big[\tilde{J}^{-1}_{ni} \circ \f{F} \big] = (\partial_k \tilde{J}^{-1}_{ni} \circ \f{F}) \cdot J_{k1} =  \textup{det}(\f{J}) \ (\partial_k \tilde{J}^{-1}_{ni} \circ \f{F}) \cdot (\j^{-1}_{k1} \circ \f{F}), $$ we can rearrange 
		\begin{align*}
			\Big[ \mathcal{Y}^s_{3}&\big(T({\mathcal{Y}}_2^{s}[\S]) \big) \Big]_m  \\
			&=  \h_1 \Big[  \j_{mn} \cdot \int_{0}^{\zeta_1}    \h_1\big[\tilde{J}^{-1}_{ni} \circ \f{F} \big]    \cdot \j_{il}(0,\cdot) \cdot \hat{S}_{lj}(0,\cdot) \cdot \j_{1j}(0,\cdot)  \ d1  \Big] \\
			&= \h_1 \Big[  \j_{mn} \cdot    (\tilde{J}^{-1}_{ni} \circ \f{F})     \cdot \j_{il}(0,\cdot) \cdot \hat{S}_{lj}(0,\cdot) \cdot \j_{1j}(0,\cdot) \   \Big] \\
			& \ \ \ - \h_1 \Big[  \j_{mn} \cdot    (\tilde{J}^{-1}_{ni} \circ \f{F}(0,\cdot))     \cdot \j_{il}(0,\cdot) \cdot \hat{S}_{lj}(0,\cdot) \cdot \j_{1j}(0,\cdot) \   \Big] \\
			&= \h_1 \Big[  \delta_{mi}    \cdot \j_{il}(0,\cdot) \cdot \hat{S}_{lj}(0,\cdot) \cdot \j_{1j}(0,\cdot) \   \Big] \\
			& \ \ \ - \h_1 \Big[  \j_{mn} \cdot    \delta_{nl} \cdot \hat{S}_{lj}(0,\cdot) \cdot \j_{1j}(0,\cdot) \   \Big]
			\\&=0 - \h_1 \Big[  \j_{mn}  \cdot \hat{S}_{nj}(0,\cdot) \cdot \j_{1j}(0,\cdot) \   \Big] \\
			&=- \h_1 \j_{mi} \cdot \j_{1l}(0,\cdot) \cdot  \hat{S}_{il}(0,\cdot) .
		\end{align*}
		To obtain the last line we changed the indices according to $n \rightarrow i, \ j \rightarrow l$.
		Hence, comparing the entries of the second and third matrix in \eqref{eq:Lem17_1}, we see that they add up to zero and we get indeed 
		$	\tilde{\mathcal{Y}}_2^{s} \circ {\mathcal{Y}}_2^{s}(\S)= \S $. \\
		On the other hand we obtain
		\begin{align*}
			{\mathcal{Y}}_2^{s} \circ \tilde{\mathcal{Y}}_2^{s}(\tS)&= \tS  -\begin{bmatrix}
				\int^{\zeta_1}_{0}  \big[\mathcal{Y}^s_{3}(T(\tS)) \big]_1   \ d1 & 0 \\
				0 & \int^{\zeta_2}_{0}  \big[\mathcal{Y}^s_{3}(T(\tS)) \big]_2  \ d2 
			\end{bmatrix} \\
		\end{align*} \vspace{-1cm} 
		$$
		- \underbrace{\begin{bmatrix}
				\int_{0}^{\zeta_1} \h_1 \j_{1i} \cdot \j_{1l}(0, \cdot) \cdot \big(\tilde{\mathcal{Y}}_2^{s}(\tS)\big)_{il} \circ \f{F}(0,\cdot)  \ d1 & 0 \\
				0 &  \int_{0}^{\zeta_2} \h_1 \j_{2i} \cdot \j_{1l}(0, \cdot) \cdot \big(\tilde{\mathcal{Y}}_2^{s}(\tS)\big)_{il} \circ \f{F}(0,\cdot) \ d2
		\end{bmatrix}}_{\eqqcolon (L3,1)}
		$$
		
		We concentrate now on the term $\big(\tilde{\mathcal{Y}}_2^{s}(\tS)\big)_{il} \circ \f{F}(0,\cdot)$.
		Due to Lemma \ref{Lemma_compatibility_BC_strong_sym} we have that \\ $\hat{\S}_H \circ \f{F}^{-1} \coloneqq \S_H \coloneqq \tilde{\mathcal{Y}}_{2,A}^{s}(\tS) \in \t{H}_{\Gamma_1}(\o,\d,\s)$. Thus this implies on $\hat{\Gamma}_1$
		\begin{align*}
			-\j_{1l}(0,\cdot) \cdot (\hat{S}_H)_{il}(0,\cdot)  =  \big(\J(0,\cdot) \cdot \hat{\S}_H^i(0,\cdot) \big)^T \cdot \hat{\f{n}}_1   =0,
		\end{align*}
		where $\hat{\f{n}}_1$ is the outer unit normal to $\hat{\Gamma}_1$.\\
		Consequently, we have $\int_{0}^{\zeta_1} \h_1 \j_{mi} \cdot \j_{1l}(0, \cdot) \cdot \big(\tilde{\mathcal{Y}}_{2,A}^{s}(\tS)\big)_{il} \circ \f{F}(0,\cdot)  \ d1=0$ for $m=1,2$ and
		\begin{equation*}
			\big(({\mathcal{Y}}_{2,\Gamma_1}^{s})^{-1}(\tS)\big)_{il} \circ \f{F}(0,\cdot) = \big(\j_{in}^{-1} \circ \f{F}(0,\cdot)\big) \cdot \ts_{nj}(0,\cdot) \cdot \big(\j_{lj}^{-1} \circ \f{F}(0, \cdot) \big).
		\end{equation*}
		Thus the integrand of the $m$-th diagonal entry of the third matrix $(L3,1)$ from above  is
		\begin{align}
			\h_1 \j_{mi} \cdot &\j_{1l}(0, \cdot) \cdot \big(\tilde{\mathcal{Y}}_2^{s}(\tS)\big)_{il} \circ \f{F}(0,\cdot) \nonumber\\ &= \h_1 \j_{mi} \cdot \j_{1l}(0, \cdot) \cdot \big(\j_{in}^{-1} \circ \f{F}(0,\cdot)\big) \cdot \ts_{nj}(0,\cdot) \cdot \big(\j_{lj}^{-1} \circ \f{F}(0, \cdot) \big) \nonumber \\&=
			\delta_{1j} \cdot \h_1 \j_{mi}  \cdot \big(\j_{in}^{-1} \circ \f{F}(0,\cdot)\big) \cdot \ts_{nj}(0,\cdot)   \nonumber\\ \label{eq_Lem17_1}
			&= \h_1 \j_{mi}  \cdot \big(\j_{in}^{-1} \circ \f{F}(0,\cdot)\big) \cdot \ts_{n1}(0,\cdot)  . 
		\end{align}
		Further, in view of \eqref{eq:gamma_2_E_equivalent_form} we can calculate
		\begin{align}
			\big[\mathcal{Y}^s_{3}(T(\tS)) \big]_m &= \h_1 \Big[  \j_{mn} \cdot \int_{0}^{\zeta_1}    \underbrace{  (\partial_k \tilde{J}^{-1}_{ni} \circ \f{F}) \cdot \textup{det}(\f{J}) \cdot (\j^{-1}_{k1} \circ \f{F})}_{\h_1\big[\j_{ni}^{-1} \circ \f{F} \big]} \cdot \tilde{S}_{i1}(0,\cdot)  \ d1  \Big] \nonumber\\
			&= \h_1 \Big[  \underbrace{\j_{mn} \cdot       (\j_{ni}^{-1} \circ \f{F})}_{=\delta_{mi}} \cdot \tilde{S}_{i1}(0,\cdot)  \Big] \nonumber \\
			& \ \ \ - \h_1 \Big[  \j_{mn} \cdot       (\j_{ni}^{-1} \circ \f{F}(0,\cdot)) \cdot \tilde{S}_{i1}(0,\cdot)  \Big] \nonumber  \\
			&= 0- \h_1\j_{mn} \cdot       (\j_{ni}^{-1} \circ \f{F}(0,\cdot)) \cdot \tilde{S}_{i1}(0,\cdot) \nonumber \\ \label{eq_Lem17_2}
			& =- \h_1\j_{mi} \cdot       (\j_{in}^{-1} \circ \f{F}(0,\cdot)) \cdot \tilde{S}_{n1}(0,\cdot) .
		\end{align}
		To get from the second last to the  last line we switched the indices $n \leftrightarrow i$.
		A comparison of  \eqref{eq_Lem17_1} and \eqref{eq_Lem17_2} yields the wanted result		
		$  {\mathcal{Y}}_2^{s} \circ \tilde{\mathcal{Y}}_2^{s}(\tS)= \tS $.		
		This finishes the proof.
	\end{proof}
	
	\newpage
	\newpage
	\nocite{*}
	\bibliographystyle{siam}
	\bibliography{Literatur}
\end{document}